\documentclass[12pt]{amsart}
\usepackage{latexsym}
\usepackage{amssymb}
\usepackage{amscd}
\renewcommand\a{\alpha}
\renewcommand\b{\beta}
\newcommand\g{\gamma}
\renewcommand\d{\delta}
\newcommand\la{\lambda}

\newcommand\e{\eta}
\renewcommand\th{\theta}

\newcommand\io{\iota}

\newcommand\n{\nu}
\newcommand\s{\sigma}

\newcommand\f{\phi}
\newcommand\vf{\varphi}
\newcommand\p{\psi}

\renewcommand\r{\rho}

\newcommand\w{\omega}

\newcommand\D{\Delta}

\newcommand\vG{\varGamma}
\newcommand\ve{\varepsilon}

\newcommand\Fq{{\mathbf F}_q}

\newcommand\Ql{\bar{\mathbf Q}_l}
\newcommand\BA{\mathbf A}
\newcommand\BP{\mathbf P}

\newcommand\BC{\mathbf C}

\newcommand\BZ{\mathbf Z}

\newcommand\BH{\mathbf H}

\newcommand\Bm{\mathbf m}

\newcommand\Bk{\mathbf k}

\newcommand\Bla{\boldsymbol\lambda}

\newcommand\Bmu{\boldsymbol\mu}

\newcommand\CA{\mathcal{A}}
\newcommand\CB{\mathcal{B}}
\newcommand\ZC{\mathcal{C}}
\newcommand\CH{\mathcal{H}}

\newcommand\CE{\mathcal{E}}
\newcommand\CL{\mathcal{L}}
\newcommand\DD{\mathcal{D}}

\newcommand\CM{\mathcal{M}}

\newcommand\CO{\mathcal{O}}

\newcommand\CP{\mathcal{P}}

\newcommand\CU{\mathcal{U}}

\newcommand\CF{\mathcal{F}}
\newcommand\CG{\mathcal{G}}

\newcommand\CZ{ \mathcal{Z}}
\newcommand\CX{ \mathcal{X}}
\newcommand\CY{ \mathcal{Y}}
\newcommand\CW{ \mathcal{W}}

\newcommand\Fb{\mathfrak b}
\newcommand\Fu{\mathfrak u}

\newcommand\Fg{\mathfrak g}

\newcommand\Fh{\mathfrak h}

\newcommand\Fl{\mathfrak l}
\newcommand\Ft{\mathfrak t}

\newcommand\iv{^{-1}}
\newcommand\wh{\widehat}
\newcommand\wt{\widetilde}
\newcommand\wg{^{\wedge}}

\newcommand\ol{\overline}

\newcommand\hra{\hookrightarrow}
\newcommand\lra{\leftrightarrow}

\newcommand\IC{\operatorname{IC}}
\newcommand\Ker{\operatorname{Ker}}
\newcommand\Hom{\operatorname{Hom}}
\newcommand\End{\operatorname{End}}

\newcommand\Ind{\operatorname{Ind}}

\newcommand\Res{\operatorname{Res}}

\newcommand\supp{\operatorname{supp}\,}
\newcommand\Lie{\operatorname{Lie}}

\newcommand\ch{\operatorname{ch}}
\newcommand\rk{\operatorname{rank}\,}

\newcommand\Ad{\operatorname{Ad}}

\newcommand\reg{_{\operatorname{reg}}}
\newcommand\rg{\operatorname{reg}}
\newcommand\unip{\operatorname{uni}}
\newcommand\uni{_{\operatorname{uni}}}
\newcommand\nil{_{\operatorname{nil}}}
\newcommand\id{\operatorname{id}}

\newcommand\lp{\operatorname{\!\langle\!}}
\newcommand\rp{\operatorname{\!\rangle\!}}
\renewcommand\Im{\operatorname{Im}}

\newcommand\nat{^{\natural}}

\newcommand\dw{\dot w}

\newcommand{\isom}{\,\raise2pt\hbox{$\underrightarrow{\sim}$}\,}
\numberwithin{equation}{section}

\newtheorem{thm}{Theorem}[section]
\newtheorem{lem}[thm]{Lemma}
\newtheorem{cor}[thm]{Corollary}
\newtheorem{prop}[thm]{Proposition}

\def \para#1{\par\medskip\textbf{#1}
              \addtocounter{thm}{1}}

\def \remark#1{\par\medskip\noindent
                \textbf{Remark #1}
                \addtocounter{thm}{1}}

\def \remarks#1{\par\medskip\noindent
                \textbf{Remarks #1}
                \addtocounter{thm}{1}}

\begin{document}
\setlength{\baselineskip}{4.9mm}
\setlength{\abovedisplayskip}{4.5mm}
\setlength{\belowdisplayskip}{4.5mm}
%%%
%%%
\renewcommand{\theenumi}{\roman{enumi}}
\renewcommand{\labelenumi}{(\theenumi)}
\renewcommand{\thefootnote}{\fnsymbol{footnote}}
%%%
\renewcommand{\thefootnote}{\fnsymbol{footnote}}
%%%
\allowdisplaybreaks[2]
%\NoBlackBoxes
\parindent=20pt
%%%%%%%%%%%%%%%%%%%%
%%%%%%%%%%%%%%%%%%%%%%%%%%%%%%%%%%%
\medskip
\begin{center}
 {\bf Exotic symmetric space over a finite field, I} 
\\
\vspace{1cm}
Toshiaki Shoji\footnote{
partly supported by JSPS Grant-in Aid for Scientific Research (A) 20244001.}
 and Karine Sorlin\footnote{ supported by 
ANR JCJC REPRED, ANR-09-JCJC-0102-01.}
\\
\vspace{0.7cm}
\title{}
{\sl Dedicated to the memory of T. A. Springer}
\end{center}

\begin{abstract}
Let $V$ be a $2n$-dimensional vector space over 
an algebraically closed field $\Bk$ with $\ch \Bk \ne 2$.  Let $G = GL(V)$
and $H = Sp_{2n}$ be the symplectic group obtained as 
$H = G^{\th}$ for an involution $\th$ on $G$. We also denote by 
$\th$ the induced involution on $\Fg = \Lie G$.
Consider the variety $G/H \times V$ on which $H$ acts naturally.
Let $\Fg^{-\th}\nil$ be the set of nilpotent elements 
in the $-1$ eigenspace of $\th$ in $\Fg$. 
The role of the unipotent variety for $G$ in our setup is played by 
$\Fg^{-\th}\nil \times V$, which coincides with Kato's exotic nilpotent cone.
Kato established, in the case where $\Bk = \BC$, 
the Springer correspondence between the set of irreducible 
representations of the Weyl group of type $C_n$ and the set of $H$-orbits
in $\Fg^{-\th}\nil \times V$ by applying Ginzburg theory for affine 
Hecke algebras.  In this paper we develop a theory of character sheaves 
on $G/H \times V$, and give an alternate proof for Kato's result on 
the Springer correspondence based on the theory of character sheaves.

\end{abstract}

\maketitle
\markboth{SHOJI AND SORLIN}{EXOTIC SYMMETRIC SPACE, I}
\pagestyle{myheadings}

\begin{center}
{\sc Introduction}
\end{center}
\par\bigskip
Let $G' = GL_n$ acting on the $n$-dimensional vector space $V'$ over 
an algebraically closed filed $\Bk$, and $\Fg' = \Lie G'$.
Let $G'\uni$ (resp. $\Fg'\nil$) be the unipotent variety of $G'$
(resp. the nilpotent cone of $\Fg'$).  We consider the action of
$G'$ on the variety $\Fg'\nil \times V'$, where $G'$ acts on $\Fg'\nil$ by 
the adjoint action, and on $V'$ by the natural action.  
By Achar-Henderson [AH] and 
Travkin [T], $\Fg'\nil \times V'$ has a finitely many $G'$-orbits 
parametrized by double partitions of $n$.  Following [AH], 
we call $\Fg'\nil \times V'$ the enhanced nilpotent cone.  
In [AH], they studied the intersection cohomology of the closure 
of such orbits, 
and showed that associated Poincar\'e polynomials give Kostka polynomials 
labelled by double partitions, introduced in [S2], 
which is an analogue of the classical result 
by Lusztig [L1] relating nilpotent orbits in $\Fg\nil$ 
and Kostka polynomials.  Passing to the group case, we consider 
the action of $G'$ on the variety $G' \times V'$, where $G'$ acts 
on $G'$ by conjugation, and on $V'$ by the natural action.  
Finkelberg-Ginzburg-Travkin [FGT] constructed a family of $G'$-equivariant
simple perverse sheaves on $G' \times V'$, and developed an analogue
of the theory of character sheaves on $G'$, where 
$G'\uni \times V' \simeq \Fg'\nil \times V'$ plays a role of the unipotent 
variety of $G'$.  They conjecture that the characteristic functions of 
such character sheaves on $G' \times V'$ provide a basis of the space of 
$G'(\Fq)$-invariant functions on $(G' \times V')(\Fq)$.
\par
Assume that $\Bk$ is an algebraic closure of a finite field 
$\Fq$ with $\ch \Bk \ne 2$, and let $V$ be a $2n$-dimensional vector
space over $\Bk$.
Let $H = Sp_{2n}$ be the symplectic group obtained as the fixed point 
subgorup $G^{\th}$ for an involutive automorphism $\th$ on 
$G = GL(V)$, and consider 
the symmetric space $G/H$.  In [BKS], Bannai-Kawanaka-Song
studied the characters of the Hecke algebra $\CH = \CH(G(\Fq), H(\Fq))$ 
associated to
the pair $H(\Fq) \subset G(\Fq)$, and showed that the character 
table of $\CH$ is basically obtained from 
the character table of $GL_n(\Fq)$  
by replacing $q$ by $q^2$ in an appropriate sense. 
On the other hand, in [H1] Henderson tried to reconstruct 
the result of [BKS] in terms of the geometry of the symmetric 
space $G/H$.  
Let $\Fg^{-\th}\nil = \{ g \in \Fg\nil \mid \th(g) = -g\}$ 
for the involution $\th$ induced on $\Fg = \Lie G$.
Then $H$ acts on $\Fg^{-\th}\nil$, and $H$-orbits are labelled by 
partitions of $n$. He showed, in particular,  that Poincar\'e polynomials 
associated to the intersection cohomology of 
the closure of those orbits provide Kostka polynomials, replacing
the variable $q$ by $q^2$, which is a geometric counter part of the 
result of [BKS].
\par
In this paper, we consider the variety $G/H \times V$ as a generalization 
of above two cases. $H$ acts on $G/H \times V$ 
as a left multiplication on $G/H$, and as the natural action on $V$.   
In this setup, the role of the unipotent variety for $G$ is played by 
the variety $\Fg^{-\th}\nil \times V$, which is nothing but the 
exotic nilpotent cone introduced by Kato [Ka1].  So we shall call
$G/H \times V$ the exotic symmetric space.  $H$ acts on 
$\Fg^{-\th}\nil \times V$. Kato showed that the number of $H$-orbits
is finite and they are parametrized by double partitions of $n$ 
(a reformulation by Achar-Henderson [AH]). It is expected 
an interesting relationship between the intersection cohomology 
of the closure of those orbits and Kostka polynomials labelled by 
double partitions.   
Our aim is to construct a theory of character sheaves on 
$G/H \times V$ as an analogue of the theroy for $G'$ and 
$G' \times V'$.  
In fact, in [HT] Henderson-Trapa propose a construction of
character sheaves on $G/H \times V$, as a natural generalization of
mirabolic character sheaves due to [FGT], 
``exotic character sheaves'' in their terminology.  
In their framework, the character sheaves constructed in this paper
just cover the principal series part.  However, we expect that 
any exotic character sheaf can be obtained by our construction. 
\par
The main result in this paper is 
the Springer correspondence between the set of 
irreducible representations 
of the Weyl group of type $C_n$ and the set of $H$-orbits of 
$\Fg^{-\th}\nil \times V$ through the intersection cohomology of the
closure of $H$-orbits.  In fact, the Springer correspondence for 
exotic nilcone was 
first established by [Ka1], by using Ginzburg theory of 
affine Hecke algebras.
In [Ka2], he determined the correspondence explicitly 
by computing Joseph polynomials associated to $H$-orbits.
So our result gives an alternate approach to Kato's result 
based on the theory of character sheaves, which is quite similar 
to the original proof of the Springer correspondence due to 
Borho-MacPherson [BM].
We prove the restriction theorem for Springer representations, 
which is an analogue of Lusztig's restriction theorem [L2] 
with respect to the generalized Springer correspondence, and 
we determine the correspondence explicitly by using this restriction
theorem.  
\par The authors are grateful to S. Kato for valuable discussions 
about his work.  They also thank A. Henderson for some useful comments. 
%%%
%%%
\newpage
\par\bigskip\bigskip
{\bf Contents}
\par\medskip
1. \ Symmetric space $GL_{2n}/Sp_{2n}$
\par
2. \ $H$-orbits on $G^{\io\th}\reg \times V$
\par
3. \ Intersection cohomology on $G^{\io\th}\reg \times V$
\par
4. \ Intersection cohomology on $G^{\io\th} \times V$
\par
5. \ Springer correspondence 
\par
6. \ Restriction theorem 
\par
7. \ Determination of Springer correspondence

\par\bigskip 
\section{Symmetric space $GL_{2n}/Sp_{2n}$}

\para{1.1.}
Let $\Bk$ be an algebraic closure of a finite field 
$\Fq$ with char $\Bk \ne 2$.
Let $V$ be a $2n$ dimensional vector space over $\Bk$, 
with basis $\{ e_1, \dots e_n, f_1, \dots f_n\}$.
Let $G =  GL_{2n}$ and $\Fg = \Lie G = \Fg\Fl_{2n}$.  
Consider an involutive automorphism  $\th : G \to G$ given by
\begin{equation*}
\th(g) = J\iv({}^tg\iv)J \quad\text{ with }
J = \begin{pmatrix}
          0 & 1_n \\
            -1_n & 0
         \end{pmatrix},
 \end{equation*}
where $1_n$ is the identity matrix of degree $n$,
and put $H = G^{\th}$.  Then $H$ is the symplectic group $Sp_{2n}$
with respect to the symplectic form 
$\lp u, v\rp ={}^tuJv$ for $u,v \in V$ 
under the identification $V \simeq \Bk^{2n}$ via the basis 
$\{ e_1, \dots, e_n, f_1, \dots f_n\}$, which  gives rise to 
a symplectic basis.
We denote by the same symbol $\th$ the involution induced on $\Fg$.   
Then 
$\th(x) = -J\iv({}^tx)J$ for $x \in \Fg$.
We have a decomposition 
$\Fg = \Fg^{\th} \oplus \Fg^{-\th}$, where 
\begin{equation*}
\Fg^{\pm \th} = \{ x \in \Fg \mid \th(x) = \pm x\}.
\end{equation*}
Let $x^*$ be the adjoint of $x \in \Fg$ with respect to
the form $\lp\ ,\ \rp$.  Then we have $x^* = J\iv({}^tx)J$, and
so $\Fg^{\pm \th} = \{ x \in \Fg \mid x^* = \mp x\}$.
In particular, $\Fg^{-\th}$ coincides with the set of self-adjoint 
matrices in $\Fg\Fl_{2n}$.

\para{1.2.}
Let $\io : G \to G$ be the anti-automorphism $g \mapsto g\iv$.
We consider the set $G^{\io\th} = \{ g \in G \mid \th(g) = g\iv\}$.
Then as in the Lie algebra case, $G^{\io\th}$ coincides with the set
of non-degenerate self-adjoint matrices.  In particular, it is 
connected.
Let $B$ be the subgroup of $G$ consisting of the elements of the 
form
\begin{equation*}
\begin{pmatrix}
b_1 & c \\
0 & b_2
\end{pmatrix},
\end{equation*}
where $b_1, b_2, c$ are square matrices of degree $n$, with
$b_1$ upper triangular and $b_2$ lower triangular.
Let $T$ be the set of all diagonal matrices in $G$.  
Then $B$ is a Borel subgroup of $G$ containing $T$, and $B$ and 
$T$ are both $\th$-stable. 
We have 
\begin{equation*}
\tag{1.2.1}
T^{\th} = \{ \begin{pmatrix}
                a & 0 \\
                0 & a\iv 
             \end{pmatrix} \mid a \in D_n \},
\qquad 
T^{\io\th} = \{ \begin{pmatrix}
                   a & 0 \\
                   0 & a 
                 \end{pmatrix}
                     \mid a \in D_n \},
\end{equation*}
where $D_n$ is the set of diagonal matrices in $GL_n$.
Moreover we have
\begin{align*}
\tag{1.2.2}
B^{\th} &= \{ g = \begin{pmatrix}
                b & c \\
                0 & {}^tb\iv
              \end{pmatrix} \in B \mid  {}^tc = b\iv c\ {}^tb\},
\\
B^{\io\th} &= \{ g = \begin{pmatrix}
                b & c \\
                0 & {}^tb
              \end{pmatrix} \in B \mid  {}^tc = -c\}.
\end{align*}
\par
We note that 
\begin{equation*}
\tag{1.2.3}
G^{\io\th} = \{ g\th(g)\iv \mid g \in G \}.
\end{equation*}
In fact, it is known by a general theory ([R, 2.2], see also [Gi, 3.3]) that 
the right hand side of (1.2.3) coincides with the connected 
component of $G^{\io\th}$.  Since $G^{\io\th}$ is connected, 
(1.2.3) holds. It is also checked directly as follows; 
take $x \in G^{\io\th}$.  Then $x$ is self-adjoint, and
so there exists an isotropic flag 
$(V_1 \subset V_2 \subset \cdots \subset V_n)$ in $V$ stable by $x$.
Since $G^{\io\th}$ is $H$-stable, by replacing $x$ by its $H$-conjugate, 
we may assume that $x \in B^{\io\th}$.  We write 
$x = \begin{pmatrix}
            b & c \\
            0 & {}^tb
     \end{pmatrix}$
as in (1.2.2).  
If we put $y = \begin{pmatrix}
                     b & 0 \\
                     0 & 1_n
                 \end{pmatrix}$, 
then $\th(y) = \begin{pmatrix}
                   1_n & 0 \\
                   0 & {}^tb\iv
                \end{pmatrix}$, and so 
$y\iv x \th(y) =  \begin{pmatrix}
                   1_n & c'\\
                   0   & 1_n
                  \end{pmatrix} \in B^{\io\th}$.
Since $c'$ is a skew-symmetric matrix, one can write
$c' = a - {}^ta$ for a square matrix $a$ of degree $n$. 
Then we have
\begin{equation*}
\begin{pmatrix}
1_n & c' \\
0  &  1_n 
\end{pmatrix}
= 
\begin{pmatrix}
1_n & a \\
0   & 1_n
\end{pmatrix}
\begin{pmatrix}
1_n & -{}^ta \\
0  &  1_n
\end{pmatrix}
= z\th(z)\iv
\end{equation*}
with $z = \begin{pmatrix}
             1_n &  a \\
             0   &  1_n
           \end{pmatrix}$.
This implies that $x =g\th(g)\iv$ for some $g \in G$.
The opposite inclusion in (1.2.3) is clear. 
\par
Note that the above argument shows, in particular,  that
\begin{equation*}
\tag{1.2.4}
B^{\io\th} = \{ b\th(b)\iv \mid b \in B\}.
\end{equation*}

\remark{1.3.}
The following properties hold.
\par
(i) If $B'$ is a $\th$-stable Borel subgroup 
of $G$, then $B$ and $B'$ are conjugate under $H$.
\par
(ii) If $B'$ is a $\th$-stable Borel subgroup containing a $\th$-stable
maximal torus $T'$, then the pairs $(B',T')$ and $(B, T)$ are conjugate 
under $H$.
\par
(iii)  A $\th$-stable torus $S$ of $G$ is called 
$\th$-anisotropic if $\th(s) = s\iv$ for any $s \in S$.
Then $T^{\io\th}$ is a maximal $\th$-anisotropic torus, and 
any maximal torus of $G$ containing a maximal 
$\th$-anisotropic torus is conjugate to 
$T$ under $H$. 
\par\medskip
In fact, for (i), $B'$ is written as $B' = xBx\iv$ for some $x \in G$.
Then we have $x\iv\th(x) \in B^{\io\th}$.
By (1.2.4), there exists $b \in B$ such that 
$b\th(b)\iv = x\iv\th(x)$, and $g = xb \in H$. Then 
$gBg\iv = B'$. 
For (ii), assume that $xBx\iv = B', xTx\iv = T'$.  
Then $x\iv\th(x) \in B^{\io\th} \cap N_G(T) = T^{\io\th}$, 
and there exists $t \in T$ such that $t\th(t)\iv = x\iv\th(x)$. 
Then $g = xt \in H$, and we have $g(B,T)g\iv = (B',T')$. 
For (iii), it is clear that $T^{\io\th}$ is maximal $\th$-anisotropic.  
By a general theory ([V], see also [R, 2.7]) that maximal $\th$-anisotropic 
tori are conjugate under $H$.  Since 
$Z_G(T^{\io\th}) = Z_H(T^{\io\th})\cdot T^{\io\th}$, (iii) follows. 
Note that (iii) implies that any maximal torus containing 
a maximal $\th$-anisotropic torus is $\th$-stable.    
\para{1.4.}
By (1.2.3), the symmetric space $G/H$ can be identified 
with $G^{\io\th}$ so that the natural map $\pi: G \to G/H$
is given by $\pi: G \to G^{\io\th}, g \mapsto g\th(g)\iv$ 
(cf. [R. Lemma 2.4]).
Under this identification, the left multiplication of $H$ on 
$G/H$ turns out to be the conjugation action of $H$ on $G^{\io\th}$.
Let $A$ be a closed subgroup of $G$ isomorphic to $GL_n$ defined by 
\begin{equation*}
A = \{ \begin{pmatrix}
           x   &   0  \\
           0   &    1_n
        \end{pmatrix} \in G  \mid x \in GL_n \}.
\end{equation*}
Then for $a = \begin{pmatrix}
                  x & 0 \\
                  0 & 1_n
               \end{pmatrix} \in A$, we have
$a\th(a)\iv = \begin{pmatrix}
                  x & 0 \\
                  0   & {}^tx
           \end{pmatrix} \in G^{\io\th}$.
\par
We define a subgroup $L$ of $G$ by 
\begin{equation*}
L = \{ \begin{pmatrix}
           x & 0 \\
           0 & y
        \end{pmatrix} \in G 
      \mid x,y \in GL_n \} \simeq GL_n \times GL_n.
\end{equation*}
Then $L$ is $\th$-stable, and $L^{\th} = \{ a\th(a) \mid a \in A \}$,
$L^{\io\th} = \{ a\th(a)\iv \mid a \in A\}$. 
The following lemma was proved by Klyachko [Kl] (cf. [BKS, Lemma 2.3.4]).
%%%
\begin{lem}  %%%  Lemma 1.5
\begin{enumerate}
\item  For $a, a' \in A$,  $a\th(a)\iv$ is $H$-conjugate to 
$a'\th(a')\iv$ if and only if $a$ and $a'$ are conjugate in $A$.
\item
The map $a \mapsto a\th(a)\iv$ induces a bijection between 
the set of conjugacy classes in $A$ and the set of $H$-conjugacy classes
in $G^{\io\th}$.
\end{enumerate}
\end{lem}

\para{1.6.} 
We note that for any $g \in G^{\io\th}$, $Z_H(g)$ is a connected 
subgroup of $H$. In fact, this result is essentially contained 
in the proof of Proposition 2.3.6 in [BKS], in view of Lemma 1.5 (ii).

\para{1.7.}
Let $g = su = us$ be the Jordan decomposition of $g \in G$, 
where $u$ is unipotent and $s$ is semisimple. If $g \in G^{\io\th}$, 
then $s, u \in G^{\io\th}$, and the Jordan decomposition makes sense
in $G^{\io\th}$.  We denote by $G^{\io\th}\uni$ the set of unipotent 
elements in $G^{\io\th}$.  Similarly, we have the Jordan decomposition 
of $\Fg^{-\th}$, and denote by $\Fg^{-\th}\nil$ the 
set of nilpotent elements in $\Fg^{-\th}$.
\par
We define a map $\log: G^{\io\th} \to \Fg^{-\th}$ by the composite 
of the maps  
\begin{equation*}
\begin{CD}
\log: G^{\io\th} @>i>> G  @>j>> \Fg = \Fg^{\th} \oplus \Fg^{-\th}
  @>p_2>> \Fg^{-\th},
\end{CD}
\end{equation*}
where $i$ is the inclusion map, and $j$ is the map defined by 
$j(g) = g-1$, and $p_2$ is the projection onto the second factor.
Then $\log$ is an $H$-equivariant morphism from $G^{\io\th}$ to 
$\Fg^{-\th}$ and $\log(1) = 0$, and 
$d\log_{e}: \Fg^{-\th} \to \Fg^{-\th}$ induces the identity map on 
$\Fg^{-\th}$.  
By Bradsley-Richardson [BR, Proposition 10.1], we see that 
the restriction of $\log$ on $G^{\io\th}\uni$ gives rise to 
an $H$-equivariant isomorphism $G^{\io\th}\uni \simeq \Fg^{-\th}\nil$. 
\par
Under the correspondence in Lemma 1.5, the set of unipotent classes 
in $A$ is mapped to the set of unipotent classes in $G^{\io\th}$.
Since the set of unipotent classes in $A$ is parametrized by 
the set $\CP_n$ of partitions of $n$, we see that 
\begin{equation*}
\tag{1.7.1}
G^{\io\th}\uni/\sim_H \ \simeq \ \Fg^{-\th}\nil/\sim_H
 \ \simeq \CP_n,
\end{equation*}
where $X/\sim_H$ denotes the set of $H$-orbits of the
$G$-variety $X$.

\par
We consider the varieties $G^{\io\th}\uni \times V$ and 
$\Fg^{-\th}\nil \times V$, with diagonal actions of $H$.
Note that $\Fg^{-\th}\nil \times V$ is nothing but the exotic 
nilpotent cone introduced by Kato [Ka1].
Now $\log \times \id$ induces an $H$-equivariant isomorphism 
between $G^{\io\th}\uni \times V$ and $\Fg^{-\th}\nil \times V$.
By [Ka1], we know that $(\Fg^{-\th}\nil \times V)/\sim_H$ is parametrized
by the set $\CP_{n,2}$ of double partitions of $n$. 
Hence we have
\begin{equation*}
\tag{1.7.2}
(G^{\io\th}\uni \times V)/\!\sim_H\ 
    \simeq \ (\Fg^{-\th}\nil \times V)/\!\sim_H\ \simeq \ \CP_{n,2}.
\end{equation*}
We denote by $\CO_{\Bla}$ the $H$-orbit of $G^{\io\th}\uni \times V$ 
corresponding to $\Bla \in \CP_{n,2}$. (Here we follow the parametrization
given in [AH, Theorem 6.1] in connection with the parametrization of 
the orbits in the enhanced nilpotent cone.)  

\par
The closure relations of orbits in $G^{\io\th}\uni \times V$ are
described as follows.
For $\Bla = (\mu, \nu) \in \CP_{n,2}$ with $\mu = (\mu_i), \nu = (\nu_i)$, 
we define 
a composition $c(\Bla)$ of $n$ by 
$c(\Bla) = (\mu_1, \nu_1, \mu_2, \nu_2, \dots)$. 
We define a partial order on $\CP_{n,2}$ by $\Bmu \le \Bla$ 
if and only if $c(\Bmu) \le c(\Bla)$, where the latter is the dominance 
order of compositions of $n$.  Then we have by [AH, Theorem 6.3],
\begin{equation*}
\tag{1.7.3}
\CO_{\Bmu} \subset \ol\CO_{\Bla} \text{ \ if and only if \ } \Bmu \le \Bla.
\end{equation*}

\para{1.8.}
Let $T^{\io\th}\reg$ be the set of regular semisimple 
elements in $T^{\io\th}$, i.e., the set of semisimple elements 
in $T^{\io\th}$ such that all the eigenspaces on $V$ have 
dimension 2.
We define $G^{\io\th}\reg = \bigcup_{g \in H}gT^{\io\th}\reg g\iv$, 
the set of regular semisimple elements in $G^{\io\th}$.  
Then $G^{\io\th}\reg$ is open dense in $G^{\io\th}$.
Take  $t \in T^{\io\th}\reg$.  Then 
$Z_H(t) = Z_H(T^{\io\th}) \simeq SL_2 \times \cdots \times SL_2$
($n$ copies).
For the Borel subgroup $B^{\th}$ of $H$, put
\begin{equation*}
\wt G^{\io\th} = \{ (x, gB^{\th})  \in G^{\io\th} \times H/B^{\th}
       \mid g\iv xg \in B^{\io\th} \}
\end{equation*}
and define a map $\pi: \wt G^{\io\th} \to G^{\io\th}$ by 
$\pi(x, gB^{\th}) = x$.
Then $\wt G^{\io\th} \simeq H \times^{B^{\th}}B^{\io\th}$ is a 
locally trivial fibration over $H/B^{\th}$ with fibre isomorphic
to $B^{\io\th}$ (the associated bundle of the principal 
$B^{\th}$-bundle $H \to H/B^{\th}$).  
Since $B^{\io\th}$ is smooth by (1.2.2), 
$\wt G^{\io\th}$ is smooth and irreducible. 
Moreover, $\pi$ is proper.
\par
We consider the pull-back $\pi\iv(G^{\io\th}\reg)$ of $G^{\io\th}\reg$
under the map $\pi$, and let 
$\psi : \pi\iv(G^{\io\th}\reg) \to G^{\io\th}\reg$ be the restriction 
of $\pi$ on $\pi\iv(G^{\io\th}\reg)$.
Then we have 
\begin{equation*}
\pi\iv(G^{\io\th}\reg) \simeq H \times^{B^{\th}}B^{\io\th}\reg 
\simeq H\times ^{(B^{\th} \cap Z_H(T^{\io\th}))}T^{\io\th}\reg,
\end{equation*}
where $B^{\io\th}\reg = B^{\io\th} \cap G^{\io\th}\reg$.
Note that $B^{\th} \cap Z_H(T^{\io\th})$ is a Borel subgroup 
of $Z_H(T^{\io\th}) \simeq SL_2\times\cdots\times SL_2$, hence 
is of the form $B_2 \times \cdots \times B_2$ with a Borel subgroup
$B_2$ of $SL_2$. 
Now the map $\psi$ factors through 
$H\times^{Z_H(T^{\io\th})}T^{\io\th}$ as follows;
\begin{equation*}
\tag{1.8.1}
\begin{CD}
\psi: \pi\iv(G^{\io\th}\reg) @>\xi >> 
      H\times^{Z_H(T^{\io\th})}T^{\io\th}\reg
                      @>\e >> G^{\io\th}\reg, 
\end{CD}
\end{equation*}
where $\xi$ is proper, and is a locally trivial fibration with fibre 
isomorphic to 
\begin{equation*}
Z_H(T^{\io\th})/(Z_H(T^{\io\th}) \cap B^{\th}) 
\simeq (SL_2/B_2)^n \simeq \BP_1^n, 
\end{equation*}
and $\e$ is a finite Galois covering with group 
$\CW = N_H(T^{\io\th})/Z_H(T^{\io\th}) \simeq S_n$  
through the identifications
\begin{equation*}
H\times^{Z_H(T^{\io\th})}T^{\io\th}\reg 
  \simeq H/Z_H(T^{\io\th}) \times T^{\io\th}\reg
\to (H/Z_H(T^{\io\th}) \times T^{\io\th}\reg)/S_n \simeq G^{\io\th}\reg,
\end{equation*}
where $S_n$ acts on $H/Z_H(T^{\io\th}) \times T^{\io\th}\reg$ 
by $(w, (gZ_H(T^{\io\th}), t)) \mapsto (gw\iv Z_H(T^{\io\th}), wtw\iv)$. 
\par
Summing up the above computation, we obtain the following lemma.
%%%
\begin{lem}  %%%  Lemma 1.9
Let $U$ be the unipotent radical of $B$.  Then 
\begin{enumerate}
\item
$\dim U^{\th} = n^2$, $\dim B^{\th} = n^2 + n$, 
$\dim U^{\io\th} = n^2 - n$, $\dim B^{\io\th} = n^2$.
\item
$\wt G^{\io\th}$ is a smooth irreducible variety with 
$\dim \wt G^{\io\th} = 2n^2$.  
$\pi$ is a proper surjective map from $\wt G^{\io\th}$ onto 
$G^{\io\th}$, and $\dim G^{\io\th} = 2n^2 - n$.
\end{enumerate}
\end{lem}

\begin{proof}
For (i), the statement for $B^{\th}, U^{\th}$ is well-known.
The statement for $B^{\io\th}, U^{\io\th}$ comes from (1.2.2).
For (ii), $\dim \wt G^{\io\th} = \dim H/B^{\th} + \dim B^{\io\th}
    = \dim U^{\th} + \dim B^{\io\th} = 2n^2$ by (i).
By 1.4, $\dim G^{\io\th} = \dim G - \dim H = 2n^2 - n$.  
\end{proof}

\para{1.10.}
We define a variety 
\begin{equation*}
\wt G^{\io\th}\uni = \{ (x, gB^{\th}) \in G^{\io\th}\uni 
         \times H/B^{\th} \mid g\iv xg \in U^{\io\th}\}
\end{equation*}
and the map $\pi_1: \wt G^{\io\th}\uni \to G^{\io\th}\uni$ by
$\pi_1(x, gB^{\th}) \mapsto x$. 
Then $\wt G^{\io\th}\uni \simeq H \times^{B^{\th}}U^{\io\th}$
is a vector bundle over $H/B^{\th}$ with fibre isomorphic to 
$U^{\io\th}$.  Hence $\wt G^{\io\th}\uni$ is smooth and irreducible, 
and $\pi_1$ is a surjective map onto $G^{\io\th}\uni$.
For $x \in G^{\io\th}\uni$, we consider 
$\pi_1\iv(x) \simeq \{ gB^{\th} \in H/B^{\th} 
                        \mid g\iv xg \in U^{\io\th}\}$.
We have the following lemma. (i) and (ii) are proved in 
Lemma 5.11 and (2.3.11) in [BKS].  (iii) is immediate from (ii).
%%%
\begin{lem}  %%% Lemma 1.11
Assume that $x = y\th(y)\iv \in G^{\io\th}\uni$ with $y \in A\uni$.
Then
\begin{enumerate}
\item $\dim \pi_1\iv(x) = (\dim Z_H(x) - \rk H)/2$.
\item $\dim H - \dim Z_H(x) = 2(\dim A - \dim Z_A(y))$.
\item Let $y = y_{\la} \in A$ be an element corresponding to 
$\la \in \CP_n$.  Put $x_{\la} = x$  and 
let $\CO_{\la}$ be the $H$-orbit of $x_{\la}$ in $G^{\io\th}\uni$.  
Then we have $\dim \CO_{\la} = 2n^2 - 2n - 4n(\la)$.
In particular, $\dim \pi_1\iv(x) = 2n(\la) + n$.
\end{enumerate}
\end{lem}
By using this, one can show that 

%%%%
\begin{lem}  %%% Lemma 1.12
\begin{enumerate}
\item  $\wt G^{\io\th}\uni$ is smooth and irreducible 
with $\dim \wt G^{\io\th}\uni = 2n^2 - n$.
\item $\dim G^{\io\th}\uni = 2\dim U^{\io\th} = 2n^2 - 2n$.
\end{enumerate}
\end{lem}

\begin{proof}
By 1.10, we have $\dim \wt G^{\io\th}\uni = 
        \dim H/B^{\th} + \dim U^{\io\th} = 
           \dim U^{\th} + \dim U^{\io\th} = 2n^2 - n$.  
Hence (i) holds.  
We show (ii).  Let $y$ be a regular unipotent element in 
$A$, and put $x = y\th(y)\iv$.  
Since $\dim Z_A(y) = n$, we have  $\dim Z_H(x) = 3n$ by 
Lemma 1.11 (ii).   Hence $\dim \psi_1\iv(x) = n$ by Lemma 1.11 (i).
Since the $H$-orbit of $x$ is open dense in $G^{\io\th}\uni$, 
we see that $\dim G^{\io\th}\uni = \dim \wt G^{\io\th}\uni - n$.
Hence (ii) holds.  
\end{proof}

\para{1.13.} 
We consider the direct image complex $\psi_*\Ql$ of the constant
sheaf $\Ql$ on $\pi\iv(\wt G^{\io\th}\reg)$.  In the notation in 1.8, 
since $\xi$ is a locally trivial fibration with fibre 
isomorphic to $\BP_1^n$, we see that 
\begin{equation*}
\xi_*\Ql \simeq H^{\bullet}(\BP_1^n)\otimes \Ql \simeq 
\bigoplus_{J \subseteq [1,n]}\Ql[-2|J|], 
\end{equation*}
where $J \subseteq [1,n]$ means that $J$ runs over all the subsets 
of the set $[1,n] = \{ 1,2, \dots, n\}$, and 
$H^{\bullet}(\BP_1^n)$ denotes the graded space 
$\bigoplus_iH^{2i}(\BP_1^n,\Ql)$ which we regard as a complex 
$\bigoplus_{J \subseteq [1,n]}\Ql[-2|J|]$. 
On the other hand, since $\e$ is a finite Galois covering with group
$S_n$, we have $\End (\e_*\Ql) \simeq \Ql[S_n]$, the group algebra 
of $S_n$ over $\Ql$.
Thus $\e_*\Ql$ is decomposed as 
\begin{equation*}
\e_*\Ql \simeq \bigoplus_{\la \in \CP_n}V_{\la} \otimes \CL_{\la},
\end{equation*}
where $V_{\la}$ is the irreducible $S_n$-module corresponding to the
partition $\la \in \CP_n$, and  
$\CL_{\la} = \Hom_{S_n}(V_{\la}, \e_*\Ql)$
is an irreducible local system on $G^{\io\th}\reg$.
Thus we have
\begin{align*}
\tag{1.13.1}
\psi_*\Ql &\simeq \e_*\xi_*\Ql \\ 
&\simeq 
H^{\bullet}(\BP_1^n) \otimes \e_*\Ql \\
&\simeq \bigoplus_{J \subseteq [1,n]}
   \bigoplus_{\la \in \CP_n}V_{\la} \otimes \CL_{\la}[-2|J|].
\end{align*}
\par
We have the following result 
(cf. [H1, Proposition 6.1, Proposition 6.2]).
%%%
\begin{prop} %%% Prop. 1.14
$\pi_*\Ql[\dim G^{\io\th}]$ and  $(\pi_1)_*\Ql[\dim G^{\io\th}\uni]$
are semisimple complexes with $S_n$-action.  They are decomposed as 
\begin{align*}
\tag{1.14.1}
\pi_*\Ql[\dim G^{\io\th}] &\simeq H^{\bullet}(\BP_1^n) \otimes 
            \bigoplus_{\la \in \CP_n} 
           V_{\la} \otimes \IC (G^{\io\th}, \CL_{\la}) [\dim G^{\io\th}],  \\
 \tag{1.14.2}
(\pi_1)_*\Ql[\dim G^{\io\th}\uni] &\simeq 
     H^{\bullet}(\BP_1^n)\otimes \bigoplus_{\la \in \CP_n}
        V_{\la} \otimes \IC(\ol\CO_{\la^{\bullet}}, \Ql)
          [\dim \CO_{\la^{\bullet}}], \\
\end{align*}
where $\CO_{\la^{\bullet}}$ is a certain 
$H$-orbit in $G^{\io\th}\uni$, and the map 
$\la \mapsto \CO_{\la^{\bullet}}$ gives a bijective
correspondence between $\CP_n$ and $G^{\io\th}\uni/\sim_H$. 
\end{prop}

\begin{proof}
Put $\Fb = \Lie B, \Ft = \Lie T$ and $\Fu = \Lie U$.  
Then $\th$ acts on $\Fb$ and $\Fu$, and one can define 
$\Fb^{-\th}, \Ft^{-\th}$ and $\Fu^{-\th}$ as in the case of $\Fg$. 
Let $\Ft^{-\th}\reg$ be the set of $t \in \Ft^{-\th}$ such that 
all the eigenspaces of $t$ have dimension 2, and put 
$\Fg^{-\th}\reg = \bigcup_{g \in H}g(\Ft^{-\th})g\iv$, the set
of regular semisimple elements in $\Fg^{-\th}$.  Put 
\begin{align*}
\wt\Fg^{-\th} &= \{ (x, gB^{\th}) \in \Fg^{-\th} \times H/B^{\th}
                    \mid \Ad(g)\iv x \in \Fb^{-\th}\}, \\
\wt\Fg^{-\th}\nil &= \{ (x, gB^{\th}) \in \Fg^{-\th}\nil \times 
       H/B^{\th}\mid \Ad(g)\iv x \in \Fu^{-\th}\}, 
\end{align*}
and define maps $\pi': \wt\Fg^{-\th} \to \Fg^{-\th}$, 
$\pi_1': \wt\Fg^{-\th}\nil \to \Fg^{-\th}\nil$ by the second projections. 
Let $\psi': {\pi'}\iv(\Fg^{-\th}\reg) \to \Fg^{-\th}\reg$ be the
restriction of $\pi'$ on ${\pi'}\iv(\Fg^{-\th}\reg)$. 
Then it is known by Henderson [H1, Proposition 6.1, Proposition 6.2], that 
$\psi'_*\Ql, \pi'_*\Ql$ and $(\pi_1')_*\Ql$ are described as in 
(1.13.1), (1.14.1) and (1.14.2) by replacing irreducible local systems 
$\CL_{\la}$ on $G^{\io\th}$ 
by irreducible local systems $\CL'_{\la}$ on $\Fg^{-\th}$, 
and by replacing $H$-orbits in $G^{\io\th}\uni$ by $H$-orbits in 
$\Fg^{-\th}\nil$.
Note that $\dim G^{\io\th} = \dim \Fg^{-\th}$ and 
$\dim G^{\io\th}\uni = \dim \Fg^{-\th}\nil$.
Let $\log : G^{\io\th} \to \Fg^{-\th}$ be the map defined in 1.7.
Then we have the following cartesian diagram
\begin{equation*}
\begin{CD}
\wt G^{\io\th} @>\pi >> G^{\io\th} \\
@V\wt\log VV                @VV\log V       \\
\wt \Fg^{-\th} @>\pi' >>  \Fg^{-\th},
\end{CD}
\end{equation*}
where $\wt\log (gB^{\th}, x) = (gB^{\th}, \log x)$.
It follows from this that $(\log)^*(\pi'_*\Ql) \simeq \pi_*\Ql$.
Moreover, one can check that this diagram induces an isomorphism 
\begin{equation*}
\wt G^{\io\th}\reg \simeq 
    \wt\Fg^{-\th}\reg \times_{\Fg^{-\th}\reg}G^{\io\th}\reg.
\end{equation*}
It follows that $\log^*(\psi'_*\Ql) \simeq \psi_*\Ql$, 
and in particular, we have 
$\log^*(\IC(\Fg^{-\th}, \CL'_{\la})) \simeq \IC(G^{\io\th}, \CL_{\la})$
for each $\la \in \CP_n$.
Since $\pi$ is a proper map, $\pi_*\Ql$ is a semisimple complex on 
$G^{\io\th}$.  Then (1.14.1)  follows from the corresponding result of 
Henderson. (1.14.2) also follows from the result of Henderson since 
$\wt\log$ gives an isomorphism between $\wt G^{\io\th}\uni$ and 
$\wt \Fg^{-\th}\nil$.
\end{proof}

\para{1.15.}
We consider the diagram 
\begin{equation*}
\begin{CD}
T^{\io\th} @<\a_0<< \pi\iv(G^{\io\th}\reg) @>\psi >> G^{\io\th}\reg, 
\end{CD}
\end{equation*}
where $\a_0$ is a map defined by $\a_0(x, gB^{\th}) = p(g\iv xg)$
(here $p: B^{\io\th} \to T^{\io\th}$ is the natural projection).
Take a tame local system $\CE$ on $T^{\io\th}$ (i.e, 
a local system $\CE$ on $T^{\io\th}$ such that 
$\CE^{\otimes m} \simeq \Ql$ for an integer $m$ not divisible by 
$p = \ch \Bk$), and consider the complex 
$\psi_*\a_0^*\CE$ on $G^{\io\th}\reg$.
Now the map $\psi$ is decomposed $\psi = \e\circ \xi$ as in 
(1.8.1).  Accordingly, $\a_0$ is also decomposed as 
$\a_0 = \a_1\circ \xi$, where 
$\a_1: H \times^{Z_H(T^{\io\th})}T^{\io\th}\reg  
  \simeq H/Z_H(T^{\io\th}) \times T^{\io\th}\reg \to T^{\io\th}$
is a projection to the second factor. 
Since $\xi$ is 
a $\BP_1^n$-bundle, we see that 
$\xi_*\a_0^*\CE \simeq H^{\bullet}(\BP^n_1)\otimes \a_1^*\CE$.
Since $\e$ is a finite Galois covering with its Galois group $S_n$, we have 
\begin{equation*}
\e_*\a_1^*\CE \simeq \bigoplus_{\r \in \CA_{\CE}\wg}\r \otimes \CL_{\r}, 
\end{equation*}
where $\CA_{\CE} = \End(\e_*\a_1^*\CE)$ 
and $\CL_{\r}$ is an irreducible local system on $G^{\io\th}\reg$
defined by $\CL_{\r} = \End_{\CA_{\CE}}(\r, \e_*\a_1^*\CE)$.
corresponding to the irreducible module $\r \in \CA_{\CE}\wg$.
Note that $\CA_{\CE}$ is a twisted group algebra of the stabilizer 
$W_{\CE}$ of $\CE$ in $S_n$.
It follows that 

\begin{align*}
\tag{1.15.1}
\psi_*\a_0^*\CE 
   &\simeq \bigoplus_{\r \in \CA_{\CE}\wg} H^{\bullet}(\BP^n_1)
        \otimes \r \otimes \CL_{\r}.
\end{align*}
We define a complex $K_{T,\CE}$ on $G^{\io\th}$ by 
\begin{equation*}
\tag{1.15.2}
K_{T,\CE} = H^{\bullet}(\BP^n_1)\otimes 
         \bigoplus_{\r \in \CA_{\CE}\wg}\r \otimes 
    \IC(G^{\io\th}, \CL_{\r})[\dim G^{\io\th}].
\end{equation*}
\par
On the other hand, we consider a diagram
\begin{equation*}
\begin{CD}
T^{\io\th} @<\a <<  \wt G^{\io\th}  @>\pi >>  G^{\io\th},
\end{CD}
\end{equation*}
where $\a : \wt G^{\io\th} \to T^{\io\th}$ is given by 
$\a(x, gB^{\th}) = p (g\iv xg)$.  We consider the complex 
$\pi_*\a^*\CE[\dim G^{\io\th}]$ on $G^{\io\th}$.
The following result is due to Grojonowski [Gr, Lemma 7.4.4]. 
%%%
\begin{thm}   %%%  Th. 1.16
$\pi_*\a^*\CE[\dim G^{\io\th}] \simeq K_{T,\CE}$ as 
semisimple complexes on $G^{\io\th}$.
\end{thm}

\para{1.17.}
Let $\CB^H = H/B^{\th}$ be the variety of Borel 
subgroups in $H$.  For $x \in G^{\io\th}$, put 
$\CB^H_x = \{ gB^{\th} \in \CB^H \mid g\iv xg \in B^{\io\th}\}$.
In the case where $x \in G^{\io\th}\uni$, $\CB^H_x \simeq \pi_1\iv(x)$. 
We shall describe the structure of $\CB^H_x$.
First consider $\CB^H_s$ for a semisimple element $s \in G^{\io\th}$.  
By Remark 1.3 (iii), there exists $s' \in T^{\io\th}$ such that 
$s'$ is $H$-conjugate to $s$.
Let $W_H = N_H(T^{\th})/T^{\th}$ be a Weyl group of $H$.
For $s' \in T^{\io\th}$, let $W_{H,s'} = \{ w \in W_H | w(s') = s'\}$, 
which is a Weyl group of $Z_H(s')$. 
Put 
\begin{equation*}
\CM_s = \{ g \in H \mid g\iv sg \in B^{\io\th} \}.
\end{equation*}
Then $Z_H(s) \times B^{\th}$ acts on $\CM_s$ from left and right, 
and we consider the set of double cosets $Z_H(s)\backslash \CM_s /B^{\th}$.  
We note that this set is labelled by the set 
$\vG = W_{H,s'}\backslash W_H$. 
Clearly it is enough to check this for the case where 
$s = s' \in T^{\io\th}$.  Then the claim follows from the following 
property.
\par\medskip\noindent
(1.17.1) \ Let $s \in T^{\io\th}$. 
Assume that $g\iv sg \in B^{\io\th}$ for $g \in H$.  Then 
there exists $g_1 \in Z_H(s)$ and $w \in W_H$ such that 
$gB^{\th} = g_1wB^{\th}$.
\par\medskip
We show (1.17.1). 
Since ${}^gB \cap Z_G(s)$ is a $\th$-stable 
Borel subgroup in $Z_G(s)$, there exists a $\th$-stable maximal 
torus $T'$ of $Z_G(s)$ contained in ${}^gB$.  Then by
Remark 1.3 (ii), there exists $g_1 \in Z_H(s)$  such that 
${}^{g_1}(B \cap Z_G(s)) = {}^gB \cap Z_G(s)$ and that 
${}^{g_1}T = T'$.  
Then ${}^{g_1\iv g}B$ is a Borel subgroup of $G$ containing $T$, 
and so ${}^{g_1\iv g}B^{\th}$ is a Borel subgroup of $H$ containing 
$T^{\th}$.  Hence there exists $w \in W_H$ such that 
$gB^{\th} = g_1wB^{\th} \in \CB^H$ as asserted.
\par
Now we have
\begin{align*}
\CB^H_s &= \{ gB^{\th} \in \CB^H \mid g\iv sg \in B^{\io\th}\} \\
            &= \coprod_{\g \in \vG}Z_H(s)x_{\g}B^{\th}/B^{\th} \\
 &= \coprod_{\g \in \vG}Z_H(s)/(Z_H(s) \cap x_{\g}B^{\th}x_{\g}\iv) \\
 &= \coprod_{\g \in \vG}Z_H(s)/B^{\th}_{\g},
\end{align*}
where $x_{\g} \in H$ is a representative of the double coset 
in $\CM_s$ corresponding to $\g \in \vG$.  Moreover,
$B_{\g} = Z_G(s) \cap x_{\g}B x_{\g}\iv$ is a $\th$-stable 
Borel subgroup of $Z_G(s)$, and $B^{\th}_{\g}$ is a Borel subgroup
of $Z_H(s)$.
\par
Next we consider the general case $\CB^H_{su}$, 
where $s$ is semisimple and $u$ is unipotent in $G^{\io\th}$.
We have
\begin{align*}
\CB^H_{su} &= \{ gB^{\th} \in \CB^H \mid 
          g\iv sg \in B^{\io\th}, g\iv ug \in B^{\io\th}\} \\
         &= \coprod_{\g \in \vG}
    \{ gB^{\th} \in Z_H(s)x_{\g}B^{\th}/B^{\th} \mid g\iv ug \in B^{\io\th}\}.
\end{align*}
By writing $gB^{\th} = g_1x_{\g}B^{\th}$ with $g_1 \in Z_H(s)$, 
the last formula turns out to be 
\begin{align*}
\tag{1.17.2}
\CB_{su}^H &= \coprod_{\g \in \vG} 
    \{ g_1B^{\th}_{\g} \in Z_H(s)/B^{\th}_{\g}
                \mid g_1\iv u g_1 \in B_{\g}^{\io\th}\}  \\
               &= \coprod_{\g \in \vG}\CB_u^{Z_H(s)},
\end{align*}
where $\CB^{Z_H(s)} \simeq Z_H(s)/B^{\th}_{\g}$ is the 
variety of Borel subgroups of $Z_H(s)$. 
\par It follows from the above computation, we have 
\begin{equation*}
\tag{1.17.3}
\dim \CB^H_{su} = \dim \CB_u^{Z_H(s)}.
\end{equation*}
Thus we have the following generalization of Lemma 1.11 (i).

\begin{lem}
For $x \in G^{\io\th}$, we have 
$\dim \CB^H_x = (\dim Z_H(x) - \rk H)/2$.
\end{lem}

\par\bigskip
%%%
%%%
\section{ $H$-orbits on  $G^{\io\th}\uni \times V$}
\par\medskip
\para{2.1.}
We consider the action of $H$ on the variety $G^{\io\th}\uni \times V$.
By (1.7.2), the set of $H$-orbits in $G^{\io\th}\uni \times V$ is
parametrized by the set $\CP_{n,2}$.  We denote by 
$\CO^H_{\Bla} = \CO_{\Bla}$ the 
$H$-orbit corresponding to $\Bla \in \CP_{n,2}$ (the explicit 
correspondence is described below).
For a given $(x,v) \in G^{\io\th}\uni \times V$, we say that
$(x,v)$ is of type $\Bla$ if $(x,v) \in \CO_{\Bla}$.
Let $M_n$ be the subspace of $V$ spanned by 
$e_1, \dots, e_n$.  Then $A$ acts on $M_n$ naturally, 
and the set of $A$-orbits in $A\uni \times M_n$ is 
parametrized also by $\CP_{n,2}$.  Let $\CO^A_{\Bla}$ be
the $A$-orbit in $A\uni \times M_n$ corresponding to 
$\Bla \in \CP_{n,2}$.  The correspondence 
$\CO^H_{\Bla} \lra \CO^A_{\Bla}$ is given as follows 
([AH, Theorem 6.1]);
Take $(y, v) \in A\uni \times M_n$ such that 
$(y, v) \in \CO^A_{\Bla}$. 
Then $x = y\th(y)\iv \in G^{\io\th}\uni$ and the $H$-orbit of 
$(x, v)$ coincides with $\CO^H_{\Bla}$. 
\par
On the other hand, for a given $(y, v) \in A\uni \times M_n$, the 
corresponding type is determined 
by the following procedure.  Put 
$E_A^y = \{ z \in \End(M_n) \mid zy = yz \}$.
Then $W = E_A^yv$ is an $y$-stable subspace of $M_n$. Let
$\la^{(1)}$ be the Jordan type of $y|_W$, and $\la^{(2)}$ 
the Jordan type of $y|_{M_n/W}$.  Then 
$\Bla = (\la^{(1)}, \la^{(2)})$ gives the type of $(y,v)$. 
In particular, $\dim E_A^yv = \dim W = |\la^{(1)}|$.  
Note that in this case, the Jordan type of $y$ is given by 
$\nu = \la^{(1)} + \la^{(2)}$.  (Here we use the notation;
for partitions $\la = (\la_i), \mu = (\mu_i)$, 
we denote the partition $\nu = (\la_i + \mu_i)$ by 
$\nu = \la + \mu$). 
Let $E_H^x = \Lie Z_H(x) \subset \Fg = \End(V)$.  
Then $E_H^xv$ is an $x$-stable subspace of $V$.
On the other hand, 
$(x,v)$ is regarded as an element in $G\uni \times V$.
Then it is easy to see that the $G$-orbit of $(x,v)$ is of 
type $\Bla \cup \Bla = 
(\la^{(1)}\cup \la^{(1)}, \la^{(2)}\cup \la^{(2)}) \in \CP_{2n,2}$.
We define $E^x_G = \{ z \in \End(V) \mid zx = xz\}$, and
put $\wt W = E^x_Gv$.
Then $\wt W$ is a subspace of $V$ containing $E_H^xv$ with 
$\dim \wt W = 2|\la^{(1)}|$.
We have the following lemmas.

\begin{lem} %%% Lemma 2.2 
$\wt W$ coincides with $E_H^xv$.  Hence $\dim E_H^xv = 2|\la^{(1)}|$.
\end{lem}

\begin{proof}
Let $L$ be as in 1.4.  Then
$L^{\th} = \{ a\th(a) \mid a \in A \}$ is a subgroup of $H$ 
isomorphic to $GL_n$, and $E_{L^{\th}}^x = \Lie Z_{L^{\th}}(x) \subset E_H^x$.  
Let $V = M_n\oplus M_n'$, where 
$M_n'$ is the subspace of $V$ generated by $f_1, \dots, f_n$.
One can find $v' \in M_n'$ such that the type of $(y', v')$
is $\Bla$, where $y' = \th(y)\iv \in \th(A)$.
Then there exists $g \in \Lie Z_H(x)$ such that $gv = v'$.
We have $E_{L^{\th}}^xv = E_A^yv$ and $E_{L^{\th}}^xv' = E_{\th(A)}^{y'}v'$.
It follows that $E_H^xv$ contains a subspace 
$E_A^yv \oplus E_{\th(A)}^{y'}v'$ whose  dimension 
is equal to $2|\la^{(1)}|$.  Hence $\wt W = E_H^xv$, and 
the lemma follows.
\end{proof}

\par
The following result is contained in [Ka1]. 
\begin{lem}  %%%  Lemma 2.3
Let $\CO_H(x,v)$ be the $H$-orbit of 
$(x,v) \in G^{\io\th}\uni \times  V$, where $(x,v)$ is of type
$\Bla$, and  
$\CO_H(x)$ the $H$-orbit of $x \in G^{\io\th}\uni$.
For $\Bla = (\la^{(1)},\la^{(2)})$ put $n(\Bla) = n(\la^{(1)} + \la^{(2)})$.
\begin{enumerate}
\item
$\dim \CO_H(x,v) = \dim \CO_H(x) + 2|\la^{(1)}|$.
In particular for $\CO_{\Bla} = \CO_H(x,v)$ we have 
\begin{equation*}
\tag{2.3.1}
\dim \CO_{\Bla} = 2n^2 - 2n - 4n(\Bla) + 2|\la^{(1)}|.
\end{equation*}
\item
$Z_H(x,v)$ is connected for any $(x,v) \in G^{\io\th}\uni \times V$. 
\end{enumerate}
\end{lem}

\begin{proof}
We consider the linear map $\vf : E_H^x \to E_H^xv$ 
given by $g \mapsto gv$.  Then by Lemma 2.2, 
$\dim E_H^x - \dim \ker \vf = 2|\la^{(1)}|$.
We note that 
\begin{equation*}
\tag{2.3.2}
\dim \ker\vf = \dim Z_H(x,v).
\end{equation*}
We show (2.3.2).
Let $H^+$ be the set of elements of $H$ such that
$-1$ is not contained in its eigenvalues, and 
$\Fh^+$ be the set of elements in $\Fh = \Lie H$ 
such that 1 is not contained in its eigenvalues.
Then $H^+$ (resp. $\Fh^+$) is open dense in $H$ 
(resp. $\Fh$).  By the Cayley transform, 
$f : Z \mapsto (I + Z)(I - Z)\iv$  gives a bijection 
from $\Fh^+$ to $H^+$.  The inverse $f\iv$ is 
given by $z \mapsto (I + z)\iv(I - z)$.
Then the map $f$ induces a bijection from 
$\ker \vf \cap \Fh^+$ onto $Z_H(x,v) \cap H^+$.
Since $\ker\vf \cap \Fh^+$ is open dense in $\ker \vf$,
and $Z_H(x,v) \cap H^+$ is open dense in $Z_H(x,v)$,
we obtain (2.3.2). 
\par
By (2.3.2) and by $\dim E_H^x = \dim Z_H(x)$, we have 
$\dim Z_H(x,v) = \dim Z_H(x) - 2|\la^{(1)}|$. 
It follows that 
\begin{align*}
\dim \CO_H(x,v) &= \dim H - \dim Z_H(x,v)  \\
                &= \dim H - \dim Z_H(x) + 2|\la^{(1)}|  \\
                &= \dim \CO_H(x) + 2|\la^{(1)}|.
\end{align*}
Now (2.3.1) follows from Lemma 1.11 (iii) since
the Jordan type of $x$ is $\la^{(1)} + \la^{(2)}$.
This proves (i).
\par 
(ii) follows from the proof of (i) as follows; 
since $\ker \vf \cap \Fh^+$ is connected, $Z_H(x,v) \cap H^+$ is
also connected. Since $Z_H(x,v) \cap H^+$ is open dense in $Z_H(x,v)$, 
we see that $Z_H(x,v)$ is connected.  
\end{proof}

\para{2.4.}
We fix an isotropic  flag 
$(0 = M_0 \subset M_1 \subset \cdots \subset M_n)$ whose stabilizer 
in $H$ coincides with $B^{\th}$, where 
$M_i$ is the subspace of $V$ spanned by $e_1, \dots, e_i$.
By fixing an integer $m$ such that $0 \le m \le n$, we define
\begin{align*}
\wt\CX_{m,\unip} &= \{ (x,v, gB^{\th}) \in G^{\io\th}\uni 
     \times V \times H/B^{\th} \mid g\iv xg \in U^{\io\th}, 
            g\iv v \in M_m \},  \\
\CX_{m,\unip} &= \bigcup_{g\in H}g(U^{\io\th} \times M_m).
\end{align*}
and a map $\pi_1^{(m)} : \wt \CX_{m,\unip} 
      \to G^{\io\th}\uni \times V$ by 
$\pi^{(m)}_1(x,v, B^{\th}) = (x,v)$.
Clearly $\Im \pi^{(m)}_1 = \CX_{m,\unip}$.  In the special case where 
$m = n$, we simply write $\wt\CX_{n,\unip}, \CX_{n, \unip}, \pi_1^{(m)}$ as
$\wt\CX\uni, \CX\uni, \pi_1$.  
$\wt\CX_{m,\unip}$ is smooth and irreducible since 
$\wt\CX_{m,\unip} \simeq H \times^{B^{\th}}(U^{\io\th} \times M_m)$.
Moreover, we have
\begin{equation*}
\tag{2.4.1}
\dim \wt\CX_{m,\unip} = \dim H/B^{\th} + \dim U^{\io\th} + m = 2n^2 - n + m.
\end{equation*}
\par
In order to obtain the dimension estimate for $\pi_1\iv(x,v)$, we 
consider the Steinberg variety defined as follows:
\begin{equation*}
\begin{split}
\CZ = \{ (x,v,gB^{\th}, &g'B^{\th}) \in G^{\io\th}\uni \times 
               V \times H/B^{\th}\times H/B^{\th} \\
             &\mid (x,v,gB^{\th}) \in \wt\CX\uni, 
                  (x,v,g'B^{\th}) \in \wt\CX\uni\}.
\end{split}
\end{equation*}
Let $\nu_H = \dim U^{\th} = n^2$. 
Let $W_n = N_H(T^{\th})/T^{\th}$ be a Weyl group of $H \simeq Sp_{2n}$,
the Weyl group of type $C_n$. 
We have the following lemma.

\begin{lem}  %%%  Lemma 2.5
\begin{enumerate}
\item  
The irreducible components of $\CZ$ are parametrized by 
$w \in W_n$, and have dimension $2\nu_H$. 

\item Let $\pi_1 : \wt\CX\uni \to G^{\io\th}\uni \times V$. 
For any $(x,v) \in \CO_{\Bmu}$, we have
$\dim \pi_1\iv(x,v) \le \nu_H - \dim \CO_{\Bmu}/2$.

\item  Let $c_{\Bmu}$ be the number of irreducible components 
of $\pi_1\iv(x,v)$ for $(x,v) \in \CO_{\Bmu}$ whose dimension 
is equal to $\nu_H - \dim \CO_{\Bmu}/2$.  Then we have
\begin{equation*}
\sum_{\Bmu \in \CP_{n,2}}c^2_{\Bmu} \le |W_n|.
\end{equation*}
\end{enumerate}

\end{lem}

\begin{proof}
Let $p: \CZ \to H/B^{\th} \times H/B^{\th}$ be the 
projection to the last two factors.   For each $w \in W_n$, 
let $X_w$ be the $H$-orbit of $(B^{\th}, wB^{\th})$ in 
$H/B^{\th} \times H/B^{\th}$.  We have 
$H/B^{\th} \times H/B^{\th} = \coprod_{w \in W_n}X_w$.
Put $Z_w = p\iv(X_w)$.  Then $Z_w$ is a vector bundle 
over $X_w \simeq H/B^{\th} \cap wB^{\th}w\iv$ with fibre 
isomorphic to $(U^{\io\th} \cap wU^{\io\th}w\iv) \times 
(M_n \cap w(M_n))$.
Let $b_w$ be the number of $i$ such that $w\iv(e_i) \in M_n$.
Then $\dim (M_n \cap w(M_n)) = b_w$.  By counting the root
vectors in $U^{\io\th}$, we see that 
$\dim (U^{\io\th} \cap wU^{\io\th}w\iv) = 
\dim (U^{\th} \cap wU^{\th}w\iv)  - b_w$.
It follows  that 
\begin{equation*}
\dim Z_w = \dim H - \dim T^{\th} = 2\nu_H.
\end{equation*}
This implies that the set $\{ \ol Z_w \mid w \in W_n\}$ gives 
rise to the set of irreducible components of $\CZ$, any of 
their dimension is equal to $2\nu_H$.  Hence (i) follows.
\par
Let $f : \CZ \to G^{\io\th}\uni \times V$ be the projection 
to the first two factors.  For each $H$-orbit 
$\CO$ of $G^{\io\th}\uni \times V$ containing $(x,v)$, the fibre
$f\iv(\CO)$ is isomorphic to a locally trivial fibration 
over $\CO$ with fibre
isomorphic to $\pi_1\iv(x,v) \times \pi_1\iv(x,v)$, i.e., we have
\begin{equation*}
f\iv(\CO) \simeq H \times^{Z_H(x,v)}(\pi_1\iv(x,v) \times \pi_1\iv(x,v)).
\end{equation*}
It follows that
\begin{equation*}
\dim f\iv(\CO) = 2 \dim \pi_1\iv(x,v) + \dim \CO \le 2\nu_H.
\end{equation*}
(ii) follows from this. 
\par 
We have
$\CZ = \coprod_{\Bmu \in \CP_{n,2}}f\iv(\CO_{\Bmu})$.
Take $(x,v) \in \CO_{\Bmu}$ and let 
 $I_{\Bmu}$ be the set of irreducible components of 
dimension $\nu_H - \dim \CO_{\Bmu}/2$ of $\pi_1\iv(x,v)$.
$Z_H(x,v)$ acts on $\pi\iv(x,v)$, which stabilizes  
each irreducible component of $\pi_1\iv(x,v)$
since $Z_H(x,v)$ is connected.   
For $X, Y \in I_{\Bmu}$, $H\times^{Z_H(x,v)}(X \times Y)$
gives an irreducible subset of $f\iv(\CO_{\Bmu})$ with 
dimension $2\nu_H$.  It follows that its closure gives an irreducible
component of $f\iv(\CO_{\Bmu})$, and that the number of irreducible
components of $f\iv(\CO_{\Bmu})$ of dimension $2\nu_H$ is bigger 
than or equal to $|I_{\Bmu}|^2 = c^2_{\Bmu}$.
The inequality in (iii)  follows from this.  
\end{proof}

\remark{2.6.}
Later it is shown in the course of the proof of Theorem 5.4 (i), 
that the inequality in Lemma 2.5 (iii) is actually an equality. 
Independently,  it is proved by Kato 
[Ka2, Proposition 7.7] that for $(x,v) \in \CO_{\Bmu}$
all the irreducible components of $\pi_1\iv(x,v)$ 
have dimension $\nu_H - \dim \CO_{\Bmu}$.
\par\medskip
More generally, we have the following 
dimension estimate for $(\pi^{(m)}_1)\iv(x,v)$ for 
$(x,v) \in G^{\io\th}\uni \times V$.

\begin{cor}  % Cor. 2.7
Let $\Bla = ((m), (n-m)) \in \CP_{n,2}$. 
Let $(x,v) \in \CO_{\Bmu}$  and assume that 
$\Bmu \le \Bla$ (see 1.7).  Then 
\begin{align*}
\tag{2.7.1}
\dim (\pi^{(m)}_1)\iv(x,v) &\le 2n(\Bmu) + n - |\mu^{(1)}| \\
                   &= (\dim \CO_{\Bla} - \dim \CO_{\Bmu}) /2 + (n-m). 
\end{align*}
\end{cor}

\begin{proof}
By Lemma 2.5 (ii) and Lemma 2.3 (i), we have 
\begin{align*}
\dim (\pi^{(n)}_1)\iv(x,v) &\le \nu_H - \dim \CO_{\Bmu}/2  \\
       &= 2n(\Bmu) + n - |\mu^{(1)}|,
\end{align*}
Clearly 
$(\pi_1^{(m)})\iv(x,v) \subseteq (\pi_1^{(n)})\iv(x,v)$
and so $\dim (\pi_1^{(m)})\iv(x,v) \le \dim (\pi_1^{(n)})\iv(x,v)$.
Then (2.7.1) follows from the above inequality by noting that 
$\dim \CO_{\Bla} = 2n^2 - 2n + 2m$.
\end{proof}

By making use of Corollary 2.7, we have the following result.
The assertion (ii) is known by Kato [Ka1, Theorem 1.2]. 
\begin{prop}  %%%%  Prop. 2.8
\begin{enumerate}
\item
Let $\Bla = ((m), (n-m)) \in \CP_{n,2}$.
Then $\pi^{(m)}_1$ is a map from $\wt\CX_{m,\unip}$ onto 
$\CX_{m,\unip} = \ol\CO_{\Bla}$.
\item
In the case where $m = n$, $\CX\uni$ coincides with $\ol\CO_{\Bla}$
with $\Bla = ((n), -)$.  The map $\pi_1 : \wt\CX\uni \to \CX\uni$ gives
a resolution of singularities.
\end{enumerate}
\end{prop}

\begin{proof}
We show (i). 
By Corollary 2.7, for each $(x,v) \in \CO_{\Bla}$, 
$\dim (\pi^{(m)}_1)\iv(x,v) \le n - m$.  
We may assume that $x = y\th(y)\iv$ for $y \in A\uni$.   
Put 
\begin{equation*}
X = \{ (x,v, gB^{\th}) \in (\pi^{(m)}_1)\iv(x,v) \mid 
                 gM_m = E_A^yv \}.
\end{equation*}
Put $V_m = E_A^yv$.  Then $x|_{V_m} \in GL_m$ is of type 
$(m)$, and $x|_{V_m^{\perp}/V_m} \in GL^{\io\th}_{2(n-m)}$ 
is of type $(n-m)$.  It follows that 
$X \simeq {\pi'}_1\iv(x|_{V_m})\times {\pi''}\iv_1(x|_{V_m^{\perp}/V_m})$,
where $\pi''_1$ is the map defined as in 1.10 replacing $n$ by $n-m$, 
and $\pi'_1$ is a similar map with respect to $GL_m$.  
 It follows, by Lemma 1.11 (iii) that  
$X$ is irreducible with $\dim X = n-m$.  Hence we see that 
$\dim (\pi^{(m)}_1)\iv(x,v) = n-m$.  Since 
$(x,v) \in \Im \pi^{(m)}_1$, we have
$\CO_{\Bla} \subset \Im \pi^{(m)}_1$.
It follows that 
$\dim (\pi^{(m)}_1)\iv(\CO_{\Bla}) = \dim \CO_{\Bla} + (n-m) = 2n^2 - n + m$.
Hence by (2.4.1), the closure of $(\pi^{(m)}_1)\iv(\CO_{\Bla})$ coincides with 
$\wt\CX_{m,\unip}$.  
This implies that $\wt\CX_{m,\unip} = (\pi^{(m)}_1)\iv(\ol\CO_{\Bla})$, and 
so $\Im \pi^{(m)}_1 = \ol\CO_{\Bla}$ as asserted.
\par
Next we show (ii). By (i),  $\CX\uni = \ol\CO_{\Bla}$. 
We have $\dim \wt\CX\uni = \dim\CX\uni$, and  $\wt\CX\uni$ is smooth, 
$\pi_1$ is proper. Hence in order to show that $\pi_1$ is a resolution 
of singularities, it is enough to see that the restriction of
$\pi_1$ gives an isomorphism  
$\pi_1\iv(\CO_{\Bla}) \isom \CO_{\Bla}$.  
In this case, $\Bla = ((n); \emptyset)$, 
and one can choose a representative $(x,v) \in \CO_{\Bla}$ such that
$x = y\th(y)\iv$ with $y \in A$ regular unipotent. Then 
there exists a unique Borel subgroup of $H$ 
containing $x$.  It follows that $\pi_1\iv(x,v)$ consists of 
exactly one point.   
We consider the map $f: \CZ \to \CX\uni$ in the proof of Lemma 2.5.
Since $f\iv(\CO_{\Bla})$ is a principal bundle  
with fibre $\pi_1\iv(x,v) \times \pi_1\iv(x,v)$, we see that 
$\CO_{\Bla} \simeq f\iv(\CO_{\Bla}) \simeq \pi_1\iv(\CO_{\Bla})$.
This proves (ii).
\end{proof}

\section{ Intersection cohomology on $G^{\io\th}\reg \times V$}

\para{3.1.} 
We consider the variety $G^{\io\th} \times V$ on which 
$H$ acts diagonally.
We define an isotropic subspace $M_i$ of $V$ by 
$M_i = \lp e_1, \dots, e_i\rp$ for $i = 0, 1, \dots, n$.
Thus the stabilizer of the isotropic flag $(M_i)$ in $H$ coincides with 
$B^{\th}$.  
For $0 \le m \le n$, we define  varieties

\begin{align*}
\wt\CX_m &= \{ (x,v, gB^{\th}) \in G^{\io\th} \times V \times H/B^{\th}
       \mid g\iv xg \in B^{\io\th}, g\iv v \in M_m \}, \\
\CX_m &= \bigcup_{g \in H}g(B^{\io\th} \times M_m).  
\end{align*}
We define
a map $\pi^{(m)} : \wt\CX_m \to G^{\io\th} \times V$ by 
$\pi^{(m)}(x,v, gB^{\th}) = (x,v)$.
Clearly $\Im \pi^{(m)} = \CX_m$.  In the special case where
$m = n$, we write $\wt\CX_n, \CX_n$ and $\pi^{(n)}$ by 
$\wt\CX, \CX$ and $\pi$.
\par
Since $\wt\CX_m \simeq H \times^{B^{\th}}(B^{\io\th} \times M_m)$, 
$\wt\CX_m$ is smooth and irreducible, and $\pi^{(m)}$ is a proper map. 
Hence $\CX_m$ is a closed irreducible subset of 
$G^{\io\th}\times V$. 
The dimension of $\wt \CX_m$ is computed as follows;

\begin{align*}
\tag{3.1.1}
\dim \wt \CX_m  &= \dim H/B^{\th} + \dim B^{\io\th} + m \\
               &= \dim U^{\th} + \dim B^{\io\th} + m \\
               &= \dim U + \dim T^{\io\th} + m \\
               &= 2n^2 + m.
\end{align*}
\par 
We also define varieties
\begin{align*}
\wt \CY_m &= \{ (x,v, gB^{\th}) \in 
    G^{\io\th}\reg \times V \times H/B^{\th} \mid 
        g\iv xg \in B^{\io\th}\reg, g\iv v \in M_m \},  \\
\CY_m &= \bigcup_{g \in H} g(B^{\io\th}\reg \times M_m),
\end{align*}
and a map $\psi^{(m)} : \wt\CY_m \to G^{\io\th} \times V$ by 
$\psi^{(m)}(x,v, gB^{\th}) = (x,v)$. Clearly $\Im \psi^{(m)} = \CY_m$.
As before, in the case where $m = n$, we write $\wt\CY_m, \CY_m$
and $\psi^{(m)}$ as $\wt\CY, \CY$ and $\psi$.
As in 1.8, $\wt\CY_m$ can be expressed in the following form. 
\begin{align*}
\tag{3.1.2}
\wt\CY_m &\simeq H \times^{B^{\th}}(B^{\io\th}\reg \times M_m) \\  
         &\simeq H \times^{B^{\th} \cap Z_H(T^{\io\th})}
                            (T^{\io\th}\reg \times M_m).
\end{align*}

\para{3.2.}
For a vector $v = \sum_{i=1}^na_ie_i$ of $M_n$, put
$\supp (v) = \{ i \mid a_i \ne 0\ \}$.   For a  
subset $I$ of $[1,n] = \{1, \dots, n \}$, put 
$M_{I} = \{ v \in M_n \mid  \supp(v) = I \}$. 
As in 1.8, $Z_H(T^{\io\th})$ is isomorphic to 
$SL_2 \times \cdots \times SL_2$, and under this identification
$B^{\th} \cap Z_H(T^{\io\th})$ corresponds to $B_2 \times \cdots \times B_2$.
Note that the action of $B^{\th} \cap Z_H(T^{\io\th})$ on $M_n$ is given 
by the action of its $T^{\th}$ part.
Hence
 $T^{\io\th}\reg \times M_I$ is $B^{\th} \cap Z_H(T^{\io\th})$-stable. 
Under the expression (3.1.2) for $\wt\CY$, 
we define, for $I \subset [1, n]$, 
a subvariety $\wt\CY_{I}$ of $\wt\CY$ by 
\begin{equation*}
\wt\CY_{I} \simeq H\times^{B^{\th} \cap Z_H(T^{\io\th})}
         (T^{\io\th}\reg \times M_{I}).
\end{equation*}
We define a map $\psi_I : \wt\CY_I \to \CY$ 
by $(x,v, g(B^{\th}\cap Z_H(T^{\io\th})) \mapsto (x,v)$.
Then $\Im \psi_I = \bigcup_{g \in H}g(T^{\io\th}\reg \times M_I)$
coincides with $\CY_m^0 = \CY_m \backslash \CY_{m-1}$ for
$m = |I|$, which depends only on $m$.
For $I \subset [1,n]$ we 
define a parabolic subgroup $Z_H(T^{\io\th})_I$ of 
$Z_H(T^{\io\th})$ by the condition that the $i$-th factor is 
$B_2$ if $i \in I$, and is $SL_2$  otherwise.
Since $Z_H(T^{\io\th})_I$ stabilizes $M_I$, one can define 
\begin{equation*}
\wh\CY_I = H\times^{Z_H(T^{\io\th})_I}(T^{\io\th}\reg \times M_I ).
\end{equation*}
Then the map $\psi_I$ factors through $\wh\CY_I$, 
\begin{equation*}
\tag{3.2.1}
\begin{CD}
\psi_I : \wt\CY_I  @>\xi_I>> \wh\CY_I   @> \e_I >>  \CY^0_{m},
\end{CD}
\end{equation*}
for $|I| = m$, 
where under the expression in (3.1.2), the map $\xi_I$
is the natural surjection, and the map $\e_I$ is given by 
$g*(t,v) \mapsto (gtg\iv, gv)$ ($g*(t,v)$ denotes the 
$Z_H(T^{\io\th})_I$-orbits of $(g, (t,v)) \in H \times 
(T^{\io\th}\reg \times M_I)$).
Then $\xi_I$ is a locally trivial fibration with 
fibre isomorphic to 
\begin{equation*}
\tag{3.2.2}
Z_H(T^{\io\th})_I/(B^{\th} \cap Z_H(T^{\io\th})) \simeq (SL_2/B_2)^{I'} 
  \simeq \BP_1^{I'},
\end{equation*}
where $I'$ is the complement of $I$ in $[1,n]$, and 
$(SL_2/B_2)^{I'}$ denotes the direct product of $SL_2/B_2$
with respect to the factors corresponding to $I'$, and
similarly for $\BP_1^{I'}$.  Thus $\BP_1^{I'} \simeq \BP_1^{n-|I|}$.
\par
Let $S_I \simeq S_{|I|} \times S_{n-|I|}$ be the subgroup of $S_n$
stabilizing the set $I$.
Then $ \CW_I = N_H(Z_H(T^{\io\th})_I)/Z_H(T^{\io\th})_I$ is isomorphic to 
$S_I$.  In the case where $I = [1,m]$, we put 
$S_I = S_{\Bm}$, and $\CW_I = \CW_{\Bm}$ for $\Bm = (m, n-m)$.  Thus 
$\CW_{\Bm} \simeq S_{\Bm} \simeq S_m \times S_{n-m}$ under the natural 
isomorphism $\CW \simeq S_n$ in 1.8.  
For $I \subset [1,n]$, $\CW_I$ acts on $\wt\CY_I$ and 
$\wh\CY_I$ since $T^{\io\th}\reg \times M_I$ is stable by 
$N_H(Z_H(T^{\io\th})_I)$. 
Now the map 
$\e_I : \wh\CY_I \to \CY^0_{m}$ 
can be identified with the finite Galois covering 
with Galois group $\CW_I$, 
\begin{equation*}
\tag{3.2.3}
\wh\CY_I \to \wh\CY_I/\CW_I \simeq \CY^0_{m}.
\end{equation*}

We have the following lemma.
\begin{lem}  %%% Lemma 3.3.
Let the notations be as before.
\begin{enumerate}
\item
$\CY_m$ is open dense in  $\CX_m$, and 
$\wt\CY_m$ is open dense in $\wt\CX_m$.
\item
$\dim \wt\CX_m = \dim \wt \CY_m = 2n^2 + m$.
\item 
$\dim \CX_m = \dim \CY_m = (2n^2 + m) - (n-m)$.
\item
$\CY = \coprod_{0 \le m \le n}\CY_m^0$ gives a stratification 
of $\CY$ by smooth strata $\CY_m^0$, and the map 
$\psi : \wt\CY \to \CY$ is semismall with respect to this 
stratification.
\end{enumerate}
\end{lem}

\begin{proof}
Since $\wt\CY_m \simeq H\times^{B^{\th}}(B^{\io\th}\reg \times
 M_{m})$, and $B^{\io\th}\reg \times M_{m}$ is open dense
in $B^{\io\th} \times M_m$, $\wt\CY_m$ is open dense in $\wt\CX_m$.
Since $\psi^{(m)}$ is a closed map, and since 
$\wt\CY_m = (\pi^{(m)})\iv(\CY_m)$, we see 
that $\CY_m$ is open dense in $\CX_m$. So (i) holds.   
(ii) follows from (3.1.1).
By using the decomposition $\psi_I = \e_I\circ \xi_I$ for 
$I = [1,m]$, we see that 
$\dim \wt\CY_m = \dim \CY_m + (n-m)$. Hence (iii) follows.
For (iv), $\CY_{m-1}$ is closed in $\CY_m$, and 
$\CY_m^0$ is an open dense smooth subset of $\CY_m$
by the description of 3.2.  Hence it gives the required stratification.
Since 
$\dim \psi\iv(x,v) = n-m$ for $(x,v) \in \CY_m^0$ by 3.2, we have 
$\dim \psi\iv(x,v) = (\dim \CY - \dim \CY^0_m)/2$ by (iii).
Thus (iv) holds.
\end{proof}

\para{3.4.}
For $0 \le m \le n$, we define $\wt \CY^+_{m}$ 
as $\psi\iv(\CY^0_{m})$.
Then by 3.2, we have $\wt\CY^+_{m} = \coprod_I\wt\CY_I$, where
$I$ runs over the subsets of $[1,n]$ such that $|I| = m$.
$\wt \CY_I$ are smooth and irreducible, and they form the 
connected components of $\wt\CY^+_{m}$. 
Since $\CY = \coprod_{0\le m\le n}\CY^0_{m}$, we have 
$\wt\CY = \coprod_{0\le m\le n}\wt\CY^+_{m}$.
In the case where $I = [1,m]$, we denote $M_I$ by $M_m^0$, and  
denote $\wt\CY_I$ by $\wt\CY_m^0$. 
$M_m^0$ is an open dense subset of $M_m$, and $\wt\CY_m^0$ is an open 
dense subset of $\wt\CY_m$.
By (3.1.2), $\CW$ acts on $\wt\CY$, which leaves $\wt\CY^+_{m}$ stable 
for any $m$.
Then we have
\begin{equation*}
\tag{3.4.1}
\wt\CY^+_{m} = \coprod_{\substack{I \subset [1,n] \\ 
          |I| = m}}\wt\CY_I = 
\coprod_{w \in \CW /\CW_{\Bm}}w(\wt\CY^0_{m}).
\end{equation*}
We denote by $\psi_m : \wt\CY_m^+ \to \CY_m^0$ the restriction of 
$\psi$ on $\wt\CY_m^+$.  Then $\psi_m$ is $\CW$-equivariant with 
respect to the natural action of $\CW$ on $\wt\CY_m^+$ and 
the trivial action on $\CY_m^0$. 
\par
We consider the diagram
\begin{equation*}
\begin{CD}
T^{\io\th} @<\a_0 <<  \wt\CY  @>\psi>>  \CY, 
\end{CD}
\end{equation*}
where 
$\a_0 : \wt\CY \to T^{\io\th}$ is given by 
$\a_0(x,v,gB^{\th}) = p(g\iv xg)$.
Let $\CE$ be a tame local system on $T^{\io\th}$, and we consider 
the complex $\psi_*\a_0^*\CE$ on $\CY$. 
One can define a map $\a_I : \wt\CY_I \to T^{\io\th}$ compatible with
$\a_0$ with respect to the inclusion $\wt\CY_I \hookrightarrow \wt\CY$.
Then by (3.4.1), we have
\begin{equation*}
\tag{3.4.2}
(\psi_m)_*\a_0^*\CE|_{\wt\CY^+_m} \simeq  
   \bigoplus_{\substack{I \subset[1,n] \\  |I| = m }}(\psi_I)_*\a_I^*\CE.
\end{equation*}
We define a map $\b_I : \wh\CY_I \to T^{\io\th}$ by 
$g*(t,v) \mapsto t$.   Then $\a_I = \b_I\circ \xi_I$.
Let $\CE_I = \b_I^*\CE$ be a local system on $\wh \CY_I$ and 
$\CW_{\CE_I}$ the stabilizer of $\CE_I$ in $\CW_I$.
In the case where $I = [1,m]$, we put $\CW_{\CE_I} = \CW_{\Bm,\CE}$.
Also put $\CW_{\CE_I} = \CW_{\CE}$ for $I = [1,n]$.
$\CW_{\CE}$ acts on $(\psi_m)_*\a_0^*\CE|_{\wt\CY_m^+}$ as 
automorphisms of complexes, and permutes each direct summand 
$(\psi_I)_*\a_I^*\CE$ according to the permutation of the sets $I$ by $S_n$.  
Put $\End((\e_I)_*\CE_I) = \CA_{\CE_I}$.
Since $\e_I$ is a finite Galois covering with group $\CW_I$, 
$(\e_I)_*\CE_I$ is decomposed as 
\begin{equation*}
\tag{3.4.3}
(\e_I)_*\CE_I \simeq \bigoplus_{\r \in \CA\wg_{\CE_I}}
                \r \otimes \CL_{\r}, 
\end{equation*}
where $\CL_{\r} = \Hom (\r, (\e_I)_*\CE_I)$ 
is a simple local system on $\CY^0_m$.
We note that $\CA_{\CE_I}$ is canonically isomorphic to the group
algebra $\Ql[\CW_{\CE_I}]$. 
In fact, by [L3, 10.2], $\CA_{\CE_I}$ is a twisted group algebra 
of $\CW_{\CE_I}$. 
However, in the case where $\CE$ is the constant sheaf $\Ql$, 
$\CA_{\CE_I}$ is canonically isomorphic to the group algebra 
$\Ql[\CW_I]$. In the general case, there exists a canonical 
embedding $\CA_{\CE_I} \hra \CA_{(\Ql)_I} \simeq \Ql[\CW_I]$, and 
this implies that $\CA_{\CE_I}$ is the group algebra (see [L4, 2.4]). 
\para{3.5.}
Since $\psi_m$ is proper and $\wt\CY_I$ is closed in $\wt\CY_m^+$, 
$\psi_I$ is proper.  Hence $\xi_I$ is also proper.
Since $\xi_I$ is a $\BP_1^{I'}$ bundle, we have  
$(\xi_I)_*\a_I^*\CE \simeq H^{\bullet}(\BP_1^{I'})\otimes \CE_I$.
It follows that 
\begin{equation*}
\tag{3.5.1}
(\psi_I)_*\a_I^*\CE \simeq (\e_I)_*(\xi_I)_*\a_I^*\CE \simeq 
   H^{\bullet}(\BP_1^{I'})\otimes (\e_I)_*\CE_I.
\end{equation*}
\par\noindent
Now $\BP_1^{I'}$ is the flag variety of the group $(SL_2)^{I'}$
whose Weyl group is isomorphic to $(\BZ/2\BZ)^{I'}$.  
Let $T_2$ be a maximal torus of $B_2$.  Then $SL_2/T_2$ is an affine 
space bundle over $SL_2/B_2$, and  
$(\BZ/2\BZ)^{I'}$ acts on 
$H^{\bullet}(\BP_1^{I'},\Ql) \simeq H^{\bullet}((SL_2/T_2)^{I'},\Ql)$ 
naturally (the Springer action of $(\BZ/2\BZ)^{I'}$  on 
$(SL_2/B_2)^{I'}$).  
We put $\CA_{\CE_I}= \CA_{\Bm,\CE}$ if $I = [1,m]$, 
and put $\CA_{\CE_I} = \CA_{\CE}$ if $I = [1,n]$. 
Then $\CA_{\Bm,\CE}$ is a subalgebra of $\CA_{\CE}$.
Recall that 
$W_n = S_n \ltimes (\BZ/2\BZ)^n$ is the Weyl group of type $C_n$, 
and put $\wt\CW = \CW \ltimes  (\BZ/2\BZ)^n$.
We define a subgroup $\wt\CW_{\CE}$ (resp. 
$\wt\CW_{\CE_I}$) of $\wt\CW$ by 
$\wt\CW_{\CE} = \CW_{\CE} \ltimes (\BZ/2\BZ)^n$ (resp.
$\wt\CW_{\CE_I} = \CW_{\CE_I}\ltimes (\BZ/2\BZ)^n$).
We define an algebra $\wt\CA_{\CE}$ by 
$\wt\CA_{\CE} = \Ql[\wt\CW_{\CE}]$.  Similarly we define 
$\wt\CA_{\CE_I}$.
It follows from the above discussion that $\wt\CA_{\CE_I}$ acts
on $(\psi_I)_*\a_I^*\CE$. (The action of $\wt\CA_{\CE_I}$ on 
$(\e_I)_*\CE_I$ is the trivial extension of the action of $\CA_{\CE_I}$, 
and that of $\wt\CA_{\CE_I}$ on $H^{\bullet}(\BP_1^{I'})$ is obtained 
from the action of $(\BZ/2\BZ)^n$ together with $\CA_{\CE_I}$, where 
$(\BZ/2\BZ)^{I'}$ acts as defined above, and $(\BZ/2\BZ)^I$ acts 
trivially).  We put $\wt\CA_{\Bm,\CE} = \wt\CA_{\CE_I}$ for $I = [1,m]$.
Then in view of (3.4.2), (3.4.3) and (3.5.1), we see that 
\begin{equation*}
\tag{3.5.2}
(\psi_{m})_*\a_0^*\CE|_{\wt\CY_m^+} \simeq \bigoplus_{\r \in \CA_{\Bm,\CE}\wg}
  \Ind_{\wt\CA_{\Bm,\CE}}^{\wt\CA_{\CE}}
    \bigl(H^{\bullet}(\BP_1^{n-m}) \otimes \r \bigr)\otimes \CL_{\r},
\end{equation*}
where $\r$ is regarded as a $\wt\CA_{\Bm,\CE}$-module through 
the trivial extension from $\CA_{\Bm,\CE}$ to $\wt\CA_{\Bm,\CE}$, and 
$H^{\bullet}(\BP_1^{n-m})$ is regarded as a 
$S_{n-m}\ltimes (\BZ/2\BZ)^{n-m}$-module through the Springer action of 
$(\BZ/2\BZ)^{n-m}$ and the $S_{n-m}$ action arising from the 
permutations of factors $\BP_1$.

\par
The group $\CW_{\CE_I}$ is decomposed  as 
$\CW_1 \times \CW_2$, where $\CW_1$ (resp. $\CW_2$) is 
the stabilizer of $\CE$ in $S_I$ (resp. $S_{I'}$) under 
the isomorphism $\CW_I \simeq S_I \times S_{I'}$.
Accordingly, the algebra $\CA_{\Bm,\CE}$ can be decomposed as 
$\CA_{\Bm,\CE} \simeq \CA_1 \otimes \CA_2$, where 
$\CA_1$ (resp. $\CA_2$) is the twisted group algebra of $\CW_1$
(resp. $\CW_2$) for $I = [1,m]$.
Then  an irreducible 
$\CA_{\Bm,\CE}$-module $\r$ can be extended to an 
$\wt\CA_{\Bm,\CE}$-module $\r'$, where $\BZ/2\BZ$ acts 
trivially on $\CA_{\Bm,\CE}$ if it belongs to $I$, and non-trivially
if it belongs to $I'$. We denote by $\wt V_{\r}$ the induced 
$\wt\CA_{\CE}$-module $\wt\CA_{\CE}\otimes_{\wt\CA_{\Bm,\CE}}\r'$.  
\par
We show the following result.

\begin{prop}  %%%  Prop. 3.6
$\psi_*\a_0^*\CE[d_n]$ is a semisimple perverse sheaf on $\CY$, 
equipped with $\wt\CA_{\CE}$-action, and is
decomposed as
\begin{equation*}
\psi_*\a_0^*\CE[d_n] \simeq \bigoplus_{0 \le m \le n}
      \bigoplus_{\r \in \CA_{\Bm,\CE}\wg} 
           \wt  V_{\r} \otimes \IC(\CY_m, \CL_{\r})[d_m],
\end{equation*}
where $d_m = \dim \CY_m$.  Moreover, $\IC(\CY_m, \CL_{\r})$
is a constructible sheaf on $\CY_m$.
\end{prop}

\begin{proof}
For each $m \le n$, let $\ol\psi_m$ be the restriction of
$\psi$ on $\psi\iv(\CY_{m})$, and $\wt\CE_m$ be the restriction of
$\a_0^*\CE$ on $\psi\iv(\CY_m)$.  
The following fact holds.
\par\medskip\noindent
(3.6.1) \  For each $m$, $\IC(\CY_m, \CL_{\r})$ is a constructible 
sheaf on $\CY_m$, and we have 
\begin{equation*}
\begin{split}
(\ol\psi_m)_*\wt\CE_m[d_m] \simeq 
&\bigoplus_{\r \in \CA_{\Bm,\CE}\wg}\Ind_{\wt\CA_{\Bm,\CE}}^{\wt\CA_{\CE}}
      (H^{\bullet}(\BP_1^{n-m})
               \otimes \r)\otimes \IC(\CY_m,\CL_{\r})[d_m] \\
 &\oplus \bigoplus_{0 \le m' < m}\bigoplus_{\r \in \CA_{\Bm',\CE}\wg}
        \wt V_{\r} \otimes\IC(\CY_{m'}, \CL_{\r})[d_{m'} - 2(n-m)]
\end{split}
\end{equation*}
where $\wt\CA_{\Bm,\CE}$-module $H^{\bullet}(\BP_1^{n-m})\otimes \r$
is given as in (3.5.2), $\wt\CA_{\CE}$-module $\wt V_{\r}$ is as in 
3.5.
\par\medskip
Note that (3.6.1) for $m = n$ implies the proposition.
First we show that 
\par\medskip\noindent
(3.6.2) \ $\IC(\CY_m,\CL_{\r})$ is a constructible sheaf 
on $\CY_m$.
\par\medskip\noindent
Recall the map $\e_I : \wh\CY_m \to \CY_m^0$.
We define a variety $\ol{\wh\CY}_m$ by 
\begin{equation*}
\ol{\wh\CY}_m = H \times^{Z_H(T^{\io\th})_I}
           (T^{\io\th}\reg \times \ol M_I), 
\end{equation*}
where $\ol M_I$ is the closure of $M_I$ in $V$.  
We define a morphism $\ol \e_I : \ol{\wh \CY}_m \to \CY_m$ by
$g*(t,v) \mapsto (gtg\iv, gv)$.  Then $\ol\e_I$ is a finite morphism 
and $\ol{\wh \CY}_m$ is smooth.
Let $\ol\b_I : \ol{\wh \CY}_m \to T^{\io\th}$ be the map 
defined by $g*(t,v) \mapsto t$, and put $\ol\CE_I = \ol\b_I^*\CE$.  Then 
$(\ol\e_I)_*\ol\CE_I$ is a perfect sheaf in the sense of [L2, (5.4.4)].
In particular, $(\ol\e_I)_*\ol\CE_I$ is a direct sum of intersection 
cohomology complexes on $\CY_m$ which are constructible sheaves. 
Since $(\ol\e_I)_*\ol\CE_I|_{\CY_m^0} = (\e_I)_*\CE_I$, we see that 
$\IC(\CY_m,\CL_{\r})$ coincides with one of the simple direct summands in 
$(\ol\e_I)_*\ol\CE_I$.  Hence (3.6.2) holds.
\par
For each $m$, $\Bm = (m, n-m)$ and $0 \le i \le n-m$, define 
\begin{align*}
\CL_m &= \bigoplus_{\r \in \CA_{\Bm,\CE}\wg}\wt V_{\r}\otimes \CL_{\r}, \\
\CL_{m,i} &= \bigoplus_{\substack{J \subset [1, n-m] \\ |J| = i}}
               \CL_m.
\end{align*}
Then the formula in (3.6.1) is equivalent to the following formula.
\begin{equation*}
\tag{3.6.3}
(\ol\psi_m)_*\wt\CE_m \simeq \bigoplus_{0 \le i \le n-m}
   \IC(\CY_m, \CL_{m,i})[-2i] \oplus \bigoplus_{0 \le m' < m}
\IC(\CY_{m'}, \CL_{m'})[2m'-2n].
\end{equation*}
We show (3.6.3) by induction on $m$.
In the case where $m = 0$, (3.6.3) follows from (1.13.1).  
Assume that (3.6.3) holds for $m$.  
Since $\ol\psi_{m+1}$ is a proper map, by the decomposition  theorem
$(\ol\psi_{m+1})_*\wt\CE_{m+1}$ is a direct sum of the complexes 
$A[i]$ for various simple perverse sheaf $A$ on $\CY_{m+1}$.
Suppose that $\supp A$ is not contained in $\CY_{m}$.
Then $\supp A \cap \CY_{m+1}^0 \ne \emptyset$, 
and $A|_{\CY_{m+1}^0}$ appears as a direct summand in the decomposition 
$(\psi_m)_*\a_0^*\CE|_{\wt\CY_m^+}$.  It follows, by (3.5.2), that 
$A|_{\CY_{m+1}^0}$ coincides with $\CL_{\r}$ 
for some $\r \in \CA_{\Bm+1,\CE}$ with $\Bm+1 = (m+1, n-m-1)$. Hence 
$A$ coincides with $\IC(\CY_{m+1}, \CL_{\r})[d_{m+1}]$. 
This implies that 
the direct sum of $A[i]$ appearing in $(\ol\psi_{m+1})_*\wt\CE_{m+1}$
such that $\supp A \cap \CY^0_{m+1} \ne \emptyset$ is given by 
\begin{equation*}
\tag{3.6.4}
\bigoplus_{0 \le i \le n-m-1}\IC(\CY_{m+1}, \CL_{m+1,i})[-2i].
\end{equation*} 
Next suppose that $\supp A \subset \CY_{m}$. 
Then $A[i]$ appears as a direct summand of the semisimple complex 
$(\ol\psi_{m})_*\wt\CE_{m}[d_{m}]$.  
By induction hypothesis, 
the complex $(\ol\psi_m)_*\wt\CE_m$ is decomposed as in (3.6.3).
Comparing (3.6.3) and (3.6.4), and by using (3.6.2), we see 
that each factor $\IC(\CY_{m'}, \CL_{m'})[2m'-2n]$ in 
$(\ol\psi_m)_*\wt\CE_m$ for $0 \le m' \le m$ 
appears in  $(\ol\psi_{m+1})_*\wt\CE_{m+1}$ as a direct summand. 
Moreover, we see that 
$R^{2i}(\ol\psi_{m+1})_*\wt\CE_{m+1} = \IC(\CY_{m+1}, \CL_{m+1,i})$ 
and $R^{2i}(\ol\psi_m)_*\wt\CE_m = \IC(\CY_m, \CL_{m,i})$ 
for $0 \le i < n-m$.  Hence we have 
$\IC(\CY_{m+1},\CL_{m+1,i})|_{\CY_m} = \IC(\CY_m, \CL_{m,i})$.
It follows that each factor 
$\IC(\CY_m, \CL_{m,i})[-2i]$ in (3.6.3) for $0 \le i < n-m$ is absorbed to 
$\IC(\CY_{m+1}, \CL_{m+1,i})[-2i]$ in (3.6.4).  
This shows that (3.6.3) holds also for $m+1$. Hence (3.6.3) is proved
and the proposition follows.
\end{proof}
   
%%%%
%%%%
\bigskip

\section{ Intersection cohomology on $G^{\io\th} \times V$} 
\para{4.1.}
We keep the notation in Section 3. In view of Proposition 3.6,
we define a semisimple perverse sheaf $K_{T,\CE}$ on $\CX$ by 
\begin{align*}
\tag{4.1.1}
K_{T,\CE} = \bigoplus_{0 \le m \le n}
       \bigoplus_{\r \in \CA_{\Bm,\CE}\wg}\wt V_{\r} 
         \otimes \IC(\CX_m, \CL_{\r})[d_m],  
\end{align*}
which is the $DGM$-extension of $\psi_*\a_0^*\CE[d_n]$.
We call $IC(\CX_m, \CL_{\r})[\dim \CX_m]$ extended by 0 to $\CX$, for various 
$\r \in \CA_{\Bm, \CE}$, character sheaves on $\CX = G^{\io\th} \times V$.  
We consider a diagram

\begin{equation*}
\begin{CD}
T^{\io\th} @<\a << \wt\CX @>\pi >> \CX,
\end{CD}
\end{equation*}
where $\a : \wt\CX \to T^{\io\th}$ is defined by 
$\a(x,v, gB^{\th}) = p(g\iv xg)$.
For a tame local system $\CE$ on $T^{\io\th}$, we consider 
a complex $\pi_*\a^*\CE$ on $\CX$.
The following result is an analogue of Theorem 1.16.

\begin{thm}  %%% Theorem 4.2
$\pi_*\a^*\CE[\dim \CX] \simeq K_{T,\CE}$ as perverse sheaves on $\CX$.
\end{thm}

\para{4.3.}
The theorem will be proved in 4.9 after some preliminaries.
For each $m \le n$, put $\CX_m^0 = \CX_m \backslash \CX_{m-1}$.
The set $\CX_m^0$ is described as follows; 
First consider the unipotent part $\CX_{m,\unip}^0$ of $\CX_m^0$.
Since $\CX_{m,\unip}$ is the closure of $\CO_{\Bla}$ with 
$\Bla = ((m), (n-m))$ (see Proposition 2.7), we see by using (1.7.3) 
that 
\par\medskip\noindent
(4.3.1) \ $\CX^0_{m,\unip}$ is the union of $\CO_{\Bmu}$
such that $\Bmu = (\mu^{(1)}, \mu^{(2)})$ with $|\mu^{(1)}_1| = m$, 
for the partition $\mu^{(1)} : \mu^{(1)}_1 \ge \mu^{(1)}_2 \ge \cdots$.
In particular, $(x,v)$ is contained in $\CX_m^0$ if and only if 
$\dim {\Bk}[x]v = m$, where $\Bk[x]v$ denotes the subspace of $V$ spanned by 
$v, xv, x^2v, \cdots$.
\par\medskip
More generally, consider $(x,v) \in  G^{\io\th} \times V$. 
Let $x = su$ be the Jordan decomposition of $x \in G^{\io\th}$ and 
consider the decomposition $V = V_1 \oplus \cdots \oplus V_t$
into eigenspaces of $s$. 
Then $Z_G(s) \simeq GL_{2n_1} \times \cdots \times GL_{2n_t}$ with 
$\dim V_i = 2n_i$.  Put $G_i = GL_{2n_i}$ for each $i$.
Then $Z_G(s)$ is $\th$-stable, and $\th$ stabilizes each factor so that
$Z_H(s) \simeq G_1^{\th} \times \cdots \times G_t^{\th}$ with 
$G_i^{\th} \simeq Sp(V_i)$.
Let $v = v_1 + \cdots + v_t$ be the decomposition of $v \in V$ 
with $v_i \in V_i$.
Let $u_i$ be the restriction of $u$ on $V_i$. 
Then $(u_i, v_i) \in G_i^{\io\th} \times V_i$.  
We denote by $\CX^{G_i,0}_{m_i, \unip}$ the subvariety of 
$G^{\io\th}_{i,\unip} \times V_i$ defined in a similar way 
as $\CX^0_{m,\unip}$ for $G^{\io\th}\uni \times V$.
Then we have 
\par\medskip\noindent
(4.3.2) \ $(x,v)$ is contained in $\CX^0_m$ if and only if 
$(u_i,v_i) \in \CX^{G_i,0}_{m_i, \unip}$ with $\sum_{i=1}^t m_i = m$. 
\par\medskip
For each $(x,v) \in G^{\io\th}\times V$, 
let $(u_i,v_i)$ be defined as above.  We define 
a subspace $W = W(x,v)$ of $V$ by 
$W = \Bk[u_1]v_1 \oplus \cdots \oplus \Bk[u_t]v_t$.
Thus $W$ is an isotropic, $x$-stable  subspace of $V$ containing $v$.
Then (4.3.2) can be rewritten as 
\begin{equation*}
\tag{4.3.3}
\CX_m^0 = \{ (x,v) \in G^{\io\th} \times V \mid  \dim W(x,v) = m\}.
\end{equation*}
Recall the map $\pi: \wt\CX \to \CX$.
For an integer $0 \le m \le n$, we define 
a locally closed subvariety $\wt\CX^+_{m}$ of $\wt\CX$
by $\wt\CX^+_{m} = \pi\iv(\CX_{m}^0)$.
Then $\CY_{m}^0$ is open dense in $\CX_{m}^0$ and $\wt\CY^+_{m}$
is an open subset of $\wt\CX^+_{m}$. 
For $(x,v, gB^{\th}) \in \wt\CX^+_m$, we define its level 
$I \subset [1,n]$ as 
follows;  assume that $(x,v) \in B^{\io\th} \times M_n$ and that
$x = su$, the Jordan decomposition of $g$  with $s \in T^{\io\th}$. 
Then $M_n$ is $s$-stable, and is decomposed as
$M_n = M_{n_1}^{[1]} \oplus \cdots \oplus M_{n_t}^{[t]}$, where 
$M_{n_i}^{[i]} = M_n \cap V_i$ is a maximal isotropic subspace of $V_i$.  
Note that 
$M_n = \lp e_1, \dots, e_n \rp$. 
Since $\{ e_i\}$ are weight vectors for $T$, $M^{[i]}_{n_i}$ 
has a basis 
$e_{j_1}, e_{j_2} \cdots$, with $j_1 < j_2 < \cdots < j_{n_i}$. 
Let $(u_i, v_i)$ be as above, and assume that $\dim \Bk[u_i]v_i = m_i$.
Let us define a subset $I_i$ of $\{ j_1, \dots, j_{n_i}\}$ by choosing 
first $m_i$ numbers.  Hence $|I_i| = m_i$.  We define 
$I = \coprod_i I_i$.  Since $(x,v) \in \CX_m^0$, we have $|I| = m$.
Note that the attachment $(x,v) \mapsto I$ depends only on the 
$B^{\th}$-conjugacy class of $(x,v)$.  Thus we have a well-defined map
$(x,v, gB^{\th}) \mapsto I$.
We define a subvariety $\wt\CX_I$ of $\wt\CX^+_{m}$ by 
\begin{equation*}
\wt\CX_I = \{ (x,v, gB^{\th}) \in \wt\CX^+_{m} 
   \mid (x,v, gB^{\th}) \mapsto I\}.
\end{equation*}
We show the following lemma.
\begin{lem}  %%%  Lemma 4.4
$\wt\CX^+_{m}$ is decomposed as 
\begin{equation*}
\wt\CX^+_{m} = \coprod_{\substack{ I \subset [1,n] \\
                           |I| = m}}\wt\CX_I, 
\end{equation*}
where $\wt\CX_I$ is an irreducible component of $\wt\CX^+_{m}$
for each $I$.  
\end{lem}

\begin{proof}
It is clear from the definition that $\wt\CX^+_{m}$ is a disjoint 
union of various $\wt\CX_I$, and that 
$\wt\CX_I$ contains $\wt\CY_I$.  One can check that 
$\wt\CY_I$ is open dense in $\wt\CX_I$.  
Since $\wt\CY^+_{m} = \coprod_I \wt\CY_I$, 
and $\wt\CY^+_{m}$ is open dense in $\wt\CX^+_{m}$, 
$\wt\CX^+_{m} = \bigcup_{I}\ol{\wt\CY_I}$ gives a decomposition 
into irreducible components, where $\ol{\wt\CY_I}$ is the closure of 
$\wt\CY_I$ in $\wt\CX^+_{m}$. Hence in order to prove the lemma, 
it is enough to show 
that $\wt\CX_I$ is closed in $\wt\CX_m^+$ for each $I$. 
But one can check that the closure $Z_I$ of $\wt\CX_I$ in $\wt\CX$
is contained in the set $\wt\CX_I \cup \bigcup_{I'}\wt\CX_{I'}$, where
$I'$ runs over the subsets of $[1,n]$ such that $|I'| < m$. 
Hence $\wt\CX_I = Z_I \cap \wt\CX^+_{m}$ is closed in $\wt\CX^+_{m}$.
\end{proof}

\para{4.5.}
For a fixed $m$, we define a variety $\CG_{m}$ by 
\begin{equation*}
\begin{split}
\CG_{m} = \{ (x,&v, W) \mid (x,v) \in G^{\io\th} \times V, 
    W \text{ : isotropic subspace of } V, \\  
       &\dim W = m, x(W) =W, v \in W \}.
\end{split}
\end{equation*}
We define a map $\pi_m : \wt\CX^+_m \to G^{\io\th} \times V$ by 
the restriction of $\pi$.
Then $\pi_m$ can be decomposed as 
\begin{equation*}
\begin{CD}
\pi_m : \wt\CX^+_{m} @>\vf' >> \CG_{m} @> \vf''>> G^{\io\th} \times V
\end{CD}
\end{equation*}
where $\vf': (x,v, gB^{\th}) \mapsto (x,v, W(x,v) )$, 
$\vf'': (x,v, W) \mapsto (x,v)$.
Let us consider  the spaces $W_0 = M_{m}$ 
and $\ol V_0 = W_0^{\bot}/W_0$.  We put $G_1 = GL(W_0)$, 
$G_2 = GL(\ol V_0)$. Then $\ol V_0$ has a natural symplectic structure,
and $G_2$ is identified with a $\th$-stable subgroup of $G$.  Put 
$H_0 = G_1 \times G_2^{\th}$.  We define a variety
\begin{equation*}
\begin{split}
\CH_{m} = \{ (x,v, &W, \f_1, \f_2) 
     \mid (x,v,W) \in \CG_{m}, \\ 
           &\f_1 : W \isom W_0, \f_2 : {W}^{\bot}/W \isom \ol V_0
                               \text{ (symplectic isom.) } \},
\end{split}
\end{equation*}
and morphisms 
\begin{align*}
q : &\CH_{m} \to \CG_{m}, \quad (x,v, W, \f_1, \f_2) 
   \mapsto (x,v, W), \\
\s : &\CH_{m} \to G_1 \times G_2^{\io\th}, 
     \quad (x,v,W, \f_1,\f_2) \mapsto 
               (\f_1(x|_{W})\f_1\iv, 
                     \f_2(x|_{{W}^{\bot}/W})\f_2\iv).
\end{align*}
Then $H \times H_0$ acts on $\CH_m$ by 
\begin{equation*}
(g,(h_1, h_2)) : (x,v,W,\f_1, \f_2) \mapsto 
    (gxg\iv, gv,g(W),h_1\f_1g\iv, h_2\f_2g\iv)
\end{equation*} 
for $g \in H, (h_1, h_2) \in H_0$.
Moreover, $\s$ is $H \times H_0$-equivariant with respect to 
the natural action of $H_0$ and the trivial action of $H$ on 
$G_1 \times G_2^{\io\th}$.
We have
\par\medskip\noindent
(4.5.1) \ The map $q$ is a principal bundle with fibre isomorphic to
$H_0$.
\par\medskip\noindent
(4.5.2) \ The map $\s$ is a locally trivial fibration 
with smooth fibre of dimension $\dim H$. 
\par\medskip 
In fact, (4.5.1) is clear.  We show (4.5.2).  
It is clear from the definition that $\s$ is a locally trivial
fibration. 
For a fixed 
$(x',x'') \in G_1 \times G_2^{\io\th}$,  the fibre
$\s\iv(x',x'')$ is determined by the following process; 
\par
(i) choose an isotropic subspace $W$ of $V$ such that $\dim W = m$,
\par
(ii) for each $W$, choose an isomorphism $\f_1: W \to W_0$ and 
a symplectic isomorphism $\f_2: W^{\perp}/W \to \ol V_0$,
\par
(iii) choose $x \in G^{\io\th}$ such that 
$\f_1(x|_W)\f_1\iv = x'$, $\f_2(x|_{W^{\perp}/W})\f_2\iv = x''$,
\par
(iv) choose $v \in W$. 
\par
Let $P$ be the stabilizer of the flag $(W_0 \subset W_0^{\perp})$ in 
$G$.  Then $P$ is $\th$-stable, and is decomposed as $P = LU_P$,
where $L$ is a $\th$-stable Levi subgroup of $P$ containing $T$
and $U_P$ is the unipotent radical of $P$.
For (i), such $W$ are parametrized by $H/P^{\th}$. 
For (ii), they are parametrized by $G_1 \times G_2^{\th}$. 
For (iii), $x$ should be contained in $P^{\io\th}$, but 
$x', x''$ determines the part corresponding to $L^{\io\th}$.
Hence the choice of $x$ is parametrized by $U_P^{\io\th}$.
Finally, for (iv), $v$ is any element in $W$. 
It follows that each fibre $\s\iv(x',x'')$ is smooth 
with the same dimension as $H$. Hence (4.5.2) follows. 
\par
Let $B_1$ is a Borel subgroup of $G_1$ which is the stabilizer 
of the flag $(M_k)_{0 \le k \le m}$ in $G_1$, and $B_2$ a 
$\th$-stable Borel subgroup of $G_2$ which is the stabilizer of the flag 
$(M_{m+1}/M_m \subset \cdots \subset M_m^{\perp}/M_m)$ in $G_2$.   
Put 
\begin{align*}
\wt G_1 &= \{ (x, gB_1) \in G_1 \times G_1/B_1 \mid g\iv xg \in B_1 \}, \\
\wt G_2^{\io\th} &= \{ (x, gB^{\th}_2) \in G_2^{\io\th} 
             \times G_2^{\th}/B_2^{\th} \mid g\iv xg \in B_2^{\io\th} \},  
\end{align*}
and define maps $\pi^1: \wt G_1 \to G_1$, 
$\pi^2: \wt G^{\io\th} \to G^{\io\th}$ by first projections. 
Thus $\pi^2$ is the map $\pi$ given in 1.8 with respect to $G_2^{\io\th}$, 
and $\pi^1$ is the corresponding map for $G_1$.
We define a variety 
\begin{equation*}
\begin{split}
\CZ^+_{m} = \{ (x,v,&gB^{\th},\f_1, \f_2) \mid 
              (x,v, gB^{\th}) \in \wt\CX^+_{m}, \\ 
           &\f_1: W(x,v) \isom W_0, 
              \f_2: W(x,v)^{\bot}/W(x,v) \isom \ol V_0  \},
\end{split}
\end{equation*}
and a map $\wt q : \CZ^+_m \to \CX^+_m$ by a natural projection. 
We define a map $\wt \s : \CZ^+_m \to \wt G_1 \times \wt G_2^{\io\th}$ 
as follows; take $(x,v,gB^{\th}, \f_1, \f_2) \in \CZ_m^+$.
By the construction of $W(x,v)$, $W(x,v)$ is contained 
in $g(M_n)$, and $(g(M_k))_{0 \le k \le n}$ is an $x$-stable isotropic 
flag in $V$.  
Then $(W(x,v) \cap g(M_k))$ gives an $x$-stable 
flag $(W_k)_{0\le k \le m}$ in $W(x,v)$, and 
$(W(x,v) + g(M_k)/W(x,v))$ gives an $x$-stable
isotropic flag $(\ol V_{\ell})_{0 \le \ell \le n-m}$ in 
$W(x,v)^{\perp}/W(x,v)$.  We denote by $g_1B_1g_1\iv$ the stabilizer 
of the flag $(\f_1(W_k))$ in $G_1$,  and by $g_2B^{\th}_2g_2\iv$ the 
stabilizer of the isotropic 
flag $(\f_2(\ol V_{\ell}))$ in $\ol V_0$.  
The map $\wt \s$ is defined as 
\begin{equation*}
(x,v, gB^{\th}, \f_1, \f_2) \mapsto 
     (\f_1(x|_{W(x,v)})\f_1\iv, g_1B_1), 
     (\f_2(x|_{W(x,v)^{\perp}/W(x,v)})\f_2\iv, g_2B_2^{\th}).
\end{equation*}

We consider a diagram 
\begin{equation*}
\tag{4.5.3}
\begin{CD}
T_1 \times T_2^{\io\th}  @<f <<  T^{\io\th}  @>\id >>  T^{\io\th}  \\
   @A\a^1 \times \a^2 AA               @AA\wt\a A     @AA\a A   \\
\wt G_1 \times \wt G_2^{\io\th} @< \wt \s << 
     \CZ^+_{m}  @> \wt q>>  \wt\CX^+_{m} \\
  @V \pi^1 \times \pi^2 VV             @VV\wt\vf' V     @VV \vf'V  \\
 G_1 \times G_2^{\io\th} @< \s << \CH_{m} 
        @> q >> \CG_{m}  \\
       @.        @.  @VV \vf'' V  \\
       @.       @.   G^{\io\th} \times V,  
\end{CD}
\end{equation*}
where the maps $\wt\vf', \wt\a$ are defined naturally. 
Note that $T^{\io\th}$ can be written as 
$T^{\io\th} \simeq T_1 \times T_2^{\io\th}$, 
where $T_1$ is a maximal torus of $G_1$ (the diagonal group),
and $T_2$ is a $\th$-stable maximal torus of $G_2$ 
(the diagonal group). We fix an isomorphism 
$f: T^{\io\th} \to T_1 \times T_2^{\io\th}$.
The maps $\a^1 : \wt G_1 \to T_1$, 
$\a^2: \wt G_2^{\io\th} \to T_2^{\io\th}$ are defined as in 1.15. 

\para{4.6.}
Let $\CE$ be a tame local system on $T^{\io\th}$.
Under the isomorphism $f : T^{\io\th} \to T_1 \times T_2^{\io\th}$,
$\CE$ can be written as $\CE \simeq \CE_1\boxtimes \CE_2$, 
where $\CE_1$ (resp. $\CE_2$) is a tame local system on $T_1$
(resp. $T_2^{\io\th}$). 
Then we have $\CW_{\Bm, \CE} \simeq (S_m)_{\CE_1} \times (S_{n-m})_{\CE_2}$, 
where $(S_m)_{\CE_1}$ is the stabilizer of $\CE_1$ in 
$S_m \simeq \CW_1 = N_{G_1}(T_1)/T_1$, and 
$(S_{n-m})_{\CE_2}$ is the stabilizer of $\CE_2$ in 
$S_{n-m} \simeq  \CW_2  = 
N_{G^{\th}_2}(T_2^{\io\th})/Z_{G^{\th}_2}(T_2^{\io\th})$.
Let $\CA_{\CE_2}$ be the algebra defined in 1.15, by replacing 
$G^{\io\th}$ by $G_2^{\io\th}$, and let $\CA_{\CE_1}$ be the 
corresponding algebra for $G_1$.  Then we have 
$\CA_{\Bm, \CE} \simeq \CA_{\CE_1}\otimes \CA_{\CE_2}$.
\par
For $\r_1 \in \CA_{\CE_1}\wg, \r_2 \in \CA_{\CE_2}\wg$, 
we consider the complexes
\begin{equation*}
K _1 = \IC(G_1, \CL_{\r_1})[\dim G_1], \qquad
K_2 = \IC(G^{\io\th}_2, \CL_{\r_2})[\dim G_2^{\io\th}].
\end{equation*}
Then $K_1\boxtimes K_2$ is an $H_0$-equivariant  simple perverse sheaf on 
$G_1 \times G_2^{\io\th}$, and so $\s^*(K_1\boxtimes K_2)[2n^2 + n]$
is an $H_0$-equivariant simple perverse sheaf on $\CH_m$ by 
(4.5.2). 
Since $q$ is a principal bundle with group $H_0$ by (4.5.1), there exists 
a unique simple perverse sheaf $A_{\r}$ on $\CG_m$ such that
\begin{equation*}
q^*(A_{\r})[m^2 + 2(n-m)^2 + (n-m)] \simeq \s^*(K_1\boxtimes K_2)[2n^2 + n],
\end{equation*}
where $\r = \r_1\otimes \r_2 \in \CA_{\Bm,\CE}\wg$ under the isomorphism 
$\CA_{\Bm,\CE} \simeq \CA_{\CE_1} \otimes \CA_{\CE_2}$.
\par
Note that the map $\vf' : \wt\CX^+_m \to \CG_m$ is not surjective.
The image $\vf'(\wt\CX_m^+)$ is described as follows; for each $m' \le m$, 
we define a subvariety $\CG_{m,m'}$ of $\CG_m$ by 
\begin{equation*}
\CG_{m,m'} = \{ (x,v, W) \in \CG_m \mid \dim W(x,v) = m'\}.
\end{equation*}
Then $\CG_m = \coprod_{0 \le m' \le m}\CG_{m,m'}$ gives a partition of
$\CG_m$, and the closure of $\CG_{m,m'}$ is a union of 
$\CG_{m,m''}$ with $m'' \le m'$.  Put $\CG_m^0 = \CG_{m,m}$.  Then 
$\CG^0_{m}$ is an open dense subset of $\CG_m$.
It is clear from the definition of $\vf'$ that $\vf'(\wt\CX^+_m)$ coincides 
with $\CG_m^0$.  
We denote by $A_{\r}^0$ 
the restriction of $A_{\r}$ on $\CG_m^0$, which gives rise to a simple 
perverse sheaf if the support of $A_{\r}$ has a non-trivial intersection 
with $\CG_m^0$.  
\par
Let $\CL_{\r}$ be the local system on $\CY^0_{m}$ as in (4.1.1).  
Since  
$\CY^0_{m}$ is an open dense subset of $\CX^0_m$, one can 
consider the intersection cohomology $\IC(\CX^0_m, \CL_{\r})$.
We show the following lemma.

\begin{lem} %%% Lemma 4.7
Under the notation as above, we have 
\begin{equation*}
\vf''_*A^0_{\r} \simeq \IC(\CX^0_m, \CL_{\r})[\dim \CX_m].
\end{equation*}
\end{lem}

\begin{proof}
Since $\CG_m^0 = \{ (x,v, W) \in \CG_m \mid  W = W(x,v)\}$,   
$\vf''$ gives an isomorphism $\CG_m^0$ onto $\CX_m^0$ by (4.3.3). 
Hence $\vf''_*A_{\r}^0$ is a simple perverse sheaf if $A^0_{\r}$ is so.
Put $K = \vf''_*A_{\r}^0$ and $d = \dim \CX^0_m = \dim \CX_m$. 
In order to prove the lemma, it is enough to show that 
\begin{equation*}
\tag{4.7.1}
\CH^{-d}K|_{\CY^0_m} \simeq \CL_{\r}.
\end{equation*} 

We consider the following commutative diagram

\begin{equation*}
\tag{4.7.2}
\begin{CD}
T_1 \times T_2^{\io\th} @< f << T^{\io\th} @>\id >> T^{\io\th} \\
@A\a_0^1 \times \a_0^2 AA         @AA \wt\a_0A            @AA \a_0 A    \\
\wt G_{1,\rg} \times \wt G_{2,\rg}^{\io\th} @< \wt \s_0<<  
        \wt\CZ^0_m  @>\wt q_0>>  \wt \CY^0_m  \\
 @V\xi^1 \times \xi^2 VV          @VV\wt\xi_0 V        @VV \xi_0 V  \\
(G_1/T_1 \times T_{1,\rg}) \times 
(G^{\th}_2/Z_{G^{\th}_2}(T_2^{\io\th}) \times T^{\io\th}_{2,\rg})
    @< \wh \s_0<< \wh\CZ^0_m   @> \wh q_0>>  \wh \CY^0_m \\
@V\e^1 \times \e^2 VV          @VV\wt\e_0 V        @VV \e_0 V   \\
G_{1,\rg} \times G^{\io\th}_{2,\rg} @< \s_0<<  
                 \CH_{m,\rg} @> q_0>>   \CG_{m,\rg}, \\
@.             @.             @VV\vf''_0V  \\
    @.               @.        \CY^0_m     
\end{CD}
\end{equation*}
where $\wt\CY_m^0 = \wt\CY_I, \wh\CY_m^0 = \wh\CY_I$ for $I = [1,m]$
(cf. 3.1, 3.4), and 
\begin{align*}
\wt G_{1,\rg} &=(\pi^1)\iv(G_{1,\rg}), \\ 
\wt G_{2,\rg}^{\io\th} &= (\pi^2)\iv(G_{2,\rg}), \\
\CG_{m,\rg} &= (\vf'')\iv(\CY^0_m), \\
\CH_{m,\rg} &= q\iv(\CG_{m,\rg}), \\
\wt\CZ^0_m &= \wt q\iv(\wt\CY^0_m),  \\
\wh\CZ^0_m &= \{ (g*(t,v), \f_1, \f_2) \mid 
       g*(t,v) \in \wh\CY^0_m, \\
   &\phantom{*******}
  \f_1 : g(W_0) \isom W_0, \f_2 : g(W_0)^{\bot}/g(W_0) \isom \ol V_0\},
\end{align*}
and the maps $\wt q_0, q_0, \wt \s_0, \s_0$ are defined as the
 restrictions of the corresponding maps $\wt q, q, \wt \s,\s$.
The map $\xi_0 = \xi_I$ for $I = [1,m]$.
Note that in this case, $W(t,v) = M_m$ for 
$(t,v) \in T^{\io\th}\reg \times M_m^0$.  We define a map 
$\e_0: \wh\CY^0_m \to \CG_{m,\rg}$ by 
$(g*(t,v)) \mapsto (gtg\iv, g(v), g(M_m))$, and the maps
$\wt\xi_0, \wt\e_0$ are defined accordingly.
The maps in the first column are defined by applying the discussion 
in 1.8 to the case 
$\p^2: \wt G^{\io\th}_{2,\rg} \to G^{\io\th}_{2,\rg}$, and to
the $GL$ case $\p^1 : \wt G_{1,\rg} \to G_{1,\rg}$.
The map $\wh \s_0$ is defined as follows; for 
$(g*(t,v), \f_1, \f_2) \in \wh\CZ^0_{m}$, $gtg\iv$ stabilizes 
$g(W_0)$, and $t'_1 = \f_1(gtg\iv|_{g(W_0)})\f_1\iv \in G_{1,\rg}$,
$t'_2 = \f_2(gtg\iv|_{g(W_0)^{\bot}/g(W_0)})\f_2\iv \in G^{\io\th}_{2,\rg}$.
We have $Z_{G_1}(t'_1) = g_1T_1g_1\iv$ with $g_1 \in G_1$, and 
$Z_{G^{\th}_2}(t'_2) = g_2Z_{G^{\th}}(T^{\io\th}_2)g_2\iv$ 
with $g_2 \in G_2^{\th}$.  Then $\wh \s_0(g*(t,v), \f_1, \f_2)
= ((g_1T_1, t_1), (g_2Z_{G_2^{\th}}(T_2^{\io\th}), t_2))$, 
where $(t_1, t_2) = f(t) \in T_1 \times T_2^{\io\th}$.
\par
It follows from the commutative 
diagram that 
\begin{equation*}
\tag{4.7.3}
\wt \s_0^*((\a_0^1)^*\CE_1\boxtimes (\a_0^2)^*\CE_2) \simeq 
           \wt q_0^*\a_0^*\CE.
\end{equation*}
One can check that the squares in the middle row and the lower row
are all cartesian squares.   Since $\p^1 = \e^1\circ \xi^1$, 
$\p^2 = \e^2\circ\xi^2$, we see that 
\begin{equation*}
\tag{4.7.4}
\s_0^*(\p^1_*(\a_0^1)^*\CE_1\boxtimes \p^2_*(\a_0^2)^*\CE_2)
  \simeq q_0^*(\e_0\circ\xi_0)_*\a_0^*\CE.
\end{equation*}
Now $\p^2_*(\a_0^2)^*\CE_2$ is decomposed as in (1.15.1), 
and similarly for $\p^1_*(\a_0^1)^*\CE_1$.  On the other hand, 
since 
\begin{equation*}
\CG_{m,\rg} = \{ (x,v, W) \mid (x,v) \in \CY^0_m, W = W(x,v)\}
\simeq \CY^0_m,
\end{equation*}
$\vf_0''$ is an isomorphism.  It follows, by (3.4.3) and (3.5.1) 
that 
\begin{equation*}
(\e_0\circ\xi_0)_*\a_0^*\CE \simeq \bigoplus_{\r \in \CA_{\Bm,\CE}\wg}
      H^{\bullet}(\BP_1^{n-m})\otimes \r \otimes \CL_{\r} 
\end{equation*}
under the identification $\CG_{m,\rg} \simeq \CY_m^0$.
Since the decomposition by the Galois
 covering is compatible with $\s_0$ and $q_0$ by (4.7.2), we have
$\s_0^*(\CL_{{\r}_1} \boxtimes \CL_{{\r}_2}) \simeq q_0^* \CL_{\r}$.
This implies that 
$A^0_{\r}|_{\CG_{m,\rg}} \simeq A_{\r}|_{\CG_{m,\rg}} \simeq \CL_{\r}[d]$, 
and so 
\begin{equation*}
K|_{\CY^0_m} \simeq (\vf_0'')_*(A^0_{\r}|_{\CG_{m,\rg}}) 
 \simeq \CL_{\r}[d].
\end{equation*}
This proves (4.7.1), and the lemma follows.
\end{proof}

By using Lemma 4.7, we show the following.

\begin{prop}   %%%%  Prop. 4.8.
Under the notation in Lemma 4.7, 
$(\pi_m)_*\a^*\CE|_{\wt\CX_m^+}$ is decomposed as  
\begin{equation*}
   \bigoplus_{\r \in \CA\wg_{\Bm,\CE}}
    H^{\bullet}(\BP_1^{n-m})\otimes \wt V_{\r} \otimes \IC(\CX^0_m, \CL_{\r}),
\end{equation*}
where $\wt V_{\r} = \CA_{\CE}\otimes_{\CA_{\Bm,\CE}}\r$
is considered as a vector space, ignoring the $\CA_{\CE}$-action. 
\end{prop}

\begin{proof}
First we recall the isomorphism given in 1.17.2 in the language of flags. 
Assume that $s \in T^{\io\th}$.
Let $V = V_1 \oplus \cdots \oplus V_t$ be the decomposition of 
$V$ to the eigenspaces of $s$.  
We assume that 
$V_i = \lp e_k, f_k \mid k \in J_i\rp$ with $J_i \subset [1,n]$. 
We regard $J_i$ an ordered set under the natural ordering.
Let $\CF^{\th}(V)$ be the set of isotropic flags on $V$, and  
$\CF^{\th}_x(V)$  the subset of isotropic flags stable by 
$x \in G^{\io\th}$.   
Let $\vG = W_{H,s}\backslash W_H$ be as in 1.17. We choose 
a set of distinguished representatives $\{ w_{\g} \mid \g \in \vG \}$  
for $\vG$, where $w_{\g}$ is a permutation on the set $[1,n]$
such that $w_{\g}\iv$ is order preserving on each subset $J_i$.
Then $\CF^{\th}_s(V)$ is decomposed as 
$\CF^{\th}_s(V) = \coprod_{\g \in \vG}\CF^{\g}(V)$, where 
$\CF^{\g}(V)$ is the set of $Z_H(s)$ conjugates of the isotropic flag 
$(\lp e_{w_{\g}(1)}\rp \subset \lp e_{w_{\g}(1)}, e_{w_{\g}(2)}
 \rp \subset \cdots
\subset \lp e_{w_{\g}(1)}, \dots, e_{w_{\g}(n)}\rp)$.  
If we replace $s$ by $x = su$ 
(Jordan decomposition), 
then the above isomorphism restricts 
to the isomorphism between $\CF_x^{\th}(V)$ and 
$\coprod_{\g \in \vG}\CF^{\g}_u(V)$. 
\par
We fix a subset $I \subset [1,n]$ such that $|I| = m$.
For $x = su \in G^{\io\th}$ such that $s \in T^{\io\th}$, we 
consider the decomposition of $V$ as above.
Let $U = W(x,v)$.  Recall that $m_i = \dim \Bk[u_i]v_i$ under the 
notation in 4.3.
We define a subset $\wt\CX_I^{\spadesuit}$ of $\wt\CX_m^+$  
as the set of all $H$-conjugates of 
$(x,v,w_{\g}B^{\th}) \in \wt\CX_m^+$ such that 
$|w_{\g}\iv(J_i) \cap I| = m_i$ for each $i$, for all the possible
choice of $(x,v)$ with $s \in T^{\io\th}$ and 
$w_{\g} \in W_{s,H}\backslash W_H$.
Then $\wt\CX_I^{\spadesuit}$ is a closed subset of $\wt\CX_m^+$
containing $\wt\CX_I$.  (In fact, $\wt\CX_I$ coincides with 
the $H$-conjugates of $(x,v, w_{\g}B^{\th}) \in \wt\CX^+_m$ such that
$w_{\g}\iv(J_i) \cap I$ consists of the first $m_i$ letters.)   
We have  $\CG_m^0 = \vf'(\wt\CX^{\spadesuit}_I)$ for any $I$. 
We define 
a subset $\CZ^{\spadesuit}_I$ of 
$\CZ^+_m$ by $\CZ^{\spadesuit}_I = \wt q\iv(\wt\CX^{\spadesuit}_I)$, 
and a subset $\CH^0_m$ of $\CH_m$ by $\CH_m^0 = q\iv(\CG_m^0)$.
We have a commutative diagram

\begin{equation*}
\tag{4.8.1}
\begin{CD}
\wt G_1 \times \wt G_2^{\io\th} @< \wt \s_I << 
     \CZ^{\spadesuit}_I  @> \wt q_I >>  \wt\CX^{\spadesuit}_I \\
  @V \pi^1 \times \pi^2 VV             @VV\wt\vf_I' V     @VV \vf'_I V  \\
 G_1 \times G_2^{\io\th} @< \s_1 << \CH^0_m 
        @> q_1 >> \CG_m^0, 
\end{CD}
\end{equation*}
where $\vf_I', \wt\vf_I', q_1, \s_1, \wt q_I, \wt\s_I$ are restrictions of 
$\vf', \wt\vf', q, \s, \wt q, \wt\s$, respectively.
The square in the right hand side is cartesian since 
the corresponding square in (4.5.3) is cartesian.  
\par
Let $\ZC_m$ be the fibre product of $\wt G_1 \times \wt G_2^{\io\th}$ 
with $\CH_m^0$ over $G_1 \times G_2^{\io\th}$, 
and $c_I: \CZ_I^{\spadesuit} \to \ZC_m$ be the natural map obtained 
from the diagram (4.8.1).   
Recall that an $\a$-partition of a variety $X$, in the sense of [DLP],
is a partition 
$X = \coprod_{\a}X_{\a}$ into locally closed subvarieties $X_{\a}$
such that there exists a filtration $X = X_1 \supset X_2 \supset \cdots$
by closed subvarieties $X_i$ where $X_i \,\backslash\, X_{i+1}$ coincides 
with some $X_{\a}$. 
We show that 
\par\medskip\noindent
(4.8.2) \  There exists an $\a$-partition of $\ZC_m$ 
by locally closed subvarieties $\ZC_{\a}$ such that the restriction of 
$c_I$ on $c_I\iv(\ZC_{\a})$ is a locally 
trivial fibration over $\ZC_{\a}$. 
\par\medskip
First we consider a similar problem for the unipotent part  
$\ZC_{m, \unip}$ of $\ZC_m$, i.e., $\ZC_{m,\unip}$ is 
the fibre product of $\wt G_1 \times \wt G_2^{\io\th}$ with 
$\CH_{m,\unip}^0$ over $G_1 \times G_2^{\io\th}$, 
where $\CH_{m,\unip}^0 = \{ (x,v, W(x,v), \f_1, \f_2) \in \CH_m^0
         \mid (x,v) \in \CX\uni\}$. 
In this case, $W(x,v) = \Bk[x]v$, $I = [1,m]$ and so  
$x$ is a regular unipotent element on $U = W(x,v)$, which determines 
a unique Borel subgroup of $GL(U)$ containing $x$. Thus $\ZC_{m,\unip}$
is isomorphic to 
\begin{equation*}
\{ (x,v, \f_1, \f_2, ( \ol U_i)) \mid
        (\ol U_i) : \text{ $\ol x$-stable isotropic flag in }
                     \ol U = U^{\perp}/U \},   
\end{equation*}   
where $(x,v,U,\f_1, \f_2) \in \CH_{m,\unip}$ and $\ol x$ is 
a transformation on $\ol U$ induced from $x$.
We define $\ZC_{\a}$ as a subset of $\ZC_{m,\unip}$ consisting of 
$(x,v, \f_1, \f_2, (\ol U_i))$ such that $(x,v)$ is contained in 
a fixed $H$-orbit in $\CX\uni$, and that 
the restriction of $x$ on $\ol U_m/\ol U_i$ has a fixed Jordan type 
for each $i$.  Then those $\ZC_{\a}$ give a partition of $\ZC_{m,\unip}$ 
into finitely many locally closed pieces, and in fact, they give 
an $\a$-partition of $\ZC_{m,\unip}$.  This follows from the fact 
that a similar partition of $\CF^{\th}_x(\ol U)$ gives 
an $\a$-partition.    
For each $z = (x,v, \f_1, \f_2, (\ol U_i)) \in \ZC_{\a}$, 
the fibre $c_I\iv(z)$ is given as
\begin{equation*}
c_I\iv(z) \simeq \{ (V_j) \in \CF_x^{\th}(V) \mid V_n \subset U^{\perp},  
    (V_j + U /U) \to  (\ol U_i) \}.
\end{equation*}
where $(V_j + U/U) \to (F'_i)$ means that 
the flag of $\ol U$ obtained from 
$V_1 +U/U \subseteq V_2 +U/U \subseteq \cdots$ coincides with 
$(\ol U_i)$.
Then $c_I\iv(\CZ_{\a}) \to \CZ_{\a}$ gives a locally trivial 
fibration.  This can be checked by using the following property 
of $\CF^{\th}_x(V)$; 
for $x \in G^{\io\th}\uni$, consider the map 
$f : \CF^{\th}_x(V) \to \BP(\Ker x)$, $(V_i) \mapsto V_1$.
Then  for $(V_i), (V_i') \in \CF_x(V)$, $V_1$ and $V_2$ are 
in the same $Z_H(x)$-orbit in $\BP(\Ker x)$ 
if and only if the restriction of $x$ on 
$V_n/V_1$ and on $V_n'/V'_1$ have the same Jordan type. 
\par
We return to the general setting.
Take $z' = (x,v, w_{\g}B^{\th}, \f_1, \f_2) \in \CZ_I^{\spadesuit}$ 
with $x = su, s \in T^{\io\th}$, and put 
$s_1 = \f_1(s) \in T_1, s_2 = \f_2(s) \in T_2^{\io\th}$, where 
$T_1, T_2$ are maximal tori in $G_1, G_2$. Then the sets 
$\vG_1, \vG_2$ are defined with respect to $s_1, s_2$, as $\vG$ 
with respect to $s$. 
Then by the map $\wt\s_I$, $z'$ determines an $s_1$-stable flag in 
$G_1$ and $s_2$-stable isotropic flag in $G_2$, hence determines 
$\g_1 \in \vG_1, \g_2 \in \vG_2$.
On the other hand, take
$(x,v, U, \f_1, \f_2) \in \CH_m^0$ and $(x_1, g_1B_1) \in \wt G_1$, 
$(x_2, g_2B_2^{\th}) \in \wt G_2^{\io\th}$ so that it gives 
$z \in \ZC_m$. Assume that the semisimple 
parts of $x_1, x_2$ are $s_1, s_2$ with $s_1 \in T_1, s_2 \in T_2^{\io\th}$, 
and $\vG_1, \vG_2$ are as before.  Then 
$\g_1 \in \vG_1$ is determined from 
$(x_1, g_1B_1)$, and $\g_2 \in \vG_2$ is determined from 
$(x_2, g_2B_2^{\th})$. We show that 
$\g \in \vG$ is reconstructed from $(\g_1, \g_2)$;    
We assume that $x = su$ with $s \in T^{\io\th}$, and let 
$V_i$ be as before.  Let $U = W(x,v)$. For each $i$, let 
$J_i' = J_i \cap I$.  Put $J' = \coprod_iJ_i'$ and $J''$ 
the complement of $J'$ in $[1,n]$.
Then $w_{\g_1}$ (resp. $w_{\g_2}$) 
is regarded as a permutation on the set $J'$ (resp. $J''$) 
under a suitable identification.
We define a permutation $w$ by 
\begin{equation*}
    w\iv (j) = \begin{cases}
              w_{\g_1}\iv(j) &\quad\text{ if } j \in J', \\
              w_{\g_2}\iv(j) &\quad\text{ if } j \in J''.
           \end{cases}
\end{equation*}
We take $\g = W_{H,s}w \in W_{H,s}\backslash W_H$.
Then $w_{\g}$ satisfies the condition that 
$|w_{\g}\iv(J_i) \cap I| = m_i$ for each $i$.
In this way, we have a bijection between $(\g_1, \g_2)$ 
arising from $\wt G_1 \times \wt G^{\io\th}$ and 
$\g$ arising from $\CZ_I^{\spadesuit}$.
Then for 
such $z \in \ZC_m$, $c_I\iv(z)$ is contained $\CF_u^{\g}(V)$.  
\par
We define a subset $\ZC_{\a}$ of $\ZC_m$ as the set 
of elements given by 
$(x_1, g_1B_1) \in \wt G_1, (x_2, g_2B_2^{\th}) \in \wt G_2^{\io\th}$ 
and $(x,v, U, \f_1, \f_2) \in \CH_m^0$, which satisfies 
the following condition; let $x = su$ be the Jordan decomposition,
then $s$ has a fixed eigenspace decomposition 
$V \simeq V_1 \oplus\cdots\oplus V_t$ up to $H$-conjugate, 
$(u,v)$ is contained in a fixed 
$Z_H(s)$-orbit in $Z_H(s)\uni^{\io\th} \times V$. 
Moreover, let $(\ol U)_{0 \le i \le m}$ 
be the isotropic flag in $\ol U$  obtained 
from $g_2B^{\th}_2$ and $\f_2$, then the restriction of 
$u$ on each $\ol U_m/\ol U_i$ has a fixed Jordan type for each $i$.  
Those $\ZC_{\a}$ give an $\a$-partition of $\ZC_m$, and one can 
show, as in the case of 
$\ZC_{m,\unip}$, that $c_I\iv(\ZC_{\a}) \to \ZC_{\a}$ gives a locally 
trivial fibration for each $\a$.  Thus (4.8.2) is proved.

\par
We define a complex $K_{T_2,\CE_2}$ on $G_2^{\io\th}$ 
as in (1.15.2), by replacing
$G$ by $G_2$.  A similar construction works for the group case $G_1$. 
They are given as follows; 
\begin{align*}
\tag{4.8.3}
K_{T_1,\CE_1} &= \bigoplus_{\r_1 \in \CA_{\CE_1}\wg}
                \r_1 \otimes \IC(G_1, \CL_{\r_1})[\dim G_1], \\
K_{T_2, \CE_2} &= \bigoplus_{\r_2 \in \CA_{\CE_2}\wg}
           H^{\bullet}(\BP_1^{n-m})\otimes \r_2 \otimes 
                   \IC(G_2^{\io\th}, \CL_{\r_2})[\dim G_2^{\io\th}].
\end{align*}
Then we have the following result, the first
one is due to Lusztig [L2], 
and the second one follows from Theorem 1.16 due
to Grojonowski. 
\begin{align*}
\pi^1_*(\a^1)^*\CE_1[\dim G_1] &\simeq  K_{T_1, \CE_1}, \\ 
\pi^2_*(\a^2)^*\CE_2[\dim G^{\io\th}_2] &\simeq K_{T_2, \CE_2}. 
\end{align*}
Since $\ZC_m$ is a fibre product of $\wt G_1 \times \wt G_2^{\io\th}$
with $\CH_m^0$ over $G_1 \times G_2^{\io\th}$, we have

\begin{equation*}
\s_1^*(K_{T_1,\CE_1}\boxtimes K_{T_2, \CE_2}) 
      \simeq (\vf_1)_*\a^*_{12}(\CE_1\boxtimes \CE_2)[d],
\end{equation*} 
where $d = \dim G_1 + \dim G_2^{\io\th}$, 
$\vf_1: \ZC_m \to \CH_m^0$ is the natural map, and 
$\a_{12} : \ZC_m \to T_1 \times T_2^{\io\th}$ is the map 
obtained from $\a^1 \times \a^2$. 
For each piece $\ZC_{\a}$, we define a complex 
$K_{\a}$ on $\CH_m^0$ by 
\begin{equation*}
K_{\a} = 
   (\vf_{\a})_!(\a_{12}^*(\CE_1\boxtimes\CE_2))|_{\ZC_{\a}}[d],
\end{equation*}
where $\vf_{\a}$ is the restriction to $\ZC_{\a}$ 
of the natural map $\ZC_m \to \CH_m^0$.  Since $\ZC_m = \coprod \ZC_{\a}$ 
is an $\a$-partition, we have a filtration of $\ZC_m$ 
by closed subsets such as 
$\cdots \supset \ZC_{\ol \a}  \supset  \ZC_{\ol \a'} \supset \cdots$ with 
$\ZC_{\ol \a} \ \backslash \ \ZC_{\ol\a'} = \ZC_{\a}$. 
The complex $K_{\ol \a}$ is defined similarly.
(In what follows, we use the general theory of perverse sheaves. 
The basic reference is [BBD].
In particular, we consider the derived category 
$\DD X$ of (constructible) $\Ql$-sheaves on the variety $X$, 
and the category 
of perverse sheaves $\CM X$ as a full subcategory of $\DD X$.  
We denote by ${}^pH^i(X) \in \CM X$ the $i$-th perverse cohomology 
complex on $X$.)  
Now 
$(K_{\a}, K_{\ol\a}, K_{\ol\a'})$ becomes a distinguished 
triangle in $\DD \CH_m^0$.  
Thus we obtain a perverse cohomology 
exact sequence
\begin{equation*}
\begin{CD}
@>>> {}^pH^i(K_{\a}) @>>> {}^pH^i(K_{\ol\a}) @>>> {}^pH^i(K_{\ol\a'})
      @>>> 
\end{CD}
\end{equation*}
It follows that 
\par\medskip\noindent
(4.8.4) \ Any simple constituent appearing in ${}^pH^i(K_{\a})$
appears as a direct summand in the semisimple complex 
$\s_1^*(K_{T_1,\CE_1}\boxtimes K_{T_2,\CE_2})$.  
\par\medskip
On the other hand, 
put $\wt K_{\a} = (\wt\vf_{\a})_!\wt\a_{12}^*(\CE_1\boxtimes\CE_2)[d]$,
where 
$\wt\vf_{\a}$ is the restriction of $\wt\vf_I'$ on $c_I\iv(\ZC_{\a})$, 
and 
$\wt\a_{12} = (\a^1 \times \a^2)\circ \wt\s_I$.  
Since $c_I\iv(\ZC_{\a}) \to \ZC_{\a}$ is a locally trivial fibration, 
we see that 
\begin{align*}
\wt K_{\a} &= (\wt\vf_{\a})_!c_I^*\a_{12}^*(\CE_1 \boxtimes\CE_2)[d]  \\
           &= (\vf_{\a})_!(c_I)_!c_I^*\a_{12}^*(\CE_1\boxtimes\CE_2)[d] \\
           &= H^{\bullet}_{\a}\otimes K_{\a}, 
\end{align*}
where $H^{\bullet}_{\a} = (c_{\a})_!\Ql$ for 
$c_{\a}: c_I\iv(z) \to \{ z\}$ for a point $z \in \ZC_{\a}$, hence 
is a complex of vector spaces.  
Let us define $K_I = (\wt\vf'_I)_!\wt\a_{12}^*(\CE_1\boxtimes\CE_2)[d]$.
Since 
$\CZ_I^{\spadesuit} = \coprod c_I\iv(\ZC_{\a})$ gives an $\a$-partition 
of $\CZ_I^{\spadesuit}$, by a similar argument as before, we see that 
any simple summand, up to shift,  of the semisimple complex $K_I$ 
appears as a simple constituent of the perverse cohomology 
${}^pH^i(\wt K_{\a})$ for some $\a$, hence we have 
\par\medskip\noindent
(4.8.5) \ Any simple summand of the semisimple complex 
$K_I$ appears as a simple constituent  
of ${}^pH^i(K_{\a})$ for some $i$ and some $\a$.
\par\medskip  
By using the cartesian square in the right hand side of 
the diagram (4.8.1) together with (4.8.3), (4.8.4) and (4.8.5), 
we see that 
any simple summand (up to shift) of the semisimple complex 
$(\vf'_I)_*\a^*\CE$ 
is contained  in the set $\{ A_{\r}^0 \mid \r \in \CA_{\Bm, \CE}\wg \}$.
Since $\wt\CX_I$ is a connected component of $\wt\CX_I^{\spadesuit}$, 
$(\vf'|_{\wt\CX_I})_*\a^*\CE$ is a direct summand of 
$(\vf'_I)_*\a^*\CE$. Since 
$\wt\CX_m^+ = \coprod_I\wt\CX_I$, we see that 
\par\medskip\noindent
(4.8.6)
Any simple summand (up to shift) of the semisimple complex   
$(\vf')_*\a^*\CE$ is contained in the set 
\begin{equation*}
\bigcup_{\substack{I \subset [1,n] \\
           |I| = m}}\{ A_{\r}^0 \mid \r \in \CA_{\Bm,\CE}\wg \}.
\end{equation*}
\par\medskip
By applying $\vf''_*$ on (4.8.6), and by applying Lemma 4.7,
we have
\par\medskip\noindent
(4.8.7) Any simple summand (up to shift) of the semisimple complex 
$(\pi_m)_*\a^*\CE|_{\wt\CX_m^+}$ is contained in the set
$\{ \IC(\CX_m^0,\CL_{\r}) \mid \r \in \CA_{\Bm,\CE}\wg\}$.
\par\medskip
(4.8.7) implies, in particular, that any simple summand of 
$K = (\pi_m)_*\a^*\CE|_{\wt\CX_m^+}$ has its support $\CX_m^0$. 
Since the restriction of $K$ on $\CY_m^0$ coincides with 
$K_0 = (\psi_m)_*\a_0^*\CE|_{\wt\CY_m^+}$,
the decomposition of $K$ into simple summands is determined by 
the decomposition of $K_0$.  Hence the proposition follows from 
(3.5.2). 
\end{proof}

\para{4.9.}
We are now ready to prove Theorem 4.2. 
For each $0 \le m \le n$, let $\ol\pi_m$ be 
the restriction of $\pi$ on $\pi\iv(\CX_m)$.
We show that the direct image complex $(\ol\pi_m)_*\a^*\CE|_{\pi\iv(\CX_m)}$ 
on $\CX_m$ is decomposed as follows;
\begin{equation*}
\tag{4.9.1}
\begin{split}
(\ol\pi_m)_*(\a^*\CE|_{\pi\iv(\CX_m)})[d_m] 
        \simeq &\bigoplus_{\r \in \CA_{\Bm,\CE}\wg}
   H^{\bullet}(\BP_1^{n-m})\otimes  \wt V_{\r} \otimes 
    \IC(\CX_m, \CL_{\r})[d_m] \\
     &\oplus \bigoplus_{0 \le m' < m}\bigoplus_{\r \in \CA_{\Bm',\CE}}
   \wt V_{\r} \otimes \IC(\CX_{m'}, \CL_{\r})[d_{m'} - 2(n-m)]. 
\end{split}
\end{equation*}
Note that the theorem will follow from (4.9.1) by applying it to the 
case where $m = n$.
We show (4.9.1) by induction on $m$. In the discussion below, 
we use the same symbol $\a^*\CE$ to denote its restriction to 
$\pi\iv(\CX_m), \pi\iv(\CY_m)$, etc. for saving the notation.
Since $\ol\pi_m$ is a proper map, by the decomposition theorem, 
$(\ol\pi_m)_*\a^*\CE$ can be written as a direct sum of 
complexes of the form $A[h]$, where $A$ is a simple perverse sheaf. 
Here the restriction of $(\ol\pi_m)_*\a^*\CE$ on $\CY_m$ 
coincides with $(\ol\psi_m)_*\a^*\CE$, 
and we know by (3.6.1) that it can be expressed by a similar formula as 
(4.9.1), by replacing $\CX_m, \CX_{m'}$ by $\CY_m, \CY_{m'}$. 
Hence in order to prove (4.9.1), it is enough to show that 
$\supp A \cap \CY_m \ne \emptyset$ for any direct summand $A[h]$
of $(\ol\pi_m)_*\a^*\CE$.  
We have $\CX_m = \CX_m^0 \cup \CX_{m-1}$ with  $\CX_m^0$ open. 
The restriction of $(\ol\pi_m)_*\a^*\CE$ on 
$\CX^0_m$ (resp. $\CX_{m-1}$) coincides with $(\pi_m)_*\a^*\CE$
(resp. $(\ol\pi_{m-1})_*\a^*\CE$).  
Let $A[h]$ be a direct summand of $(\ol\pi_m)_*\a^*\CE$. If 
$\supp A \cap \CX^0_m \ne \emptyset$, then 
$A|_{\CX_m^0}$ is a direct summand (up to shift) of 
$(\pi_m)_*\a^*\CE$.   By comparing Proposition 4.8 and
(3.5.2), we see that 
$\supp A \cap \CY_m \ne \emptyset$.  
If $\supp A \cap \CX_m^0 = \emptyset$, then 
$\supp A \subset \CX_{m-1}$, and $A|_{\CX_{m-1}}$ (up to shift) 
is a direct summand of $(\ol\pi_{m-1})_*\a^*\CE$.     
By induction hypothesis, we see again that 
$\supp A \cap \CY_m \ne \emptyset$.
This proves (4.9.1) and so proves the theorem.
%%%%
%%%%
\par\bigskip

\section{Springer correspondence}

\para{5.1.}
The Springer correspondence is a bijective correspondence 
between the set of $H$-orbits in $G^{\io\th}\uni \times V$ and the set of 
irreducible representations (up to isomorphism) of Weyl group
$W_n$, which is an analogue of the Springer correspondence in 
the case of reductive groups. Kato proved in [Ka1] the Springer 
correspondence for the exotic nilpotent cone $\Fg^{-\th}\nil \times V$
over $\BC$.  Here we give an alternate proof for 
$\CX\uni = G^{\io\th}\uni \times V$ over an algebraically closed 
field $\Bk$ of odd characteristic, based on Theorem 4.2.
\par 
In the setup of Section 4, we consider the special case where 
$\CE  = \Ql$, the constant sheaf on $T^{\io\th}$.  
In this case under the notation of 3.5, 
$\CA_{\CE}$ is isomorphic to the group algebra 
$\Ql[\CW]$, and $\wt\CA_{\CE}$ is isomorphic to $\Ql[W_n]$.  Hence   
by Theorem 4.2, $\pi_*\Ql$ is equipped with the action of $W_n$.
It follows that for each $(x,v) \in \CX$, 
$H^i(\pi\iv(x,v), \Ql)$  turns out to be a $W_n$-module.  In the case 
where $(x,v) = (1,0)$, $\pi\iv(x,v) = H/B^{\th}$.  Hence 
$W_n$ acts on $H^i(H/B^{\th}, \Ql)$, which we call 
the exotic action of $W_n$.  On the other hand, $W_n$ acts 
on $H^i(H/B^{\th}, \Ql)$ as the original Springer 
action of Weyl groups, which we call the classical 
action.  We have the following lemma.

\begin{lem}  %%%  Lemma 5.2
The exotic action and the classical action of $W_n$ on 
$H^i(H/B^{\th}, \Ql)$ coincide with each other.
\end{lem}

\begin{proof}
We prove the lemma following the argument analogous to the
classical case (see e.g. [S1, Proposition 5.4]).
Let $\pi' : \wt G^{\io\th} \to G^{\io\th}$, 
$\psi': {\pi'}\iv(G^{\io\th}\reg) \to  G^{\io\th}\reg$ 
be the maps defined in 1.8.  Under the embedding 
$G^{\io\th} \simeq G^{\io\th} \times \{ 0\} \hookrightarrow 
G^{\io\th} \times V$,  $\pi_*\Ql|_{G^{\io\th}} \simeq \pi'_*\Ql$.
and $\psi_*\Ql|_{G^{\io\th}\reg} \simeq \psi'_*\Ql$.
Moreover, $\psi_*\Ql|_{G^{\io\th}\reg}$ 
(resp.  $\pi_*\Ql|_{G^{\io\th}}$) coincides with $(\psi_m)_*\Ql$ 
(resp. $(\pi_m)_*\Ql$) defined in 3.4 (resp. 4.5) for $m = 0$.
Thus the arguments in Section 3 and 4 are applied for this situation.  
In particular, $\pi'_*\Ql$  and $\psi'_*\Ql$ have $W_n$-actions 
as a restriction of $\psi_*\Ql$ and $\pi_*\Ql$.  
Thus, forgetting the vector space part $V$, 
we have only to consider $\pi'_*\Ql$. We can also consider 
the Lie algebra counter part of $\pi'_*\Ql$.  Let 
$\pi'': \wt\Fg^{-\th} \to \Fg^{-\th}$ be the map analogous to 
$\pi'$, where $\wt\Fg^{-\th} = H \times ^{B^{\th}}\Fb^{-\th}$
with $\Fb = \Lie B$.  Then one can define an action of $W_n$ 
on $\pi''_*\Ql$ by a similar construction.  Consider the map 
$\log : G^{\io\th} \to \Fg^{-\th}$. Then $\log$ maps $B^{\io\th}$
to $\Fb^{-\th}$, and it induces a map 
$\wt\log : \wt G^{\io\th} \to \wt \Fg^{-\th}$, which is a base change 
of $\log$ via $\pi''$.  Thus we have $\log^*\pi'_*\Ql \simeq \pi''_*\Ql$.
This isomorphism is $W_n$-equivariant.  Hence in order to prove the
lemma, we may consider the case $\pi'': \wt\Fg^{-\th} \to \Fg^{-\th}$.
For saving the notation, we use the same symbol 
in the Lie algebra case as in 1.8, 
such as $\pi: \wt\Fg^{-\th} \to \Fg^{-\th}$.
\par
Let $\Ft = \Lie T$, and $\Fg^{-\th}\reg, \Ft^{-\th}\reg$ be as 
in Proposition 1.14. Put $\Fb^{-\th}\reg = \Fb \cap \Fg^{-\th}\reg$.  
Similar to (1.8.1), we have  a diagram
\begin{equation*}
\begin{CD}
\p : H \times^{B^{\th}}\Fb^{-\th}\reg 
    \simeq H \times^{B^{\th}\cap Z_H(T^{\io\th})}\Ft^{-\th}\reg
      @>\xi>> 
    H \times^{Z_H(T^{\io\th})}\Ft^{-\th}\reg @>\e>> \Fg^{-\th}\reg,
\end{CD}
\end{equation*}
where $\xi$ is a locally trivial fibration with fibre isomorphic to
$\BP^n_1$, and $\e$ is a finite Galois covering with Galois group $S_n$.
Let $\CO$ be an $H$-orbit in $\Fg^{-\th}\reg$, and put 
$\wt\CO = \xi\iv\e\iv(\CO)$.
We have the following diagram

\begin{equation*}
\tag{5.2.1}
\begin{CD}
\p\iv(\CO) @<\sim <<  \wt\CO @>\xi>> \e\iv(\CO) @>\e>> \CO  \\
  @V q VV         @VV \wt p V     @VV p V               \\
H/B^{\th}  @<\g << H/B^{\th}\cap Z_H(T^{\io\th}) @>\xi'>> H/Z_H(T^{\io\th}), \\
\end{CD}
\end{equation*}
where $q, \wt p, p$ are natural projections.
Now $\p_*\Ql$ is equipped with $W_n$-action.
Hence for any subvariety $X$ of $\Fg^{-\th}\reg$,  
$W_n$ acts on 
$H^i(\p\iv(X), \Ql) \simeq \BH^i(X, \p_*\Ql)$, which is 
the exotic action of $W_n$. 
Moreover, for the inclusion $X \hookrightarrow Y$, the induced 
map $j^*: H^i(\p\iv(Y), \Ql) \to H^i(\p\iv(X),\Ql)$ 
arising from $j : \psi\iv(X) \hra \psi\iv(Y)$ 
is $W_n$-equivariant.
Here we fix an action of $W_n$ on $H^i(\psi\iv(\CO),\Ql)$ defined 
as above. 
First we note that 
\par\medskip\noindent 
(5.2.2) \ The map $q^* : H^i(H/B^{\th}, \Ql) \to H^i(\p\iv(\CO), \Ql)$
is $W_n$-equivariant with respect to the exotic action of 
$W_n$ on $H^i(H/B^{\th},\Ql)$.
\par\medskip
In fact, let $j$ be the inclusion map 
$\p\iv(\CO) \hookrightarrow \wt\Fg^{-\th}$, 
and $q_1 : \wt\Fg^{-\th} \to H/B^{\th}$ be the projection. 
Then $q = q_1\circ j$.  Since 
$j^*: H^i(\wt\Fg^{-\th}, \Ql) \to H^i(\psi\iv(\CO),\Ql)$ 
is $W_n$-equivariant from the above remark, it is 
enough to show that 
$q_1^*: H^i(H/B^{\th},\Ql) \to H^i(\wt\Fg^{-\th},\Ql)$ is $W_n$-equivariant.  
But since $\wt\Fg^{-\th}$ is a vector bundle over $H/B^{\th}$, 
we have
$H^i(\wt\Fg^{-\th}, \Ql) \simeq H^i(H/B^{\th}, \Ql)$.  This 
implies that $q_1^* = (j_1^*)\iv$, where $j_1$ is the natural inclusion 
$H/B^{\th} \hookrightarrow \wt\Fg^{-\th}$ so that $q_1\circ j_1 = \id$.
Since $j_1^*$ is $W_n$-equivariant, $q_1^*$ is $W_n$-equivariant. 
Hence (5.2.2) holds. 
\par
Let $\CW = N_H(T^{\io\th})/Z_H(T^{\io\th})$ be as before.  Then 
$\CW$ acts on $H/Z_H(T^{\io\th})$, and $p$ is $\CW$-equivariant.
It follows that 
$p^* : H^i(H/Z_H(T^{\io\th}), \Ql) \to H^i(\e\iv(\CO), \Ql)$ is 
$\CW$-equivariant.  
$B^{\th} \cap Z_H(T^{\io\th}) \simeq B_2 \times \cdots \times B_2$ 
as in 1.8, and the action of $S_n$ (as a permutation group of the set 
$\{ e_1, \dots, e_n\}$ and $\{f_1, \dots, f_n\}$) on $H$ 
leaves $B^{\th} \cap Z_H(T^{\io\th})$ invariant. 
Hence $\CW \simeq S_n$ acts on $H/B^{\th} \cap Z_H(T^{\io\th})$, 
and $\xi'$ is $\CW$-equivariant.
Since $\xi, \xi'$ are both $\BP_1^n$-bundles, we see that the map
$\wt p^* : H^i(H/B^{\th} \cap Z_H(T^{\io\th}), \Ql) 
\to H^i(\wt\CO, \Ql)$ is $\CW$-equivariant. 
The group $\CW$ acts on $B^{\th}\cap Z_H(T^{\io\th})$ as a permutation 
of factors $B_2$, and its restriction on $T^{\th}$ coincides with 
the permutation action of $S_n$ on the diagonal groups.  
Since $\g$ is an affine space bundle, we have 
$H^i(H/B^{\th}, \Ql) \simeq H^i(H/T^{\th}, \Ql) \simeq 
  H^i(H/B^{\th} \cap Z_H(T^{\io\th}),\Ql)$, and $\g^*$ becomes 
$\CW$-equivariant, where $\CW \simeq S_n$ acts naturally 
on $H^i(H/T^{\th}, \Ql)$, which is the classical action on 
$H^i(H/B^{\th}, \Ql)$.  This shows that 
\par\medskip\noindent
(5.2.3) \ $q^*$ is $S_n$-equivariant with respect to the 
classical action of $S_n$ on $H^i(H/B^{\th}, \Ql)$.  
\par\medskip
On the other hand, $\xi_*\Ql \simeq H^{\bullet}(\BP_1^n)\otimes \Ql$
is equipped with the action of $(\BZ/2\BZ)^n$.   Thus
$\BH^i(\e\iv(\CO), \xi_*\Ql) \simeq H^i(\wt\CO, \Ql)$ has a structure of 
$(\BZ/2\BZ)^n$-module.  Similarly, 
$H^i(H/B^{\th} \cap Z_H(T^{\io\th}), \Ql)$ has a structure of 
$(\BZ/2\BZ)^n$-module, and the map $\wt p^*$ is
 $(\BZ/2\BZ)^n$-equivariant.
Let $\xi_1 : H/T^{\th} \to H/Z_H(T^{\io\th})$.  Then 
$H/B^{\th} \cap Z_H(T^{\io\th})$ is a vector bundle over $H/T^{\th}$, 
and 
\begin{equation*}
H^i(H/B^{\th} \cap Z_H(T^{\io\th}),\Ql) \simeq H^i(H/T^{\th}, \Ql) \simeq
   \BH^i(H/Z_H(T^{\io\th}), (\xi_1)_*\Ql).
\end{equation*}
Here the $(\BZ/2\BZ)^n$-action on 
$H^i(H/B^{\th}\cap Z_H(T^{\io\th}),\Ql)$
comes from the $(\BZ/2\BZ)^n$-action on $(\xi_1)_*\Ql$, which is
 described as follows: 
$Z_H(T^{\io\th}) \simeq SL_2 \times \cdots \times  SL_2$, 
and $Z_H(T^{\io\th})/T^{\th} \simeq (SL_2/T_2)^n$, where $T_2$ is a
maximal torus of $SL_2$.  Hence $(\BZ/2\BZ)^n$ acts on 
$Z_H(T^{\io\th})/T^{\th}$ as its Weyl group, and acts on  
$H/T^{\th} \simeq H\times^{Z_H(T^{\io\th})}(Z_H(T^{\io\th})/T^{\th})$  
naturally, which gives the $(\BZ/2\BZ)^n$-action on $(\xi_1)_*\Ql$.
But this action of $(\BZ/2\BZ)^n$ on $H/T^{\th}$ is nothing but 
the action on $H/T^{\th}$ as a subgroup of 
$W_n = N_H(T^{\th})/T^{\th}$.   It follows that 
\par\medskip\noindent
(5.2.4) \ $q^*$ is $(\BZ/2\BZ)^n$-equivariant with respect to the 
classical action of $(\BZ/2\BZ)^n$ on $H^i(H/B^{\th}, \Ql)$.
\par\medskip
By (5.2.3) and (5.2.4), we see that $q^*$ is $W_n$-equivariant with 
respect to the classical action of $W_n$ on $H^i(H/B^{\th}, \Ql)$. 
In view of (5.2.2), in order to prove the lemma, it is enough to 
show that $q^*$ is injective.   But since 
$\CO \simeq H/Z_H(T^{\io\th})$, and $\e$ is a finite covering with 
group $S_n$, we see that $\e\iv(\CO) = \coprod_{w \in S_n}X_w$ 
with $X_w \simeq H/Z_H(T^{\io\th})$.
Since $\xi$ is a $\BP_1^n$-bundle, $\wt\CO$ is written as 
$\wt\CO = \coprod_{w \in S_n}\wt X_w$, where $\wt X_w$ is a
$\BP^n_1$-bundle over $X_w$, which is 
isomorphic to $H/B^{\th} \cap Z_H(T^{\io\th})$.
This implies the injectivity of the map $\wt p^*$. 
The injectivity of $q^*$ follows from this.  Thus 
the lemma is proved.
\end{proof}

\para{5.3.}
Note that $\dim \CX\uni = 2\nu_H$ by Proposition 2.7.  
Then  $\pi_1$ is semismall by Lemma 2.5 (ii), and so  
$K_1 = (\pi_1)_*\Ql[\dim \CX\uni]$ 
is a semisimple perverse sheaf. As $\pi_1$ 
is $H$-equivariant, $K_1$ is $H$-equivariant.  Since 
the number of $H$-orbits in $\CX\uni$ is finite, and the isotropy 
group of each orbit is connected by Lemma 2.3 (ii), 
an $H$-equivariant  
simple perverse sheaf is of the form 
$A_{\Bmu} = \IC(\ol\CO_{\Bmu}, \Ql)[\dim \CO_{\Bmu}]$
for an $H$-orbit $\CO_{\Bmu}$ in $\CX\uni$.  
It follows that 
\begin{equation*}
\tag{5.3.1}
K_1 \simeq \bigoplus_{\Bmu \in \CP_{n,2}}
   \r_{\Bmu}\otimes A_{\Bmu},  
\end{equation*}
where $\r_{\Bmu} = \Hom (A_{\Bmu}, K_1)$ is a $W_n$-module.
\par
Since $\CE = \Ql$, we have
$\CA_{\Bm, \CE} \simeq \Ql[S_m \times S_{n-m}]$.
Then $\r \in  \CA_{\Bm,\CE}\wg$
is given by $\r_1 \otimes \r_2$, where $\r_1 \in S_m\wg$ 
(resp. $\r_2 \in S_{n-m}\wg$)   
corresponding to  $\mu^{(1)} \in \CP_m$ 
(resp. $\mu^{(2)} \in \CP_{n-m}$).
The induced representation $\wt V_{\r}$ (defined in 3.5)
 gives rise to 
an irreducible representation of $W_n$ , which we denote by 
$\wt V_{\Bmu}$ for $\Bmu = (\mu^{(1)},\mu^{(2)}) \in \CP_{n,2}$.
In this case we denote the local system $\CL_{\r}$ 
on $\CY^0_m$ by $\CL_{\Bmu}$.  For a given $\Bmu \in \CP_{n,2}$, 
put $m(\Bmu) = |\mu^{(1)}|$.   Hence $\CL_{\Bmu}$ is a local system 
on $\CY^0_{m(\Bmu)}$.

The following result is a reformulation of Kato's result 
([Ka1, Theorem 8.3]).
%%%%%
\begin{thm} [Springer correspondence] %%%% Th. 5.4
\begin{enumerate}
\item
For each $\Bmu \in \CP_{n,2}$, $\r_{\Bmu}$ is an irreducible 
$W_n$-module.  The map $\CO_{\Bmu} \mapsto \r_{\Bmu}$ gives a bijective
correspondence between the set of $H$-orbits in $\CX\uni$ and the set of 
irreducible representations (up to isomorphism)  of $W_n$. 
\item
Let $\CO_{\Bmu^{\bullet}}$ be the $H$-orbit, for a certain 
$\Bmu^{\bullet} \in \CP_{n,2}$,  corresponding to 
$\wt V_{\Bmu}$ under the above correspondence.  
Then we have
\begin{equation*}
\IC(\CX_{m(\Bmu)}, \CL_{\Bmu})|_{\CX\uni} \simeq 
            \IC(\ol\CO_{\Bmu^{\bullet}},\Ql)[a],
\end{equation*}
where $a = \dim \CO_{\Bmu^{\bullet}} - \dim \CX\uni - 
               \dim \CX_{m(\Bmu)} + \dim \CX$. 
\end{enumerate}
\end{thm}

\begin{proof}
We show (i).
We have an algebra homomorphism
\begin{equation*}
\tag{5.4.1}
\begin{CD}
\Ql[W_n] @>\a>> \End K_1 @>\b>> \End H^*(H/B^{\th}, \Ql).
\end{CD}
\end{equation*}
In order to prove (i), it is enough to show 
that $\a$ is an isomorphism $\Ql[W_n] \isom \End K_1$.
Via $\b\circ\a$, $W_n$ acts on $H^*(H/B^{\th}, \Ql)$.  By Lemma 5.2, 
this action is the classical action of $W_n$. Since 
$W_n$ acts on $H^*(H/B^{\th}, \Ql)$ as the regular representation, 
we see that $\b\circ \a$ is injective.
Hence $\a$ is also injective, and we have $|W_n| \le \dim\End K_1$.
On the other hand, (5.3.1) implies that    
\begin{equation*}
\tag{5.4.2}
H^i(\pi_1\iv(x,v), \Ql) \simeq 
\bigoplus_{\Bmu \in \CP_{n,2}}\r_{\Bmu}\otimes 
\CH^{i- \dim \CX\uni + \dim \CO_{\Bmu}}_{(x,v)}\IC(\ol\CO_{\Bmu},\Ql).
\end{equation*}
Take $(x,v) \in \CO_{\Bmu}$  and assume that $\r_{\Bmu} \ne \{ 0\}$.
Put $d_{\Bmu} = (\dim \CX\uni - \dim \CO_{\Bmu})/2 
= \nu_H - \dim \CO_{\Bmu}/2$. 
Then (5.4.2) implies that 
\begin{equation*}
\tag{5.4.3}
H^{2d_{\Bmu}}(\pi_1\iv(x,v), \Ql) \simeq \r_{\Bmu} \otimes 
  \CH^0_{(x,v)}\IC(\ol\CO_{\Bmu}, \Ql) \simeq \r_{\Bmu}.
\end{equation*}
By applying Lemma 2.5 (ii), we have $\dim \pi_1\iv(x,v) = d_{\Bmu}$.
Hence $\dim \r_{\Bmu}$ coincides with the number of irreducible 
components of $\pi_1\iv(x,v)$ with dimension $d_{\Bmu}$.
Thus by Lemma 2.5 (iii), we have 
\begin{equation*}
\dim \End K_1 = \sum_{\Bmu \in \CP_{n,2}}(\dim \r_{\Bmu})^2
  \le \sum_{\Bmu \in \CP_{n,2}}c^2_{\Bmu} \le |W_n|.
\end{equation*}
It follows that $\a$ is an isomorphism, and (i) follows.
\par
Next we show (ii). It follows from Theorem 4.2, we have
\begin{equation*}
\pi_*\Ql[\dim \CX]|_{\CX\uni} \simeq \bigoplus_{\Bmu \in \CP_{n,2}}
   \wt V_{\Bmu}\otimes \IC(\CX_{m(\Bmu)}, \CL_{\Bmu})[d_{m(\Bmu)}]|_{\CX\uni}.
\end{equation*}
On the other hand, by (i) and (5.3.1), we have 
\begin{equation*}
\tag{5.4.4}
(\pi_1)_*\Ql[\dim \CX\uni] \simeq \bigoplus_{\Bmu\in \CP_{n,2}}
    \wt V_{\Bmu}\otimes \IC(\ol\CO_{\Bmu^{\bullet}}, \Ql)
    [\dim \CO_{\Bmu^{\bullet}}].
\end{equation*}
By comparing these two formulas, we obtain (ii). 
Thus the theorem is proved.
\end{proof}

\remarks{5.5.} (i) \ Later in Theorem 7.1, it is proved 
that the correspondence $\Bmu \to \Bmu^{\bullet}$ in Theorem 5.4
is identical, namely $\Bmu^{\bullet} = \Bmu$ for each 
$\Bmu \in \CP_{n,2}$.
\par
(ii) In view of the proof of the theorem, together with Theorem 7.1, 
the Springer correspondence is given in the following form; 
for each $(x,v) \in \CO_{\Bmu}$, the Springer module 
$H^{2d_{\Bmu}}(\pi_1\iv(x,v),\Ql)$ gives rise to an irreducible 
$W_n$-module $\wt V_{\Bmu}$ (the standard $W_n$-module associated to 
the double partition $\Bmu \in \CP_{n,2}$), and the correspondence
$(x,v) \mapsto H^{2d_{\Bmu}}(\pi_1\iv(x,v),\Ql)$ gives a bijective 
correspondence as in the theorem.

\para{5.6.}
Let $P = LU_P$ be a $\th$-stable parabolic subgroup of $G$ containing 
$B$, where $L$ is a $\th$-stable Levi subgroup containing $T$.
Then $L^{\th}$ is isomorphic to 
$GL_{n_1} \times \cdots \times GL_{n_k} \times Sp_{2n'}$ 
with $\sum_i n_i + n' = n$. Accordingly, we have a decomposition
$M_n = \bigoplus_{i \ge 0} V_i$ and 
$M_n' = \bigoplus_{i \ge 0}V'_i$, where 
$V_{i}, V_{i}'$ are $T$-stable subspaces with 
$\dim V_i = \dim V'_i =  n_i$ for $i \ge 1$, 
and $\dim V_0 = \dim V_0' = n'$.
Let $V_L = V_0 \oplus V'_0$ and $
\wt V_L = V_L \oplus V_1\oplus\cdots\oplus V_k$  
so that $V = \wt V_L \oplus V'_L$, where 
$V_L' = V_1' \oplus\cdots\oplus V_k'$.
Here we may assume that $V'_L$ is $U_P$-stable, and 
$\wt V_L$ is identified with $V/V'_L$, on which $L^{\th}$ acts naturally.
We consider the variety 
\begin{equation*}
\tag{5.6.1}
L^{\io\th} \times \wt V_L \simeq (GL^{\io\th}_{2n'} \times V_L) \times 
                          \prod_{i=1}^k(GL_{n_i} \times V_i),
\end{equation*}
on which $L^{\th}$ acts via the conjugation action on 
$L^{\io\th}$ and the natural action on $\wt V_L$. In particular, 
$GL^{\io\th}_{2n'} \times V_L$ is a similar variety as 
$H \times V$ for $n = n'$.
It is known by [AH], [T] that the set of $GL_{n_i}$-orbits on 
$(GL_{n_i})\uni \times V_{n_i}$ is parameterized by the set 
$\CP_{n_i,2}$.
Let $\pi_P : P^{\io\th}\uni\times V \to L^{\io\th}\uni \times \wt V_L$ 
be the map defined by the natural projection 
$P^{\io\th}\uni \to L^{\io\th}\uni$
on the first factor, 
and the projection $\wt V_L\oplus V_L' \to \wt V_L$ on the second factor.
Let $\CO$ be an $H$-orbit in 
$G^{\io\th}\uni \times V$, and  $\CO_L$  an 
$L^{\th}$-orbit in $L^{\io\th}\uni \times V_L$.
Let us consider a variety 
\begin{equation*}
\begin{split}
\CZ = \{ (x,v, &gP^{\th}, g'P^{\th}) \in G^{\io\th}\uni \times V 
             \times H/P^{\th} \times H/P^{\th}  \\
       &\mid (g\iv xg, g\iv v) \in \pi_P\iv(\CO_L),
       ({g'}\iv xg', {g'}\iv v) \in \pi_P\iv(\CO_L) \}.
\end{split}
\end{equation*}
We consider the partition 
$H/P^{\th} \times H/P^{\th} = \coprod X_{\w}$ into $H$-orbits, and 
let $\CZ_{\w} = p\iv(X_{\w})$, where 
$p : \CZ \to H/P^{\th} \times H/P^{\th}$ is the projection onto
the last two factors.
Recall that $\n_H = \dim U^{\th} $.  Here we put 
$\ol \n = \n_{L^{\th}}$, the number of positive roots for 
$L^{\th}$.  Let 
$c = \dim \CO$ and $\ol c = \dim \CO_L$.
Note that $c$ is even by Lemma 2.3, however, $\ol c$ 
is not necessarily even since a $GL_{n_i}$-orbit on 
$(GL_{n_i})\uni \times V_i$ is not so in general
(see [AH, Proposition 2.8]).
The following result, which is a generalization of 
Proposition 2.5,  is an analogue of the dimension formula
due to Lusztig [L2, Proposition 1.2], and can be proved in 
a similar line.  
%%%
\begin{prop}  %%Prop.5.7
\begin{enumerate}
\item  
For $(\ol x, \ol v) \in \CO_L$, we have 
$\dim (\CO \cap \pi_P\iv(\ol x, \ol v)) \le (c - \ol c)/2$.

\item
For $(x,v) \in \CO$, 
\begin{equation*}
 \dim \{ gP^{\th} \in H/P^{\th} \mid  
(g\iv xg, g\iv v) \in \pi_{P}\iv(\CO_L)\} \le 
   (\nu_H - c/2) - (\ol \nu - \ol c/2). 
\end{equation*}

\item
Put $d_0 = 2\n_H - 2\ol \n + \ol c$. Then 
$\dim \CZ_{\w} \le d_0$ for all $\w$.  Hence 
$\dim \CZ  \le d_0$.
\end{enumerate}
\end{prop}

\begin{proof}
The proposition makes sense for the groups of the same type as 
$L$ and we prove it for those groups.   
More precisely, we assume that 
\begin{equation*}
G = G_0 \times \prod_{i = 1}^k(G_i \times G_i), \qquad 
V = \wt V_0 \oplus\bigoplus_{i=1}^k(V_i \oplus V_i),
\end{equation*}  
where
$G_0 = GL(\wt V_0), G_i = GL(V_i)$.  Let $\th$ be an involutive 
automorphism on $G$ such that $\th$ preserves $G_0$ with 
$G_0^{\th} = Sp(\wt V_0)$ and that $\th$ acts on $G_i \times G_i$
permuting the factors.  Hence we have 
\begin{equation*}
G^{\io\th} \times V \simeq (G_0^{\io\th} \times \wt V_0)
                            \times \prod_{i=1}^k(G_i \times V_i)
\end{equation*}
on which $H = G^{\th}$ acts naturally. 
We use a similar notation as before for this general setting. 
For a $\th$-stable Levi subgroup $L$, we define a subspace $\wt V_L$
of $V$ as 
$\wt V_L = \wt V_{0,L} \oplus\bigoplus_{i=1}^kV_{i,L}$, 
where $\wt V_0 = \wt V_{0,L} \oplus V'_{0,L}$ is as in 5.6 for 
$G = GL_{2n}$, and $V_{i,L} = V_i$ for $i \ge 1$. 
In particular, $V = \wt V_L \oplus V'_L$ with $V_L' = V'_{0,L}$. 
We consider $L^{\io\th} \times \wt V_L$ instead of (5.6.1).
\par
Since the assertion
holds in the case where $G = T$, a maximal torus, we may assume that 
the proposition holds for a proper subgroup $L$ of $G$.
Let $W_H$ be the Weyl group of $H$. 
Consider the map $p: \CZ_{\w} \to X_{\w}$.  $X_{\w}$ is an 
$H$-orbit of $(P^{\th}, wP^{\th})$, where $w \in W_H$ 
is a representative of the double coset 
$W_{P^{\th}}\backslash W_H/W_{P^{\th}}$
for a parabolic subgroup $W_{P^{\th}}$ of $W_H$.  
Hence in order to show (iii), it is enough to see that 
\begin{equation*}
\tag{5.7.1}
\begin{split}
\dim\{ (x,v) \in &\pi_P\iv(\CO_L) \mid 
         (\dw\iv x \dw, \dw\iv v) \in \pi_P\iv(\CO_L))\}  \\
   &\le 2\nu_H - 2\ol\nu + \ol c - \dim X_{\w},
\end{split}
\end{equation*}
where $\dw \in H$ is a representative of $w$.
Then an element $x \in P^{\io\th} \cap {}^wP^{\io\th}$ can be written as 
$x = \ell u = \ell' u'$ with $\ell \in L^{\io\th}, \ell' \in {}^wL^{\io\th}$, 
$u \in U^{\io\th}_P, u' \in {}^wU^{\io\th}_P$.
Moreover, there exists $z \in L^{\io\th} \cap {}^wL^{\io\th}$ 
such that $\ell = zy', \ell' = zy$ with 
$y' \in L^{\io\th} \cap {}^wU_P^{\io\th}$, 
$y \in {}^wL^{\io\th} \cap U_P^{\io\th}$. 
Hence (5.7.1) can be rewritten as
\begin{equation*}
\tag{5.7.2}
\begin{split}
\dim&\{ (v, u, u', y, y', z) \in V \times 
   U_P^{\io\th} \times {}^wU_P^{\io\th} \\ 
  &\times 
({}^wL \cap U_P)^{\io\th} \times 
 (L\cap {}^wU_P)^{\io\th} \times (L \cap {}^wL)^{\io\th}  \\
   &\mid y'u = yu', (zy', \ol v) \in \CO_L, 
    (\dw\iv zy\dw, \ol{\dw\iv v}) \in \CO_L\}
     \le 2\nu_H  - 2\ol\nu + \ol c - \dim X_{\w},
 \end{split}
\end{equation*}
where $\ol v \in \wt V_L$ is the image of $v$ under the map 
$V \to \wt V_L$, and similarly 
for $\ol{\dw\iv v} \in \wt V_L$.
For a fixed $(y,y')$, the variety 
$\{ (u, u') \in (U_P \times {}^wU_P)^{\io\th} \mid y'u = yu'\}$
has dimension $\dim (U_P \cap {}^wU_P)^{\io\th}$ if non-empty.
On the other hand, for a fixed $z, y, y', \ol v, \ol{\dw\iv v}$,
\begin{equation*}
\dim \{ v \in V \mid (zy',\ol v) \in \CO_L, (\dw\iv zy\dw, \ol{\dw\iv v}) 
   \in \CO_L\} = \dim( V_L' \cap w(V_L')).
\end{equation*}
We shall compute $\dim (V'_L \cap w(V_L'))$. Since $V_L' = V'_{0,L}$, 
this dimension only depends on the factor $Sp(\wt V_0)$ of $H$. 
So, for a moment, we assume that $G = GL_{2n}$ and $H = Sp_{2n}$. 
Here we choose $w \in W_{P^{\th}}\backslash W_H/W_{P^{\th}}$ as 
a distinguished representative.  Then we have
$\dim (V_L' \cap w(V_L')) = 
\sharp \{ e_i \in V'_L \mid w\iv(e_i) \in V'_L\}$, which we set 
$b_w$.
On the other hand, 
if we denote the set of positive roots of type $C_n$ by 
$\D^+ = \{ \ve_i \pm \ve_j \ (1 \le i < j \le n), 
        \ 2\ve_i \ (1 \le i \le n) \}$ and $\D^+_{P^{\th}}$
the subset of $\D^+$ corresponding to $U_{P^\th}$, we see that 
$\dim (U_P \cap {}^wU_P)^{\th} = 
\dim (U_P \cap {}^wU_P)^{\io\th} + c_w$, where 
\begin{equation*}
c_w = 
\sharp \{ 2\ve_i \in \D^+_{P^\th} 
        \mid  w\iv(2\ve_i) \in \D^+_{P^{\th}} \}.
\end{equation*}
It is easy to see that $c_w = b_w$.  
Since $L \cap {}^wP$ is a parabolic subgroup of $L$ with 
a Levi decomposition 
$L \cap {}^w P = (L \cap {}^w L)(L \cap {}^wU_P)$, and similarly for 
${}^wL \cap P = ({}^w L \cap L)({}^w L \cap U_P)$, 
we see that 
$\dim (P \cap {}^wP)^{\th} = 
2\ol \nu + \dim T^{\th} + \dim (U_P \cap {}^wU_P)^{\th}$.  
It follows that 
\begin{equation*}
\dim (U_P \cap {}^wU_P)^{\io\th} + b_w 
  = 2\nu_H - 2\ol \nu - \dim X_{\w}.
\end{equation*}
Note that this formula holds for the general setting. 
Hence (5.7.2) will follow if we can show that 
\begin{equation*}
\tag{5.7.3}
\begin{split}
\dim \{ (z, y, y', &\ol v) \in 
            (L \cap {}^wL)^{\io\th} \times 
           ({}^wL \cap U_P)^{\io\th} \times 
           (L \cap {}^wU_P)^{\io\th}) 
              \times \wt V_L \\
         &\mid (zy', \ol v) \in \CO_L, 
           (\dw\iv zy\dw, \dw\iv\ol v) \in \CO_L\} \le \ol c.
\end{split}
\end{equation*}
Since ${}^wP$ and $L$ contain a common maximal torus $T$, 
$Q = L \cap {}^wP$ is a $\th$-stable parabolic subgroup of $L$
with Levi decomposition 
$Q = MU_Q$, where $M = L \cap {}^wL, U_Q = L \cap {}^wU_P$.
Hence by replacing $G, P, L$ by $L, Q, M$, one can define 
a map $\pi_Q : Q^{\io\th}\uni \times \wt V_L 
              \to M^{\io\th}\uni \times \wt V_M$
as in 5.6.  
Similarly, for a parabolic subgroup $Q' = {}^wL \cap P$ of ${}^wL$, 
the map  
$\pi_{Q'}: (Q')^{\io\th}\uni \times \wt V_L 
                \to M^{\io\th}\uni \times \wt V_M$
is defined.
Now $M^{\io\th}\uni \times \wt V_M$ is partitioned into 
finitely many $M^{\th}$-orbits $\wh \CO_1, \dots, \wh\CO_m$, 
and for $(x',v') \in \wh \CO_i$, the variety 
\begin{equation*}
\{(z, y, y', \ol v) \mid (zy', \ol v) 
        \in \CO_L \cap \pi_Q\iv(x',v'), 
        (zy, \ol v) 
        \in \CO_{{}^wL} \cap \pi_{Q'}\iv(x',v')\} 
\end{equation*}
is isomorphic to the product of two varieties 
appeared in Proposition 5.7 (i) by replacing 
$G, L$ by $L, L \cap {}^wL$, and by ${}^wL, L \cap {}^wL$.
Hence by induction hypothesis, its dimension is smaller than 
$(\ol c - \dim\wh \CO_i)/2
          + (\ol c - \dim\wh \CO_i)/2$. 
It follows that 
\begin{equation*}
\tag{5.7.4}
\dim\{ (z,y,y',\ol v) \mid (z, \ol v') \in \wh\CO_i,
          (zy',\ol v) \in \CO_L, (zy,\ol v) \in \CO_{{}^wL} \} \le \ol c,
\end{equation*}
where $\ol v'$ is the image of $\ol v$ under the map $\wt V_L \to \wt V_M$.
Since (5.7.4) holds for any $i$, we obtain (5.7.3), and so (iii) holds.
\par
We show (ii).  Let $q$ be the projection 
$\CZ \to G^{\io\th}\uni \times V$  to the first two factors.
For each $\CO$, we consider $q\iv(\CO)$.  If $q\iv(\CO)$ is 
empty, then the variety in (ii) is also empty, and the 
assertion holds.  So we assume that $q\iv(\CO)$ is not empty.
By (iii), we have $\dim q\iv(\CO) \le d_0$. 
For each $(x,v) \in \CO$, $q\iv(x,v)$ is a product of 
two varieties isomorphic to the variety given in (ii).
Hence the dimension of the variety in (ii) is 
equal to $(\dim q\iv(\CO)  - \dim \CO)/2 \le (d_0 - c)/2
   = \nu_H - \ol\nu -c/2 + \ol c/2$.
Thus (ii) follows. 
\par
We show (i).  Consider the variety 
\begin{equation*}
\CX_{\CO} = \{ ((x,v), gP^{\th}) \in \CO \times H/P^{\th} \mid 
       (g\iv xg, g\iv v) \in \pi_P\iv(\CO_L)\},
\end{equation*}
and let $f$ be the projection on the $\CO$-factors.  Then 
for each $(x,v) \in \CO$, 
the fibre $f\iv(x,v)$ is isomorphic to the variety 
given in (ii), hence the dimension of $\CX_{\CO}$ is less than or 
equal to  $\nu_H - \ol\nu  + c/2 + \ol c/2$.
If we project it to the $H/P^{\th}$-factor, 
its fibre is isomorphic to the variety 
$\CO \cap \pi_P\iv(\CO_L)$.
Hence 
\begin{equation*}
\dim \CO \cap \pi_P\iv(\CO_L) \le \nu_H - \ol\nu + c/2 + \ol c/2
   - \dim H/P^{\th} = (c + \ol c)/2.      
\end{equation*}
Now $\pi_P$ maps the variety $\CO \cap \pi_P\iv(\CO_L)$ onto 
$\CO_L$, and each fibre is isomorphic to the variety given in 
(i).  Hence its dimension 
$\le (c + \ol c)/2 - \ol c = (c - \ol c)/2$.  
This proves (i), and the proposition follows.   
\end{proof}
%%%%
%%%%
\par\medskip
\section{Restriction theorem}

\para{6.1.}
In this section we shall prove the restriction theorem 
for Springer representations, which plays a crucial role
for determining the Springer correspondence explicitly.
Let $P = LU_P$ be as in 5.6.  
Here we consider the special case where 
$L^{\th} \simeq GL_1 \times Sp_{2n-2}$. 
Let 
$V = \wt V_L \oplus V_L'$ be the decomposition in 5.6. 
Take $z = (x,v) \in G^{\io\th}\uni \times V$ and 
$z' = (x',v')  \in L^{\io\th}\uni \times \wt V_L$.
Here we assume that $v' \in \wt V_L$.   
Then $U_P$ stabilizes $v' + V_L'$.
We define a variety
$Y_{z,z'}$ by 
\begin{equation*}
Y_{z,z'} = \{ g \in H \mid   g\iv xg \in x'U^{\io\th}_P, 
                 g\iv v \in v' + V_L' \}.
\end{equation*}
The group $Z_H(z) \times Z_{L^{\th}}(z')U^{\th}_P$ 
acts on $Y_{z,z'}$ by $(g_0, g_1)g = g_0gg_1\iv $. 
\par
Let $d_{z,z'} = 
(\dim Z_H(z) - \dim Z_{L^{\th}}(z'))/2 + \dim U_P^{\th}$.
We note that 
\begin{equation*}
\tag{6.1.1}
\dim Y_{z,z'} \le d_{z,z'}.
\end{equation*}
In fact, put $\CU = \{ (z,gP^{\th}) \in \CO \times H/P^{\th}
   \mid g\iv z \in \pi_P\iv(\CO')\}$, where $\CO$ is the $H$-orbit
containing $z$ and $\CO'$ is the $L^{\th}$-orbit containing $z'$, 
and define a map $f : \CU \to \CO$ by $(z, gP^{\th}) \to z$.
Put $d_{\CO} = (\dim \CX\uni - \dim \CO)/2$ and 
 $\d = (\dim \CX^L\uni - \dim \CO')/2$.
Then by Proposition 5.7 (ii), we see that 
$\dim f\iv(z) \le d_{\CO} - \d$. 
Let $\wh\CU_z = \{ g \in H \mid g\iv z \in \pi_P\iv(\CO')\}$. 
Since $\wh\CU_z \to f\iv(z)$ is a principal $P^{\th}$-bundle, 
$\dim \wh\CU_z \le d_{\CO}-\d + \dim P^{\th}$. 
We define a map $f_1: \wh\CU_z \to \CO'$ by 
$g \mapsto \pi_P(g\iv z)$. Then for each $z' \in \CO'$, $f_1\iv(z')$
coincides with $Y_{z,z'}$. Since $f_1$ is $L^{\th}$-equivariant 
with respect to the action of $L^{\th}$ on $\wh\CU_z$ by the right 
multiplication, and on $\CO'$ by conjugation, we see that 
\begin{align*}
\dim f_1\iv(z') &\le d_{\CO} -\d + \dim P^{\th} - \dim L^{\th} \\
      &= (\dim Z_H(z) - \dim Z_{L^{\th}}(z'))/2 + \dim U_P^{\th}.
\end{align*} 
Hence (6.1.1) holds.
\par
Let $S_{z,z'}$ be the set of irreducible components of 
$Y_{z,z'}$ of dimension $d_{z,z'}$.
\par
Let $W = W_n$ be the Weyl group of $H$ and $W_L$ be the Weyl subgroup
of $W$ corresponding to $L^{\th}$.  Then $W_L \simeq W_{n-1}$, 
where $W_{n-1}$ is embedded in $W_n$ as the subgroup with respect to
$n-1$ letters $\{ 2, \dots, n\}$.  
For $z \in G^{\io\th}\uni \times V$, we denote by $\r^G_z$ the 
irreducible representation of $W$ corresponding to the $H$-orbit 
containing $z$ under the 
Springer correspondence (Theorem 5.4). 
Note that the variety $L^{\io\th}\uni \times V_L$ with $L^{\th}$-
action is exactly the same as $G^{\io\th}\uni \times V$ 
with $Sp_{2n}$-action, by replacing $n$ by $n-1$.
Hence one can consider the Springer correspondence with respect to
$W_L$.  For 
$z' \in L^{\io\th}\uni \times V_L$, we denote by 
$\r^{L}_{z'}$ the irreducible representation of $W_L$ corresponding 
to $z'$.
We shall prove the following theorem, which is an 
analogue of [L2, Theorem 8.3].

\begin{thm}[restriction theorem]  %%%% Theorem 6.2
For $z \in G^{\io\th}\uni \times V$, 
$z' \in L^{\io\th}\uni \times V_L$, we have
\begin{equation*}
\lp \Res^W_{W_L}\r^G_z, \r^L_{z'}\rp_{\, W_L} = |S_{z,z'}|.
\end{equation*}
\end{thm}
 
\para{6.3.}
The theorem will be proved in 6.7 after some preliminaries.
Let $\wt L$ be the subgroup of $G$ generated by $L$ and 
$B \cap Z_G(T^{\io\th})$. Then $\wt L$ is $\th$-stable, and 
$\wt L^{\th} \simeq (B_2 \times L^{\th})/T_2$, where 
under the isomorphism 
$Z_G(T^{\io\th}) \simeq GL_2 \times \cdots \times GL_2$, 
$B_2$ is the Borel subgroup of $SL_2$ 
corresponding to the first factor (contained in $B$), 
and $T_2$ is a maximal torus 
of $B_2$ contained in $L^{\th}$. 
We define varieties
\begin{align*}
\wt\CY^L &= \{(x,v, g\wt L^{\th}) \in 
    G^{\io\th}\reg \times V \times H/\wt L^{\th}
     \mid (g\iv xg, g\iv v) \in \CY^L \}, \\
\CY^L &= \bigcup_{g \in \wt L^{\th}}g(T^{\io\th}\reg \times M_n).
\end{align*}
Then the map $\psi : \wt\CY \to \CY$ is decomposed as 
\begin{equation*}
\begin{CD}
\psi : \wt\CY @>\psi'>>  \wt\CY^{L} @>\psi''>>  \CY,
\end{CD}
\end{equation*}
where $\psi': (x,v, g(B^{\th} \cap Z_H(T^{\io\th})) \mapsto 
   (x,v, g\wt L^{\th})$ under the second expression in (3.1.2), 
and $\psi'': (x,v,g\wt L^{\th}) \mapsto (x,v)$.
Following the notations in 3.1, 3.2, for each $I \subset [1,n]$ 
such that $|I| = m$, we also define 
varieties $\wt \CY_I^{L}, \CY_I^L$; 
\begin{align*}
\wt\CY_I^{L} &= \{ (x,v, g\wt L^{\th}) \in 
  G^{\io\th}\reg \times V \times H/\wt L^{\th} 
      \mid (g\iv xg, g\iv v) \in \CY_I^L \} 
     \simeq H \times{}^{\wt L^{\th}} \CY_I^L,  \\ 
\CY_I^L &= \bigcup_{g \in \wt L^{\th}}g(T^{\io\th}\reg \times M_I).
\end{align*}
Note that in the case where $m \ge 1$, $\CY^L_I$ coincides 
with $\CY^L_{[1,m]}$ if 
$1 \in I$, and coincides with $\CY^L_{[2,m+1]}$ if 
$1 \notin I$.
Then $\wt\CY^L_I$ coincides with $\psi'(\wt\CY_I)$ for any $I$, and 
coincides with $\wt\CY^L_{[1,m]}$ or $\wt\CY^L_{[2,m+1]}$
according to the case where $1 \in I$ or $ 1\notin I$ if $|I| \ne 0$.
Now the map $\psi_I : \wt\CY_I \to \CY_m^0$ is decomposed as
\begin{equation*}
\tag{6.3.1}
\begin{CD}
\psi_I : \wt\CY_I @>\psi_I'>> \wt\CY_I^{L} @>\psi_I''>> \CY_m^0, 
\end{CD}
\end{equation*}
where $\psi_I'$ is the restriction of $\psi'$ on $\wt\CY_I$, and 
$\psi_I''$ is the restriction of $\psi''$ on $\wt\CY_I^L$.
Put $\wh\CY_I^L = H \times^{Z_H(T^{\io\th})_J}(T^{\io\th}\reg \times M_I)$, 
where
$J = I \cup \{1\}$.  Then 
$Z_H(T^{\io\th})_{J} \subset \wt L^{\th}$ and 
one can define a map $\e_I': \wh \CY^L_I \to \wt\CY_I^L$
by $(x,v, gZ_H(T^{\io\th})_{J}) \mapsto (x,v, g\wt L^{\th})$.
Now $\psi_I'$ is decomposed as 
\begin{equation*}
\tag{6.3.2}
\begin{CD}
\psi_I': \wt\CY_I @>\xi'_I>>  \wh \CY^L_I  @>\e_I'>> \wt\CY_I^{L}, \\ 
\end{CD}
\end{equation*}
where 
$\xi_I': (x,v, g(B^{\th} \cap Z_H(T^{\io\th})) 
  \mapsto (x,v, gZ_H(T^{\io\th})_J)$.
Here $\e_I'$ is a finite Galois covering with group $\CW_{I,L}$, where
$\CW_{I,L}$ is a subgroup of $\CW_L \simeq S_{n-1}$ such that 
$\CW_{I,L} \simeq S_{I - \{1\}} \times S_{I'}$ if $J =  I$, and 
$\CW_{I,L} \simeq S_{I} \times S_{I' - \{1\}}$ if $J \ne  I$.
Moreover, $\xi_I'$ is a $\BP_1^{n-m}$-bundle 
(resp. $\BP_1^{n-m-1}$-bundle) if $J = I$ (resp. $J \ne I$).
It follows that $\End ((\e_I')_*\Ql) \simeq \Ql[\CW_{I,L}]$ and 
we have 
\begin{equation*}
\tag{6.3.3}
(\e_I')_*\Ql \simeq 
          \bigoplus_{\r' \in \CW\wg_{I,L}}
                       \r' \otimes \CL_{\r'},
\end{equation*}
where $\CL_{\r'} = \Hom_{\CW_{I,L}}(\r', (\e_I')_*\Ql)$
is an irreducible local system on $\wt\CY^L_I$.
This implies  that 
\begin{equation*}
\tag{6.3.4}
(\psi_I')_*\Ql \simeq \bigoplus_{\r' \in \CW\wg_{I,L}}
                         H^{\bullet}(\BP_1^{n-m-\ve})
                                \otimes \r'\otimes \CL_{\r'},
\end{equation*}
where $\ve = $ 0 or 1 according to the case $J = I$ or 
$J \ne I$.  (In the case where $I = \emptyset$, we understand that
$\ve = 1$ and $\CW_{I,L} = \CW_L$).
Since $\psi_I = \psi''_I\circ \psi_I'$, 
we see that 
\begin{equation*}
(\psi_I)_*\Ql \simeq (\psi_I'')_*(\psi_1')_*\Ql 
    \simeq \bigoplus_{\r' \in \CW\wg_{I,L}} 
    H^{\bullet}(\BP_1^{n-m-\ve})\otimes  \r' \otimes (\psi_I'')_*\CL_{\r'}.
\end{equation*}
Since 
$\psi_I = \e_I\circ\xi_I$,
by using the decomposition (3.4.3) and (3.5.1), we have
\begin{equation*}
\tag{6.3.5}
(\psi''_I)_*\CL_{\r'} \simeq \bigoplus_{\r \in \CW\wg_I}
        H^{\bullet}(\BP_1^{\ve}) \otimes 
           \Hom_{\CW_{I,L}}(\r', \r)\otimes \CL_{\r}.
\end{equation*}
\par
We put $I_1 = [1,m]$ and $I_2 = [2,m+1]$. 
Let $\psi_m'$ be the restriction of $\psi'$  on $\wt\CY_m^+$, and 
let  $\psi_m'': \wt\CY^{L,+}_m = (\psi'')\iv(\CY_m^0) \to \CY_m^0$
be the restriction of $\psi''$.  
Then $\wt\CY^{L,+}_m = \wt\CY^L_{I_1}\coprod \wt\CY^L_{I_2}$ gives 
the decomposition to irreducible components.
We have 
\begin{equation*}
\begin{CD}
\psi_m : \wt\CY_m^+ @>\psi_m'>>  \wt\CY^{L,+}_m   @>\psi_m''>> \CY_m^0.
\end{CD}
\end{equation*}
Note that $\CW_L \simeq S_{n-1}$ is the stabilizer of 1 in $\CW$ under 
the identification $\CW \simeq S_n$.  
Let $\wt\CW_L, \wt\CW_{I,L}$ be the corresponding 
subgroups of $\wt\CW$.   
Now (6.3.4) implies, by a similar argument as in 3.5, that 
\begin{equation*}
\tag{6.3.6}
\begin{split}
(\psi'_m)_*\Ql \simeq &\bigoplus_{\r' \in \CW\wg_{I_1, L}}
\Ind_{\wt\CW_{I_1,L}}^{\wt\CW_L}(H^{\bullet}(\BP_1^{n-m}) 
   \otimes \r')\otimes \CL_{\r'} \\
&\oplus \bigoplus_{\r' \in \CW\wg_{I_2,L}}
    \Ind_{\wt\CW_{I_2,L}}^{\wt\CW_L}(H^{\bullet}(\BP_1^{n-m-1})
           \otimes \r')\otimes \CL_{\r'}
\end{split}
\end{equation*}
for $m \ge 1$, 
where $\CL_{\r'}$ is a local system on $\wt\CY^L_{I_{\ve}}$ 
corresponding to
$\r' \in \CW\wg_{I_{\ve},L}$ for $\ve = 1$ or 2.
While $(\psi'_m)_*\Ql$ for $m = 0$ is given by $(\psi_I')_*\Ql$ 
for $I = \emptyset$.
\par
Now $\CW\wg_{I_1,L}$ is parametrized by a subset of 
$\CP_{n-1,2}$ consisting of $\Bmu'$ such that $m(\Bmu') = m-1$, 
and we denote $\CL_{\r'}$ by $\CL_{\Bmu'}$ 
if $\r'\in \CW\wg_{I_1,L}$ 
corresponds to $\Bmu' \in \CP_{n-1,2}$.  
Let $\wt\CY^L_m$ be the closure of 
$\wt\CY^L_{[1,m]}$ in $\wt\CY^L$.
Then $\wt\CY_m^L = \coprod_{m'< m}\wt\CY_{m'}^{L,+} 
            \coprod \wt\CY^L_{[1,m]}$.
(The closure of $\wt\CY^L_{[2,m+1]}$ in $\wt\CY^L$ is given by 
$\coprod_{1 \le m' \le m}\wt\CY^L_{[2, m'+1]}$.)
Put $d^L_m = \dim \wt\CY^L_m$.
Then $\wt\CY^L_{[1,m]}$ is an open dense smooth
subset of $\wt\CY^L_m$, and we consider 
the intersection cohomology $\IC(\wt\CY^L_m, \CL_{\Bmu'})$.
By using a similar argument as in the proof of Proposition 3.6, 
one can show that 
\begin{equation*}
\tag{6.3.7}
\psi'_*\Ql[d_n^L] \simeq \bigoplus_{\Bmu' \in \CP_{n-1,2}}
     \wt V_{\Bmu'} 
     \otimes \IC(\wt\CY^{L}_{m(\Bmu')+1},
               \CL_{\Bmu'})[d^L_{m(\Bmu')+1}],
\end{equation*}
where $d_n^L = \dim \wt\CY^{L}$. 
Moreover, $\IC(\wt\CY^L_{m(\Bmu')+1}, \CL_{\Bmu'})$ is a constructible 
sheaf on $\wt\CY^L_{m(\Bmu')+1}$.
\par
On the other hand, in view of (6.3.5), we have
\begin{equation*}
(\psi''_m)_*\CL_{\r'}
     \simeq \begin{cases}
              \bigoplus_{\r \in \CW_{I_1}\wg}
                \Hom_{\CW_{I_1,L}}(\r',\r)\otimes \CL_{\r}
       &\quad\text{ if } \r' \in \CW\wg_{I_1,L}, \\
               \bigoplus_{\r \in \CW_{I_2}\wg}
                 H^{\bullet}(\BP_1)\otimes 
                    \Hom_{\CW_{I_2,L}}(\r', \r)\otimes \CL_{\r}
       &\quad\text{ if } \r' \in \CW\wg_{I_2,L},               
\end{cases}
\end{equation*}
Let $\IC^L_{\Bmu'} = 
\IC(\wt\CY^{L}_{m(\Bmu')+1}, \CL_{\Bmu'})[d^L_{m(\Bmu')+1}]$. 
By applying the argument in the proof of Proposition 3.6 
to the above formula, we obtain 
\begin{equation*}
\tag{6.3.8}
\begin{split}
  \psi''_*(\IC^L_{\Bmu'})&[d_n - d_n^L] \\
  &\simeq
   \bigoplus_{\Bmu \in \CP_{n,2}}
    \Hom_{\wt\CW_L}(\wt V_{\Bmu'}, \wt V_{\Bmu}) \otimes 
         \IC(\CY_{m(\Bmu)}, \CL_{\Bmu})[d_{m(\Bmu)}].
\end{split}
\end{equation*}
%%%
\para{6.4.}
Recall the map $\pi : \wt\CX \to \CX$ in 4.1. 
In this subsection, we follow the notation in 4.3, 4.5.
We define varieties 
\begin{align*}
\wt\CX^P &= \{ (x,v, gP^{\th}) \in G^{\io\th} \times V 
  \times H/P^{\th} \mid (g\iv xg, g\iv v) \in \CX^P\}, \\
\CX^P &= \bigcup_{g \in P^{\th}}g(B^{\io\th} \times M_n).
\end{align*}
We define $\pi': \wt\CX \to \wt\CX^P$, $\pi'': \wt\CX^P \to \CX$
by $\pi'(x,v, gB^{\th}) = (x,v, gP^{\th})$, 
$\pi''(x,v, gP^{\th}) = (x,v)$.  Then $\pi = \pi''\circ \pi''$, 
and we obtain the following commutative diagram.

\begin{equation*}
\tag{6.4.1}
\begin{CD}
\wt\CY  @>\psi'>> \wt\CY^L  @>\psi''>>  \CY  \\
   @Vj_0VV        @Vj_1VV         @VVj_2V   \\
\wt\CX   @>\pi'>>  \wt\CX^P  @>\pi''>>   \CX,  \\
\end{CD}
\end{equation*}
where $j_0, j_2$ are natural inclusion maps, and 
$j_1 : (x,v, g\wt L^{\th}) \mapsto (x,v, gP^{\th})$. 
Put $\wt\CX_m^{P,+} = (\pi'')\iv\CX_m^0$.  Then we have 
a commutative diagram 
\begin{equation*}
\tag{6.4.2}
\begin{CD}
\wt\CY_m^+  @> \psi_m'>>  \wt\CY_m^{L,+} @> \psi_m''>>  \CY_m^0 \\
@VVV    @VVV     @VVV      \\
\wt\CX_m^+   @>\pi_m'>>  \wt\CX_m^{P,+} @>\pi_m''>>  \CX_m^0,
\end{CD}
\end{equation*}
where $\pi_m'$ is the restriction of $\pi'$ on $\wt\CX_m^+$, 
$\pi_m''$ is the restriction of $\pi''$ on $\wt\CX_m^{P.+}$,
and vertical maps are restrictions of the corresponding 
vertical maps in (6.4.1).   Since $\pi_m = \pi_m''\circ \pi_m'$
is proper, $\pi_m'$ is proper. 
We note that $\pi_m''$ is also proper.  In fact, since
$\CX^P$ is the image of $P^{\th} \times^{B^{\th}}(B^{\io\th} \times M_n)$
under the map $\pi$, $\CX^P$ is closed in $\CX$.  Hence 
$\pi'': \wt\CX^P \simeq H\times^{P^{\th}}\CX^P \to \CX$ is proper, 
and so $\pi''_m$ is proper. 
\par 
Recall the decomposition $\wt\CX_m^+ = \coprod_{I}\wt\CX_I$.
We denote by $\wt\CX^P_I$ the image of $\wt\CX_I$ by $\pi_m'$.  
Then $\wt\CX^P_I$ coincides with $\wt\CX^P_{[1,m]}$ 
(resp. $\wt\CX^P_{[2,m+1]}$) if $1 \in I$ (resp. $1 \notin I$).
Thus $\wt\CX_m^{P,+} = \wt\CX^{P}_{I_1} \coprod \wt\CX^{P}_{I_2}$, 
and $\wt\CY^L_{I_{\ve}}$ is open dense in $\wt\CX^{P}_{I_{\ve}}$ for 
$\ve = $ 1 or 2. 
\par  
We now show the following formula.
\begin{equation*}
\tag{6.4.3}
\begin{split}
(\pi'_m)_*\Ql \simeq &\bigoplus_{\r' \in \CW\wg_{I_1, L}}
\Ind_{\wt\CW_{I_1,L}}^{\wt\CW_L}(H^{\bullet}(\BP_1^{n-m}) 
   \otimes \r')\otimes \IC(\wt\CX^{P}_{I_1}, \CL_{\r'}) \\
&\oplus \bigoplus_{\r' \in \CW\wg_{I_2,L}}
    \Ind_{\wt\CW_{I_2,L}}^{\wt\CW_L}(H^{\bullet}(\BP_1^{n-m-1})
           \otimes \r')\otimes \IC(\wt\CX^{P}_{I_2}, \CL_{\r'}).
\end{split}
\end{equation*}
In fact, since $\pi_m'$ is proper, by the decomposition theorem, 
$(\pi'_m)_*\Ql$ can be written as a direct sum of the complexes 
of the form $A[i]$, where $A$ is a simple perverse sheaf on 
$\wt\CX_m^{P,+}$ with some degree shift $i$.
Note that $(\pi_m')_*\Ql|_{\wt\CY_m^{L,+}} \simeq (\psi_m')_*\Ql$, 
and the decomposition of $(\psi_m')_*\Ql$ as a semisimple complex 
is given in (6.3.6). Hence in order to prove (6.4.3), it is enough 
to show that $\wt\CY_m^{L,+} \cap \supp A \ne \emptyset$ for any 
direct summand $A[i]$.  Since $\pi_m''$ is proper, again by the 
decomposition theorem, $(\pi_m'')_*A[i]$ can be written 
as a direct sum of the complexes $B[j]$, where $B$ is a simple
perverse sheaf on $\CX_m^0$.
Since $\supp A$ is irreducible, we may assume that $\supp A$ 
is contained in $\wt\CX^{P}_{I_1}$ or $\wt\CX^P_{I_2}$. 
Suppose that $\wt\CY_m^{L,+} \cap \supp A  = \emptyset$.  
Then there exists
$m' < m$ such that $\wt\CY_{m'}^{L,+} \cap \supp A$ is open dense 
in $\supp A$. 
By applying $\psi''_{m'}$ on it, we see that
$\dim \supp B = \dim \CX_{m'}^0 < \dim \CX_m^0$. 
But Proposition 4.8 asserts that $\dim \supp B = \dim \CX_m^0$ 
for any simple component $B$ in $(\pi_m)_*\Ql$. 
Hence $\wt\CY_m^{L,+} \cap \supp A \ne \emptyset$, and (6.4.3) 
holds.
\par
The above argument also shows that 
\begin{equation*}
\tag{6.4.4}
\begin{split}
(\pi_m'')_*&\IC(\wt\CX^P_{I}, \CL_{\r'}) \\
   &\simeq
     \begin{cases}
       \bigoplus_{\r \in \CW_{I_1}\wg}\Hom_{\CW_{I_1,L}}(\r',\r)
                       \otimes \IC(\CX_m^0, \CL_{\r})  
         &\quad\text{ if } I = I_1, \\
       \bigoplus_{\r \in \CW_{I_2}\wg}H^{\bullet}(\BP_1)\otimes 
          \Hom_{\CW_{I_2,L}}(\r', \r)\otimes  \IC(\CX_m^0,\CL_{\r})
          &\quad\text{ if } I = I_2.        
     \end{cases}
\end{split}
\end{equation*}
\par
Let $\wt\CX^P_m$ be the closure of $\wt\CX^P_{[1,m]}$ in 
$\wt\CX^P$.  Then $\wt\CX^P_m = (\coprod_{m' < m}\wt\CX_{m'}^{P,+})
  \coprod \wt\CX^P_{[1,m]}$ and 
$\wt\CX^P_m \simeq H \times^{P^{\th}}\CX^P_m$, where 
$\CX^P_m = \bigcup_{g \in P^{\th}}g(B^{\io\th} \times M_m)$.  
(The closure of $\wt\CX^P_{[2,m+1]}$
in $\wt\CX^P$ is given by $\coprod_{1 \le m' \le m}\wt\CX^P_{[2, m'+1]}$.)
Now by a similar argument as in the last step of Theorem 4.2
(see 4.9), we have the following formulas from (6.3.7) and (6.3.8);
\begin{align*}
\tag{6.4.5}
\pi'_*\Ql[d_n^L] &\simeq \bigoplus_{\Bmu' \in \CP_{n-1,2}}
    \wt V_{\Bmu'}\otimes \IC(\wt\CX^P_{m(\Bmu')+1}, \CL_{\Bmu'})
                       [d^L_{m(\Bmu')+1}]  \\ 
\tag{6.4.6}
\pi''_*(\IC^P_{\Bmu'})[a]  
     &\simeq  \bigoplus_{\Bmu \in \CP_{n,2}}
           \Hom_{\wt \CW_L}(\wt V_{\Bmu'}, \wt V_{\Bmu})
              \otimes \IC(\CX_{m(\Bmu)}, \CL_{\Bmu})[d_{m(\Bmu)}],
\end{align*}
where $\IC^P_{\Bmu'} = \IC(\wt\CX^P_{m(\Bmu')+1}, \CL_{\Bmu'})$
and $a = d_n - d_n^L + d_{m(\Bmu')+1}^L$.

\para{6.5.}
Let $\wt\CX^P\uni 
  = \{ (x,v, gP^{\th}) \in \wt\CX^P \mid x \in G^{\io\th}\uni \}$.
Let $\CO' = \CO_{{\Bmu'}^{\bullet}}$ be the $L^{\th}$-orbit 
in $L^{\io\th}\uni \times V_L$ 
corresponding to $\wt V_{\Bmu'}$ under the Springer correspondence 
for $L$.  We define a variety 
\begin{equation*}
D = \{ (x,v, gP^{\th}) \in \wt\CX_m^P \mid 
          (g\iv xg, g\iv v) \in \pi_P\iv(\ol\CO')\}, 
\end{equation*}
where $\pi_P : P^{\io\th}\uni \times V \to L^{\io\th}\uni \times V_L$ 
is as in 5.5, and $m = m(\Bmu')+1$.
Let $\IC^P_{\Bmu'}$ be as in 6.4.  
The following fact holds. 

\begin{equation*}
\tag{6.5.1}
\supp \IC_{\Bmu'}^P \cap \wt\CX\uni^P \subset D.
\end{equation*}
We show (6.5.1). 
Let $\CX_{m-1}^L = \bigcup_{\ell \in L^{\th}}
                   \ell(B_L^{\io\th} \times M^L_{m-1})$, 
where $B_L = B \cap L$ is a Borel subgroup of $L$ containing $T$, and 
$M^L_{m-1}$ is the image of $M_m$ under the projection
$V \to V_L$. 
We have a diagram 
\begin{equation*}
\tag{6.5.2}
\begin{CD}
\wt\CX_{m}^P  @< p_1 <<  H \times \CX_{m}^P @>p_2 >> \CX^L_{m-1},
\end{CD}
\end{equation*}
where $p_1$ is the natural map 
$ H \times \CX^P_{m} \to 
   \wt\CX_{m}^P \simeq 
      H\times^{P^{\th}}\CX_{m}^P$, and  
$p_2$ is the restriction to $H \times \CX_{m}^P$ of the map 
$H \times P^{\io\th} \times V \to L^{\io\th} \times V_L$, 
$(g, x, v) \mapsto (\ol x, \ol v)$. (Here $x \to \ol x$ is the 
natural projection $P^{\io\th} \to L^{\io\th}$ and 
$v \mapsto \ol v$ is the projection $V \to V_L$.)
Then we have
\par\medskip\noindent
(a) \ $p_1$ is a principal $P^{\th}$-bundle, 
\par\noindent
(b) \ $p_2$ is a locally trivial fibration with fibre isomorphic 
to $H \times U_P^{\io\th}\times \Bk$.
\par\medskip\noindent
In fact, (a) is clear from the definition.  
Since $M_m$ is $U_P^{\th}$-stable, 
we have 
\begin{equation*}
\CX_{m}^P = \bigcup_{\ell \in L^{\th}}\ell(B^{\io\th}\times M_{m}) 
   \simeq U_P^{\io\th} \times \Bk \times 
     \bigcup_{\ell \in L^{\th}}\ell(B_L^{\io\th} \times M^L_{m-1}), 
\end{equation*}
and (b) follows.  
\par
We consider the complex $\IC_{\Bmu'}^P$ on $\wt\CX_{m}^P$. 
By (a), (b), both of $p_1, p_2$ are smooth maps with connected fibre.
It follows that there exists a simple perverse sheaf $A_{\Bmu'}$ on 
$\CX_{m-1}^L$ such that $p_1^*\IC_{\Bmu'}^P \simeq p_2^*A_{\Bmu'}$ 
up to shift.   Note that $\CX_{m-1}^L$ is a variety constructed from 
$(L,V_L)$ defined in a similar way as $\CX_m$ from $(G,V)$.  
Let $\IC(\CX^L_{m-1},\CL^L_{\Bmu'})$ be the complex on $\CX^L_{m-1}$
appeared in 4.1 with respect to $L$.
We show that 
\begin{equation*}
\tag{6.5.3}
A_{\Bmu'} \simeq \IC(\CX_{m-1}^L, \CL^L_{\Bmu'})[\dim \CX_{m-1}^L].
\end{equation*}
\par
For each $I =  [1,m]$ put 
\begin{align*}
\wt\CY_I\nat &=   \wt L^{\th} \times^{B^{\th}\cap 
         Z_H(T^{\io\th})}(T^{\io\th}\reg \times M_I), \\
\wh\CY_I\nat &= \wt L^{\th} \times^{Z_H(T^{\io\th})_I}
             (T^{\io\th}\reg \times M_I).
\end{align*}
Then $\wt\CY_I \simeq H\times^{\wt L^{\th}}\wt\CY_I\nat$,  
$\wh\CY_I \simeq H\times^{\wt L^{\th}}\wh\CY_I\nat$.
Also we have $\wt\CY_I^L \simeq H\times ^{\wt L^{\th}}\CY_I^L$
by 6.3.
We  put 
\begin{align*}
\wt\CY_{L,I} &= L^{\th}\times^{B_L^{\th}\cap Z_{L^{\th}}(T^{\io\th})}
                         (T^{\io\th}\reg \times M^L_{I'}), \\
\wh\CY_{L,I} &= L^{\th}\times^{Z_{L^{\th}}(T^{\io\th})_I}
                         (T^{\io\th}\reg \times  M^L_{I'}),  \\
\CY^0_{L,m-1} &= \bigcup_{\ell \in L^{\th}}\ell(T^{\io\th}\reg 
                       \times M_{I'}^L).
\end{align*}
where $M^L_{I'}$ is the image of $M_I$ under the natural map 
$V \to V_L$ with $I' = [2,m]$, and 
$Z_{L^{\th}}(T^{\io\th})_I = Z_H(T^{\io\th})_I \cap L^{\th}$.  
Then we have a commutative diagram

\begin{equation*}
\tag{6.5.4}
\begin{CD}
\wt\CY_I @< \wt q_1<< H \times \wt\CY_I\nat @>\wt q_2 >>  \wt\CY_{L, I} \\   
@V\xi_I VV            @VV \a_IV           @VV \xi_{L,I}V   \\
\wh\CY_I @< \wh q_1<<  H \times \wh\CY_I\nat  @>\wh q_2>>  \wh\CY_{L,I} \\
@V\e'_I VV            @VV \b_I V           @VV\e_{L,I} V    \\ 
\wt\CY_I^L  @< q_1 <<  H \times \CY^L_I  @> q_2 >>  \CY^0_{L,m-1},
\end{CD}
\end{equation*}
where $\wt q_1, \wh q_1, q_1$ and $\a_I, \b_I$ are defined naturally.
$\wt q_2$ is defined by 
$(g, \ell*(t, v)) \mapsto \ol \ell*(t, \ol v)$,  where 
$\ell \mapsto \ol \ell$ is the projection $\wt L^{\th} \to L^{\th}$, 
and $v \mapsto \ol v$ is the map $M_I \to  M^L_{I'}$, and the maps
$\wh q_2, q_2$ are defined similarly.  
$\xi_{L,I}, \e_{L,I}$ are defined in a similar way as $\xi_I, \e_I$ 
for $\wt\CY_I, \wh\CY_I$.
Then 
$\wt q_1, \wh q_1, q_1$ are principal $\wt L^{\th}$-bundles, 
and $\wt q_2, \wh q_2, q_2$ are principal $H \times \Bk^*$-bundles.
Moreover, $\a_I$ is a $\BP_1^{n-m}$-bundle, and $\b_I$ is a finite
Galois covering with Galois group $\CW_{L,I}$.  All the squares 
in the diagram are cartesian squares.
\par
Note that $\wt\CY_{L,I} \simeq \Bk^* \times \wt\CY^{(n-1)}_{I'}$, 
where $\wt\CY_{I'}^{(n-1)}$ is an object for $Sp_{2n-2}$ with 
$I' = [2,m] \subset [2, n]$, similar to $\wt\CY_I$ for $Sp_{2n}$. 
Similar formulas hold for $\wh\CY_{L,I}, \CY^0_{L,m-1}$, and the maps 
$\xi_I, \e_I$ are compatible with the situation for $Sp_{2n-2}$.
Thus one can define a local system $\CL^L_{\r'}$ on $\CY^0_{m-1}$ 
with respect to $Sp_{2n-2}$.  It follows from the property of the 
diagram (6.5.4), one sees that 
$q_1^*\CL_{\r'} \simeq q_2^*\CL^L_{\r'}$. 
Note that $\wt\CY^L_I$ is open dense in $\wt\CX_m^P$, 
$\CY_{L, m-1}^0$ is open dense in $\CX^L_{m-1}$, and the 
lowest row in (6.5.4) is the restriction of the diagram  
(6.5.2).  This shows that the restriction of $A_{\Bmu'}$ on 
$\CY_{L, m-1}^0$ gives a local system $\CL^L_{\r'}$ 
corresponding to $\Bmu'$.  Hence (6.5.3) holds. 
\par
Now by the Springer correspondence for $L$, the support of 
the restriction of $A_{\Bmu'}$ on $L^{\io\th}\uni \times V_L$ 
is contained in $\ol \CO'$.
Then (6.5.1) follows from (6.5.2) and (6.5.3).

\para{6.6.}
Let $\CO = \CO_{\Bmu^{\bullet}}$ be the $H$-orbit 
in $G^{\io\th} \times V$ corresponding 
to $\wt V_{\Bmu}$ under the Springer correspondence. 
Let $d_{\CO}$ be as in 6.1.   
In view of Theorem 5.4, (5.4.3) implies that
$\CH^{2d_{\CO}}(\pi_*\Ql)|_{\CO} \simeq \wt V_{\Bmu} \otimes \Ql$,
where $\Ql$ is the constant sheaf on $\CO$. 
Then (6.4.6) implies that 
\par\medskip\noindent
(6.6.1) \ 
The inner product $\lp \wt V_{\Bmu}, \wt V_{\Bmu'}\rp$ coincides with
the rank of the constant sheaf 
$\CH^{2d_{\CO} + c}\pi''_*(\IC^P_{\Bmu'})|_{\CO}$, where 
$c = d^L_{m(\Bmu')+1} - d_n^L$. 
\par\medskip
Then by (6.5.1), we have
\par\medskip\noindent
(6.6.2) \
$\lp \wt V_{\Bmu}, \wt V_{\Bmu'}\rp$ coincides with 
the rank of the constant sheaf 
$\CH^{2d_{\CO}+c}(\pi''|_D)_!(\IC^P_{\Bmu'}|_D)|_{\CO}$. 
\par\medskip
Let $D^0 = \{(x,v, gP^{\th}) \in \wt \CX^P_m \mid 
                g\iv(x,v) \in \pi_P\iv(\CO')\}$ 
be an open subset of $D$.
Let $x_{\Bmu, \Bmu'}$ be the rank of the constant sheaf 
$\CH^{2d_{\CO}+c}(\pi''|_{D^0})_!(\IC_{\Bmu'}^P|_{D^0})|_{\CO}$.
We want to show that
\begin{equation*}
\tag{6.6.3}
\lp \wt V_{\Bmu}, \wt V_{\Bmu'}\rp = x_{\Bmu, \Bmu'}.
\end{equation*}
First we show that 
\par\medskip\noindent
(6.6.4) \ The natural map 
$\CH^{2d_{\CO}+c}(\pi''|_{D^0})_!(\IC_{\Bmu'}^P)|_{D^0})|_{\CO} 
   \to  \CH^{2d_{\CO}+c}(\pi''|_{D})_!(\IC_{\Bmu'}^P)|_D)|_{\CO}$ 
is surjective.
\par\medskip
In order to prove (6.6.4), it is enough to show that 
\begin{equation*}
\CH^{2d_{\CO}+c}(\pi''|_{D - D^0})_!(\IC^P_{\Bmu'}|_{D - D^0})|_{\CO} = 0,
\end{equation*}
which is equivalent to the statement that 
\begin{equation*}
\BH^{2d_{\CO}+c}_c({\pi''}\iv(z) \cap (D - D^0), \IC_{\Bmu'}^P) =0
\end{equation*}
for any $z \in \CO$.
For any $L^{\th}$-orbit $\CO'_1 \subset \ol\CO' - \CO'$, put 
$D_{\CO'_1} = \{ (x,v,gP^{\th}) \in \wt\CX_m^P 
                \mid g\iv(x,v) \in \pi_P\iv(\CO'_1)\}$.
Then $D - D^0$ can be partitioned into locally closed pieces 
$D_{\CO_1'}$,  Thus it is enough to show that 
\begin{equation*}
\BH^{2d_{\CO}+c}_c({\pi''}\iv(z) \cap D_{\CO'_1}, \IC_{\Bmu'}^P) = 0
\end{equation*}
for any such a piece $\CO'_1$.
The last equality will follow from the following statement 
by a hypercohomology spectral sequence.
\begin{equation*}
\tag{6.6.5}
H^i_c({\pi''}\iv(z) \cap D_{\CO_1'}, \CH^j(\IC^P_{\Bmu'})) \ne 0 
 \ \Rightarrow \ i+j < 2d_{\CO} +c. 
\end{equation*}
We show (6.6.5).  
The hypothesis implies that 
\begin{equation*}
i \le 2 \dim ({\pi''}\iv(z) \cap D_{\CO'_1}) \le (\dim \CX\uni - \dim \CO) 
  - (\dim {\CX}\uni^L - \dim \CO'_1)
 \end{equation*}
by Proposition 5.7 (ii). It also implies that 
$\CH^j(\IC_{\Bmu'}^P)|_{D_{\CO_1'}} \ne 0$. 
By (6.5.3), this last condition is equivalent to the condition that
$\CH^j\IC(\CX^L_{m-1}, \CL^L_{\Bmu'})|_{\CO_1'} \ne 0$.
By Theorem 5.4 (ii) applied to $L$, 
this is equivalent to 
$\CH^j\IC(\ol\CO',\Ql)[a]|_{\CO'_1} \ne 0$
with $a = \dim \CO' - \dim \CX^L\uni - 
              \dim \CX^L_{m(\Bmu')} + \dim \CX^L$.
It follows that 
\begin{equation*}
j < \dim {\CX}\uni^L - \dim \CO_1' + \dim \CX^L_{m(\Bmu')} - \dim \CX^L.
  \end{equation*} 
Here we note that 
$\dim \CX^L_{m(\Bmu')} - \dim \CX^L = d_{m(\Bmu')+1}^L - d_n^L$ by
(6.5.2) since a similar diagram holds replacing $\CX$ by $\CY$.
Hence (6.6.5) holds and so (6.6.4) follows.   
\par
Now (6.6.4) implies that 
$\lp \wt V_{\Bmu}, \wt V_{\Bmu'}\rp \le x_{\Bmu, \Bmu'}$.
Since $\sum_{\Bmu' \in \CP_{n-1,2}}
 \lp \wt V_{\Bmu}, \wt V_{\Bmu'}\rp \, \dim \wt V_{\Bmu'} 
        = \dim \wt V_{\Bmu}$, in order to prove (6.6.3), it is 
enough to show that 
\begin{equation*}
\tag{6.6.6}
\sum_{\Bmu' \in \CP_{n-1,2}}x_{\Bmu,\Bmu'}
      \dim \wt V_{\Bmu'} = \dim \wt V_{\Bmu}.
\end{equation*} 
Let $\d$ be as in 6.1.
By (6.5.3) and Theorem 5.4 (ii), 
$\CH^j(\IC_{\Bmu'}^P)|_{D^0} = 0$ unless $j = -a$. 
This implies that $x_{\Bmu,\Bmu'}$ is the rank of the constant sheaf 
\begin{equation*}
\tag{6.6.7}
\CH^{2d_{\CO} -2\d}(\pi''|_{D^0})_!
((\CH^{-a}\IC_{\Bmu'}^P)|_{D^0})|_{\CO}
\end{equation*}
since $2d_{\CO} +c +a = 2d_{\CO} - 2\d$. 
It follows from (6.4.5) that 
$R^{2\d}\pi'_!\Ql|_{D^0} \simeq \wt V_{\Bmu'}\otimes 
                         (\CH^{-a}\IC_{\Bmu'}^P)|_{D^0}$. 
This implies that 
$x_{\Bmu,\Bmu'}\dim \wt V_{\Bmu'}$ is the 
rank of the constant sheaf 
\begin{equation*}
R^{2d_{\CO} - 2\d}(\pi''|_{D^0})_!(R^{2\d}\pi'_!\Ql|_{D^0})|_{\CO}.
\end{equation*}
By the spectral sequence of the composition $\pi''\circ \pi'$, 
the last formula is the same as the constant sheaf 
$R^{2d_{\CO}}(\pi|_{{\pi'}\iv(D^0)})_!\Ql|_{\CO}$.            
If we put $D_{\Bmu'} = D^0$, 
$\wt\CX = \coprod_{\Bmu'}{\pi'}\iv(D_{\Bmu'})$ gives a partition 
of $\wt\CX$ into locally closed pieces.  Hence by the long exact sequence 
associated to $R^*\pi_!$, we see that 
$\sum_{\Bmu'}x_{\Bmu,\Bmu'}\dim \wt V_{\Bmu'}$ coincides with the 
rank of the constant sheaf $R^{2d_{\CO}}\pi_*\Ql|_{\CO}$, which is equal to
$\dim \wt V_{\Bmu}$.  This proves (6.6.6), and (6.6.3) follows. 

\para{6.7.}
We are now ready to prove Theorem 6.2.
We follow the notation in 6.1.
Then $\CU \subset D^0$ and the restriction of $\pi''$ on $\CU$ 
coincides with $f$.
It follows from (6.6.7) that 
$x_{\Bmu,\Bmu'}$ coincides with the rank of the constant sheaf 
$R^{2d_{\CO}-2\d}(\pi''|_{\CU})_!(\CH^{-a}\IC_{\Bmu'}^P|_{\CU})|_{\CO}$.
Since $\CH^{-a}\IC_{\Bmu'}^P|_{\CU}$ is a constant sheaf $\Ql$, 
$x_{\Bmu,\Bmu'}$ is equal to the rank of the constant sheaf 
$R^{2d_{\CO}-2\d}f_!\Ql$ on $\CO$, hence is equal to
$\dim H_c^{2d_{\CO} - 2\d}(f\iv(z), \Ql)$ for $z \in \CO$.   
By Proposition 5.7 (ii), we know that 
$\dim f\iv(z) \le d_{\CO} - \d$ for any $z \in \CO$.
It follows that $x_{\Bmu,\Bmu'}$ is equal to the number of 
irreducible components in $f\iv(z)$ of dimension $d_{\CO} - \d$. 
Since $\wh \CU_z \to f\iv(z)$ is a principal $P^{\th}$-bundle, 
$x_{\Bmu,\Bmu'}$ coincides with the number of irreducible components 
of $\wh \CU_z$ of dimension $d_{\CO} - \d + \dim P^{\th}$.
Recall the map $f_1 : \wh \CU_z \to \CO'$ in 6.1 such that 
$f_1\iv(z') = Y_{z,z'}$ for any $z' \in \CO'$.  
Since $f_1$ is $L^{\th}$-equivariant, 
$x_{\Bmu, \Bmu'}$ coincides with the number of irreducible 
components of $f_1\iv(z')$ of dimension $d_{z,z'}$.
In view of (6.6.3), this completes the proof of Theorem 6.2.
%%%%
%%%%
\par\bigskip

\section{Determination of Springer correspondence}

The explicit description of the Springer correspondence
was given in Kato [Ka2], by computing Joseph polynomials. 
In this section, we shall give an alternate approach 
based on the restriction theorem (Theorem 6.2). 

\begin{thm}  %%% Thm 7.1
Under the notation in Theorem 5.4, we have $\Bmu^{\bullet} = \Bmu$,
namely, the Springer correspondence 
is given by $\CO_{\Bmu} \lra \wt V_{\Bmu}$ for each $\Bmu \in \CP_{n,2}$.  
\end{thm}
\para{7.2.}
Before proving the theorem, we need a lemma.  First we prepare 
some notation.  Take 
$\Bmu = (\mu^{(1)}, \mu^{(2)}) \in \CP_{n,2}$, and 
put $\nu = \mu^{(1)} + \mu^{(2)}$ (see 2.1). We write $\nu$ as   
$\nu = (\nu_1, \dots, \nu_a) \in \CP_n$ with $\nu_a > 0$. 
Let us fix integers $a_1, \dots, a_{\ell} > 0$ such that 
$\sum_{k=1}^{\ell}a_k = a$ by the condition that 
\begin{equation*}
\nu_1 = \cdots = \nu_{a_1} > \nu_{a_1 + 1} = \cdots = \nu_{a_1 + a_2} 
     > \nu_{a_1 + a_2 + 1} = \cdots  
\end{equation*}
We denote by $\nu_{[k]}$ the constant value $\nu_j$ for 
$a_1 + \cdots + a_{k-1} < j \le  a_1 + \cdots + a_{k}$.
For such $j$, $\mu^{(1)}_j$ is also constant and we denote
it by $\mu^{(1)}_{[k]}$.
\par
Take $z = (x,v) \in \CO_{\Bmu}$, and assume that $x = y\th(y)\iv$ 
with $y \in A\uni$ and that $v \in M_n$.  Then 
$(y,v) \in A\uni \times M_n$ is of type $\Bmu$, and by [AH], one 
can find a Jordan basis 
$\{v_{i,j} \mid 1 \le i \le a, 1 \le j \le \nu_i \}$ of $y-1$ in $M_n$
such that $(y-1)v_{i,j} = v_{i, j-1}$ for $j > 1$ and 
$(y-1)v_{i,j} = 0$ for $j = 1$,  
and that  
$v = \sum_{i=1}^{\ell} v_{p_i, \mu^{(1)}_{[i]}}$, where 
$p_i = a_1 + \cdots a_{i-1}+1$. 
(Note that in [AH], the normal form 
for $\CO_{\Bmu}$ is given by $(y, v')$ with  
$v' = \sum_{i=1}^a v_{i, \mu^{(1)}_i}$.  This element is $A$-conjugate 
to our $(y,v)$.) 
Let $\{v'_{i,j}\}$ be a Jordan basis of $y' = \th(y)\iv$ in $M_n'$
($M_n'$ is the subspace of $V$ spanned by $f_1, \dots, f_n$)
such that $(y'-1)v'_{i,j} = v'_{i, j+1}$ for $j < \nu_i$ and 
$(y'-1)v'_{i,j} = 0$ for $j = \nu_i$, 
and that $\lp v_{i,j}, v'_{i',j'}\rp = 0$ unless
$i = i', j = j'$. 
\par
Let $P$ be the stabilizer of the partial flag 
$\lp v_{1,1}\rp \subset \lp v_{1,1}\rp^{\perp}$ in $G$.  
Then $P$ is an $\th$-stable parabolic subgroup of $G$ as in 
6.1, and $H/P^{\th}$ is identified with the projective space 
$\BP(V)$.   
We define $w_i = v_{q_i,1}$ or  $w_i = v'_{p_i, \nu_{[i]}}$, where 
$q_i = a_1 + \cdots + a_i$.   
Then $w_i$ is an $x$-stable vector in $V$.  
Let $P_i$ be the stabilizer of the flag 
$\lp w_i \rp \subset \lp w_i \rp^{\perp}$ in $G$.  Then $P_i$ 
is $\th$-stable, and $P_i^{\th}$ is the stabilizer of the line 
$\lp w_i \rp$ in $H$.    
We show the following lemma.

\begin{lem}  %%%  Lemma 7.3  
Under the notation as above, 
\begin{align*}
\dim Z_H(z) \cap P_i  &= \begin{cases}
  \dim Z_H(z) - 2q_i + 2 &\quad\text{ if } w_i = v_{q_i,1}, \\
  \dim Z_H(z) -2q_i + 1 
          &\quad\text{ if } w_i = v'_{p_i, \nu_{[i]}}.
                           \end{cases}
\end{align*}
\end{lem} 

\begin{proof}
Put $Q = \{ g \in G \mid gv = v, g\lp v\rp^{\perp} = \lp v \rp^{\perp}, 
          g|_{V/\lp v\rp^{\perp}} = \id\}$.  Then $Q$ is 
a $\th$-stable subgroup of $G$, and we have 
$Z_H(z) \cap P_i = Z_{P^{\th}_i \cap Q^{\th}}(x) = 
 (Z_{P_i \cap Q}(x))^{\th}$. 
Note that the structure of $Z_G(x)$ and $Z_H(x)$ are described 
as follows (cf. the proof of Proposition 2.3.6 in [BKS]); 
$Z_G(x) = C \ltimes R$, where 
$C$ is a subgroup of $G$ generated by $g$ such that 
\begin{align*}
gv_{i,j} &= \sum_{\substack{ k \\ \nu_k = \nu_i}}
                 (a_kv_{k,j} + b_kv'_{k,\nu_i - j + 1}), \\
gv'_{i,j} &= \sum_{\substack{ k \\ \nu_k = \nu_i}}
                 (a'_kv_{k,\nu_i - j +1} + b'_kv'_{k,j}),
\end{align*}
and $R$ is the unipotent radical of $Z_G(x)$.   Then 
$C \simeq \prod_{k=1}^{\ell}GL_{2a_k}$.
$C$ and $R$ are $\th$-stable, and so 
$Z_H(x) = C^{\th}\ltimes R^{\th}$.  Hence
$C^{\th} \simeq \prod_{k=1}^{\ell}Sp_{2a_k}$.
Since 
\begin{equation*}
\tag{7.3.1}
Z_{P^{\th}_i \cap Q^{\th}}(x) = 
(C \cap P_i \cap Q)^{\th} \ltimes (R \cap P_i \cap Q)^{\th},
\end{equation*}
we compute $\dim (C\cap P_i \cap Q)^{\th}$  and 
$\dim (R \cap P_i \cap Q)^{\th}$
separately.  
For $k = 1, \dots, \ell$, 
let $V_k$ be the symplectic space 
with $\dim V_k = 2a_k$.  We fix a symplectic basis 
$e^{(k)}_1, \dots, e^{(k)}_{m_k}, f^{(k)}_1, \dots, f^{(k)}_{m_k}$ of $V_k$.
Let $C_k, D_k$ be the subgroups of $GL(V_k) \simeq GL_{2a_k}$ defined by 
\begin{align*}
C_k &= \{ g \in GL(V_k) \mid gv_0 = v_0, 
            g\lp v_0\rp^{\perp} = \lp v_0\rp^{\perp},
            g|_{V_k/\lp v_0 \rp^{\perp}} = \id\},  \\
D_k &= \{ g \in C_k \mid g\lp w_0 \rp = \lp w_0\rp, 
                 g\lp w_0\rp^{\perp} = \lp w_0\rp^{\perp}   \}
\end{align*}
where $v_0 = e_1^{(k)}$, and 
$w_0 = e_{a_k}^{(k)}$ or $w_0 = f_1^{(k)}$.  
Note that $D_k = C_k$ if $v_0 = w_0$.
We denote by $\th$ the involution on $GL_{2a_k}$ defined in a simlar
way as the case of $G$. Then $C_k, D_k$ are $\th$-stable, and we have
\begin{equation*}
\tag{7.3.2}
(C \cap P_i \cap Q)^{\th} \simeq D^{\th}_i \times \prod_{k \ne i}C^{\th}_k
\end{equation*}
Since $C^{\th}_k \simeq (\{ 1\} \times Sp_{2a_k-2})\ltimes U_1$, 
where $U_1$ is the unipotent radical of a parabolic subgroup of 
$Sp_{2a_k}$ whose Levi subgroup is isomorphic to $GL_1 \times Sp_{2a_k-2}$.
It follows that 
\begin{equation*}
\tag{7.3.3}
\dim C^{\th}_k = 2a_k^2 - a_k.
\end{equation*}
Next we show that 
\begin{equation*}
\tag{7.3.4}
\dim D^{\th}_k = \begin{cases}
   2a_k^2 - 3a_k + 2 &\text{ if $w_0 = e_{m_k}^{(k)}$ and $a_k \ge 2$}, \\
   2a_k^2 - 3d_k + 1 &\text{ if } w_0 = f_1^{(k)}.
            \end{cases}
\end{equation*}
\par
First assume that $w_0 = e_{a_k}^{(k)}$ and $a_k \ge 2$.  
Then $v_0 \in \lp w_0\rp^{\perp}$.
We define a new basis $h_1, \dots, h_{2a_k}$ of $V_k$ as follows;
\begin{equation*}
h_j = \begin{cases} 
            v_0   &\quad j = 1, \\
            w_0   &\quad j = 2, \\
            e_{j-1}^{(k)} &\quad 3 \le j \le a_k, \\  
            f_{j+1-a_k}^{(k)}
                        &\quad a_k+1 \le j \le 2a_k - 1,  \\
            f_1^{(k)}  &\quad j = 2a_k.  
      \end{cases}
\end{equation*}
 Then we have
\begin{align*} 
\lp v_0\rp^{\perp} &= \lp h_1, \dots, h_{2a_k-1}\rp\, , \\
\lp w_0\rp^{\perp} &= \lp h_1, \dots, h_{2a_k-2}, h_{2a_k} \rp\, , \\
\lp w_0\rp^{\perp} \cap \lp v_0\rp^{\perp}
                   &= \lp h_1, \dots, h_{2a_k-2}\rp\, ,  
\end{align*}
and the condition for $g \in D_k$ is given by 
\begin{align*}
g h_1  &=  h_1, \\
g \lp h_2 \rp  &= \lp h_2 \rp\, , \\  
g\lp h_1 \dots h_{2a_k-2}\rp &= \lp h_1, \dots h_{2a_k-2}\rp, \\
g\lp h_1, \dots, h_{2a_k-2}, h_{2a_k}\rp &= 
           \lp h_1, \dots, h_{2a_k-2}, h_{2a_k}\rp, \\
gh_{2a_k} &\in h_{2a_k} + \lp h_1, \dots, h_{2a_k-1}\rp.
\end{align*}
We have
   $D_k^{\th} \simeq (GL_1 \times Sp_{2a_k-4})\ltimes U_2$, 
where $U_2$ is the unipotent radical of a parabolic subgroup
of $Sp_{2a_k}$ whose Levi subgroup is isomorphic to 
$GL_2 \times Sp_{2a_k-4}$.  
The first formula in (7.3.4) follows from this.
\par
Next assume that $w_0 = f_1^{(k)}$. 
Then $v_0 \notin \lp w_0\rp^{\perp}$. 
We define a new basis $h_1, \dots, h_{2a_k}$ as follows;
\begin{equation*}
h_j = \begin{cases}
          v_0  &\quad j = 1, \\
          e_j^{(k)}  &\quad 2 \le j \le a_k, \\
          f_{j + 1- a_k}^{(k)} 
                     &\quad a_k+1 \le j \le 2a_k-1, \\
          w_0  &\quad j = 2a_k
       \end{cases}
\end{equation*}
Then we have
\begin{align*}
\lp v_0\rp^{\perp} &= \lp h_1, \dots, h_{2a_k-1}\rp, \\
\lp w_0\rp^{\perp} &= \lp h_2, \dots, h_{2a_k} \rp, \\
 \lp v_0\rp^{\perp} \cap \lp w_0\rp^{\perp}
                  &= \lp h_2, \dots, h_{2a_k-1}\rp, 
\end{align*}
and the condition for $g \in D_k$ is given by 
\begin{align*}
gh_1 &= h_1, \\
g\lp h_2, \dots, h_{2a_k-1}\rp &= \lp h_2, \dots, h_{2a_k-1}\rp, \\
gh_{2a_k} &= h_{2a_k}.
\end{align*}
Then $D^{\th}_k \simeq Sp_{2a_k-2}$, and the second formula
in (7.3.4) follows. 
\par
Now (7.3.2) implies, by (7.3.3) and (7.3.4), that
\begin{align*}
\tag{7.3.5}
\dim (C \cap P_i \cap Q)^{\th} &= 
     \sum_{k \ne i} 
            (2a_k^2 - a_k) + (2a_i^2 -3a_i + 2 - \ve) \\
       &= \dim C^{\th} -2a - 2a_i +2 -\ve, 
\end{align*}
where $\ve = 0$ (resp. $\ve = 1$) if $w_i = v_{q_i,1}$
(resp. $w_i = v'_{p_i, \nu_{[i]}}$). 
Note that this formula holds even if $a_i = 1$ and $\ve = 1$.
\par
Next we show that 
\begin{equation*}
\tag{7.3.6}
\dim (R \cap P_i \cap Q)^{\th} = 
       \dim R^{\th} - 2|\mu^{(1)}| + 2a - 2\sum_{k=1}^{i-1}a_k.
\end{equation*}
For each $i$, 
choose a pair $(i',j')$ such that $i' > i$ or $i' \le i, j' <\nu_i$. 
Then the assignment $v_{i,j} \mapsto v_{i,j} + c_{i,i',j'}v_{i',j'}$
$(c_{i,i'j'} \in \Bk)$ gives rise to a unique element in $R$, 
 which is denoted by $g(c_{i,i',j'})$.  Then 
$\prod_{i,i',j'}g(c_{i,i',j'}) \lra (c_{i,i',j'})_{i,i',j'}$ 
gives an isomorphism
between $R$ and an affine space  $\BA^c$, where $c = \dim R$
(the product is taken under a suitable order).
Then for $g \in R$, the condition $gv = v$ is given by 
a system of linear equations with respect to the variables 
$(c_{i,i',j'})$, one equation for each $v_{i, j}$ such that
$j < \mu^{(1)}_i$, and for each $v'_{i,j}$ such that   
$\nu_i - j < \mu^{(1)}_i$.
It follows that the number of such 
equations is equal to $\sum_{k=1}^{\ell}2a_k(\mu^{(1)}_{[k]}-1)$, 
and they are linearly independent.  
On the other hand, for $g \in R$, the condition 
$g\lp v\rp^{\perp} = \lp v \rp^{\perp}$   
is given by a system of linear equations, one equation for each
$v_{i,j}$ such that $\nu_i - j < \mu^{(1)}_i$, and for each 
$v'_{i,j}$ such that $j < \mu^{(1)}_i$, 
together with $\ell-1$ linear equations arising from 
the $\ell-1$ dimensional space 
$\lp v'_{p_i, \mu^{(1)}_{[i]}}\mid 1 \le i \le \ell\, \rp  
   \cap \lp v \rp^{\perp}$. 
(The condition $g|_{V/\lp v\rp^{\perp}} = \id$ is then 
automatically satisfied since $g \in R$.)
Hence the number of equations is given by 
$\sum_{k=1}^{\ell}2a_k(\mu^{(1)}_{[k]}-1) + (\ell-1)$.
By a similar argument, the condition for $g \in R$ to be 
$g\lp w_i \rp = \lp w_i \rp$ 
(resp. $g\lp w_i \rp^{\perp} = \lp w_i\rp^{\perp}$)  
is both given by a system of linear equations whose number is
equal to $2a_1 + \cdots + 2a_{i-1}$.  Since all of those linear 
equations are linearly independent, we see that 

\begin{equation*}
\dim (R \cap P_i \cap Q) = \dim R - 4|\mu^{(1)}| + 4a -(\ell-1) 
      - 4\sum_{k=1}^{i-1}a_k.
\end{equation*}
If we pass to $(R \cap P_i \cap Q)^{\th}$,  
the linear equation corresponding to $v_{i,j}$ with $j < \mu^{(1)}_i$
and the equation corresponding to $v'_{i,j}$ with $j < \mu^{(1)}_i$
gives one equation, and similarly for $v_{i,j}$ and $v'_{i,j}$ with 
$\nu_i-j < \mu_i^{(1)}$.  This phenomenon also occurs for 
the linear equations with re spec to $P_i$, 
and we obtain the half of those linear equations.
However, the $(\ell - 1)$ linear equations does not give 
a restriction on $R^{\th}$.  Hence we obtain (7.3.6).
\par
The lemma now follows from (7.3.5) and (7.3.6), if we notice 
that $\dim Z_H(z) = \dim Z_H(x) - 2|\mu^{(1)}|$ (see the proof of 
Lemma 2.3). 
\end{proof}

\para{7.4.}
Take $\Bmu \in \CP_{n,2}$ with $\nu = \mu^{(1)} + \mu^{(2)}$.
We follow the notation in 7.2.  
Let $\Bmu'= ({\mu'}^{(1)}, {\mu'}^{(2)}) \in \CP_{n-1,2}$
be such that 
\par
(i) the Young diagram ${\mu'}^{(1)}$ is obtained 
from the Young diagram $\mu^{(1)}$ by removing one node, or
\par 
(ii) ${\mu'}^{(2)}$ is obtained from $\mu^{(2)}$ by removing one node.
\par\noindent
We assume that the removing node is contained in $q_i$-th row in each 
case, i.e., ${\mu'}^{(1)}_{q_i} = \mu^{(1)}_{q_i}-1$ or
${\mu'}^{(2)}_{q_i} = \mu^{(2)}_{q_i} -1$.   
Let $d_{\Bmu} = \nu_H - \dim \CO_{\Bmu}/2$ and 
$d_{\Bmu'} = \nu_{H'} - \dim \CO_{\Bmu'}/2$, where $H' = Sp_{2n-2}$.
We note that 
\begin{equation*}
\tag{7.4.1}
d_{\Bmu} - d_{\Bmu'} = \begin{cases}
             2q_i -2  &\quad\text{ case (i)}, \\
             2q_i -1  &\quad\text{ case (ii)}.
                        \end{cases}
\end{equation*}
In fact, by Lemma 2.3, we have $d_{\Bmu} = n + 2n(\Bmu) - |\mu^{(1)}|$, 
and a similar formula holds for $d_{\Bmu'}$.
Since $|\mu^{(1)}| - |{\mu'}^{(1)}| = 1$ or 0 according to the case 
(i) or (ii), and $n(\Bmu) - n(\Bmu') = q_i - 1$, we obtain (7.4.1).
\par  
Take $z = (x,v) \in \CO_{\Bmu}$ as in 7.2, and put    
$w_i = v_{q_i,1}$
in case (i), and $w_i = v'_{p_i,\nu_{[i]}}$ in case (ii).
Let $x'$ be the linear transformation on the symplectic space  
$\ol V_i = \lp w_i\rp^{\perp}/\lp w_i\rp$.  We have 
$v \in \lp w_i \rp^{\perp}$ (note that $\mu^{(2)}_{q_i} \ne 0$ in case (ii)),
and let $v'$ be the image of $v$ on $\ol V_i$.
Then one can check that $z' = (x',v') \in \CO_{\Bmu'}$.   
\par
We are now ready to prove the theorem.  We consider the variety 
$f\iv(z)$ appeared in 6.7 instead of considering the variety 
$Y_{z,z'}$ in 6.1.
Let $\CU_z = \{gP^{\th} \in H/P^{\th} 
           \mid g\iv z \in \pi_P\iv(\CO_{\Bmu'})\}$ 
which is isomorphic to $f\iv(z)$ for $\CO' = \CO_{\Bmu'}$.  
Let $P_i$ be the stabilizer 
of the flag $\lp w_i\rp \subset \lp w_i\rp^{\perp}$ in $G$.
Then $g_iP^{\th} \in \CU_z$ for $P_i = g_iPg_i\iv$.  $Z_H(z)$ acts on 
$\CU_z$ from the left, and we consider the $Z_H(z)$ orbit $Y_i$ of
$g_iP^{\th}$ in $\CU_z$.   
By Lemma 7.3 and (7.4.1), we have 
\begin{equation*}
\tag{7.4.2}
\dim Y_i = \dim Z_H(z) - \dim (Z_H(z) \cap P_i) = d_{\Bmu} - d_{\Bmu'}.
\end{equation*}
By 6.7 (an equivalent form of Theorem 6.2), 
all the irreducible components in $\CU_z$ have dimension 
$\le d_{\Bmu} - d_{\Bmu'}$, and the number of irreducible components 
of $\CU_z$ of dimension $d_{\Bmu} - d_{\Bmu'}$ coincides with the 
multiplicity $\lp \r_{\Bmu}, \r_{\Bmu'}\rp$ . 
By (7.4.2), the closure of $Y_i$ gives an irreducible component of 
$\CU_z$ of the required dimension.  It follows that 
$\lp \r_{\Bmu}, \r_{\Bmu'}\rp \ge 1$.  
Since it is known that the restriction of
an irreducible representation of $W_n$ to $W_{n-1}$ is multiplicity free, 
we see that $\lp \r_{\Bmu},\r_{\Bmu'}\rp = 1$.   
By induction on the rank of $G$, we 
may assume that $\r_{\Bmu'} \simeq \wt V_{\Bmu'}$.  Then our result shows 
that the restriction of $\r_{\Bmu}$ on $W_{n-1}$ coincides with that of 
$\wt V_{\Bmu}$.  Since the irreducible representation of $W_n$ is determined
completely by its restriction on $W_{n-1}$ if $n \ge 3$ or if 
its degree $\ge 2$,   
we see that $\r_{\Bmu} \simeq \wt V_{\Bmu}$ in this case.
Hence we may assume that $n = 2$, and we have only to distinguish $\r_{\Bmu}$
for $\Bmu = ((2),-), ((1^2),-)$ and for $\Bmu = (-,(2)), (-,(1^2))$.
But by Lemma 5.2, we already know that $\r_{\Bmu} = \wt V_{\Bmu}$ for 
$\Bmu = ((1^2), -), (-, (1^2))$, and so the remaining cases are determined.
This completes the proof of Theorem 7.1. 

\para{7.5.}
We give some examples of the Springer correspondence.
Take $(x, v) \in \CX\uni$ such that $x = y\th(y)\iv$, where
$y \in A$ is regular unipotent, and that $v, yv, y^2v, \cdots$ 
spans $M_n$.  Then $(x,v) \in \CO_{\Bmu}$ with 
$\Bmu = ((n);\emptyset)$, the open dense orbit in $\CX\uni$.
The corresponding irreducible representation $\wt V_{\Bmu}$ is 
the identity representation of $W_n$.    
Take $(x,v) \in \CX\uni$ such that $x$ is the identity element 
in $GL(V)$ and $v = 0$.  Then the orbit $\CO_{\Bmu}$ containing $(x,v)$
is given by $\Bmu = (\emptyset; (1^n))$.  $\wt V_{\Bmu}$ is the 
sign representation of $W_n$. 
More generally, let $x = y\th(y)\iv$, where the Jordan type of 
$y \in A$ is given by $\nu = (\nu_1, \nu_2, \dots) \in \CP_n$.
Take $(x,0) \in \CX\uni$.  Then $(x,0) \in \CO_{\Bmu}$ with
$\Bmu = (\emptyset; \nu)$. $\wt V_{\Bmu}$ is the extension 
of the irreducible $S_n$-module $V_{\nu}$, where 
$(\BZ/2\BZ)^n$  acts trivially on it.  
On the other hand, take the same $x$, and let $v \in M_n$ 
such that $\dim\lp v, yv, y^2v,\dots, \rp = \nu_1$.
Then $(x,v) \in \CO_{\Bmu}$ with $\Bmu = (\nu; \emptyset)$.
$\wt V_{\Bmu}$ is the extension of $V_{\nu}$, where 
each factor $\BZ/2\BZ$ of $(\BZ/2\BZ)^n$ acts non-trivially on it.

\remark{7.6.}
By making use of Theorem 5.14, one can show that the correspondence
$\la \mapsto \la^{\bullet}$ for $\la \in \CP_n$  
in the formula (1.14.2) in Proposition 1.14 is identical, i.e.,
for the map $\pi'_1 : \wt G^{\io\th}\uni \to G^{\io\th}\uni$ 
(the map $\pi_1$ in 1.10. We changed the notation to distinguish this with 
$\pi_1 : \wt\CX\uni \to \CX\uni$), we have 
\begin{equation*}
\tag{7.5.1}
(\pi'_1)_*\Ql[\dim G^{\io\th}\uni] \simeq 
        H^{\bullet}(\BP_1^n)\otimes \bigoplus_{\la \in \CP_n}
            V_{\la}\otimes \IC(\ol\CO_{\la},\Ql)[\dim \CO_{\la}]. 
\end{equation*} 
In fact, 
since $G^{\io\th} \simeq G^{\io\th}\times \{ 0\}, 
\wt G^{\io\th} \simeq \wt\CX_m^+$ 
for $m = 0$ under the notation of Section 4, $(\pi_1')_*\Ql$ 
coincides with the restriction of 
$(\pi_m)_*\Ql = (\ol\pi_m)_*\Ql$ (for $m = 0$)
to $\CX\uni$ up to shift.  Here by Proposition 4.8 (or rather 
by (4.9.1)), we have 
\begin{equation*}
\tag{7.5.2}
(\ol\pi_0)_*\Ql \simeq \bigoplus_{\la \in \CP_n}
       H^{\bullet}(\BP_1^n)\otimes V_{\la}\otimes \IC(\CX_0, \CL_{\la}).
\end{equation*}  
By Theorem 5.4 (ii), together with Theorem 7.1, we see that
\begin{equation*}
\tag{7.5.3}
\IC(\CX_0, \CL_{\la})|_{\CX\uni} \simeq \IC(\ol\CO_{\la},\Ql)
\end{equation*}
up to shift.  (Here we identify $\CO_{\la} \subset G^{\io\th}\uni$ 
with $\CO_{(-;\la)} \subset \CX\uni$ 
under the closed embedding 
$G^{\io\th}\uni \simeq G^{\io\th}\uni \times \{ 0\} \hra \CX\uni$.) 
(7.5.1) follows from (7.5.2) and (7.5.3).   
%%%%%%%%%%%%%%%%%%%%%%%%%%%%%%%%%%%

\par

\newpage
\par\vspace{1cm}
\noindent
T. Shoji \\
Graduate School of Mathematics, Nagoya University \\ 
Chikusa-ku, Nagoya 464-8602, Japan  \\
E-mail: \verb|shoji@math.nagoya-u.ac.jp|
\par\bigskip\bigskip\noindent
K. Sorlin \\
L.A.M.F.A, CNRS UMR 7352, 
Universit\'e de Picardie-Jules Verne \\
33 rue Saint Leu, F-80039, Amiens, Cedex 1, France \\
E-mail : \verb|karine.sorlin@u-picardie.fr|
 \end{document}